\documentclass[11pt]{amsart}
\usepackage{amsfonts,amsmath,latexsym,amssymb,verbatim,amsbsy,times, esint}
\usepackage{amsthm}
\usepackage{pstricks}
\usepackage{mathtools}
\mathtoolsset{showonlyrefs}
\usepackage{graphicx}
\usepackage{caption}
\usepackage{mathrsfs}
\usepackage{placeins}
\usepackage{bm}

\usepackage{tikz-cd}

\usepackage{tikz}
\usetikzlibrary{patterns, arrows.meta, decorations.pathreplacing}
\usepackage{mathrsfs}

\usepackage[top=0.9 in, bottom=0.9in, left=0.9in, right=0.9in]{geometry}

\usepackage[colorlinks=true, pdfstartview=FitV, linkcolor=blue,citecolor=green, urlcolor=blue]{hyperref}

\usepackage{enumitem}

\theoremstyle{plain}
\newtheorem{THEOREM}{Theorem}[section]
\newtheorem{COROL}[THEOREM]{Corollary}
\newtheorem{LEMMA}[THEOREM]{Lemma}
\newtheorem{PROP}[THEOREM]{Proposition}

\theoremstyle{definition}
\newtheorem{DEF}[THEOREM]{Definition}

\theoremstyle{remark}
\newtheorem{REMARK}[THEOREM]{Remark}

\newcommand{\N}{\ensuremath{\mathbb{N}}}   %%% naturals
\newcommand{\Z}{\ensuremath{\mathbb{Z}}}   %%% integers
\newcommand{\R}{\ensuremath{\mathbb{R}}}   %%% reals
\def \RR {\mathbb{R}}

\def \e {\varepsilon}

\def \n {\nabla}
\def \s {\sigma}

\def \bpi {\boldsymbol{\pi}}

\def \id {{\bf id}}

\def \bu {{\bf u}}

\def \cA {\mathcal{A}}

\def \cD {\mathcal{D}}

\def\cF {\mathcal {F}}

\def \cJ {\mathcal{J}}

\def \cM {\mathcal{M}}
\def \cN {\mathcal{N}}

\def \cP {\mathcal{P}}
\def \cQ {\mathcal{Q}}
\def \cR {\mathcal{R}}

\def\cprime{$'$}

\def \loc {\mathrm{loc}}

\def \p {\partial}

\newcommand{\sgn}{\mathrm{sgn}}

\DeclareMathOperator{\diam}{diam} %
\DeclareMathOperator{\supp}{supp} %

\def \cP {\mathcal{P}}

\def \Lip {\mathrm{Lip}}

\def \dd   {\mathrm{d}}

\def \dh  {\mathrm{d}h}

\def \dt  {\mathrm{d}t}
\def \dtau  {\mathrm{d}\tau}

\def \dv  {\mathrm{d}v}
\def \dx  {\mathrm{d}x}
\def \dy  {\mathrm{d}y}
\def \dz  {\mathrm{d}z}

\def \dm  {\mathrm{d}m}
\def \dr  {\mathrm{d}r}

\def \dr  {\mathrm{d}r}

\usetikzlibrary{cd}
\usetikzlibrary{snakes}

\title[Unidirectional Entropic Solutions
of the Pressureless Euler Alignment System]{Unidirectional Entropic Solutions \\
of the Pressureless Euler Alignment System}

\author{Joshua O. Adeleke}
\address{Department of Applied Mathematics,
	Illinois Institute of Technology}
\email{jadeleke@hawk.illinoistech.edu}

\author{Trevor M. Leslie}
\address{Department of Applied Mathematics,
	Illinois Institute of Technology}
\email{tleslie@illinoistech.edu}

\date{}

\thanks{JA and TL are supported in part by NSF grant DMS-2408585 (PI: Leslie). Both authors thank Changhui Tan, Eitan Tadmor, and Roman Shvydkoy for helpful comments on an earlier draft of the manuscript.}

\makeatletter
\@namedef{subjclassname@2020}{\textup{2020} Mathematics Subject Classification}
\makeatother

\begin{document}
	
	{\vspace*{- 5 mm}}
	\maketitle
	
	\vspace{-5 mm}
	\begin{abstract}
		We develop a global wellposedness theory for weak solutions of the pressureless Euler Alignment system with measure-valued density $\rho$, bounded and unidirectional velocity $\mathbf{u} = (u,0, \ldots, 0)$, and a communication protocol $\phi$ that may be bounded or weakly singular.  To the best of the authors' knowledge, this is the first such theory that admits shocks outside  one space dimension.  We recast the system as a family of nonlocally coupled scalar balance laws---one for each horizontal slice of $\mathbb{R}^d$---and establish existence, uniqueness, and stability of entropy solutions of the reformulated system, before translating back to the level of $\rho$ and $u$.  This approach grants us access to certain key one-dimensional tools that facilitate our analysis in the direction of the flow.  The analysis in the directions transverse to the flow, however, still presents substantial challenges even in the development of a suitable framework: The horizontal slices are coupled and therefore cannot be treated independently, and two solutions need not distribute their mass over horizontal slices in the same way, so that there is no canonical way to compare them slicewise.  Accordingly, our analysis is centered around the nonlocal alignment force, and our framework leverages optimal couplings between the projections onto $\mathbb{R}^{d-1}$ of the density profiles under consideration.  We construct our solutions as limits of atomic density and momentum profiles whose atoms follow sticky particle Cucker--Smale dynamics, and under additional regularity assumptions on the initial data (and the protocol $\phi$), we obtain quantitative rates of convergence.	Finally, our investigation of the long-time behavior of unidirectional solutions highlights the role of lateral communication in driving the system toward a limiting density profile: We prove that flocking occurs, at a rate independent of the number of agents, under assumptions on $\phi$ that bound its size from below only outside a cylindrical neighborhood of the axis parallel to the flow. 	   
	\end{abstract}
	
%	\vspace{-2 mm}
	\begin{small} 
		\tableofcontents
	\end{small}

%	\newpage
	\section{Introduction}
	
	We are concerned with the following Euler Alignment system of collective behavior:
	\begin{equation}
		\label{e:EA}
		\left\{ 
		\begin{split}
			\partial_t \rho + \nabla \cdot (\rho \mathbf{u}) & = 0, \qquad \bm{x}\in \mathbb{R}^d, \,t>0, \\
			\partial_t (\rho \bu) + \nabla \cdot (\rho \bu \otimes \bu) 
			& = \int_{\R^d} \phi(\bm{x} - \bm{\xi}) (\mathbf{u}(\bm{\xi}) - \mathbf{u}(\bm{x})) \rho(\bm{x})\rho(\bm{\xi}) \,\dd \bm{\xi}.
		\end{split}
		\right.
	\end{equation}
	Here $\rho = \rho(\bm{x},t)$ (a probability measure on $\mathbb{R}^d$, for each $t$) and $\bu = \bu(\bm{x},t)$ denote the macroscopic density and velocity of a continuum of agents, which interact through the \textit{communication protocol} $\phi:\mathbb{R}^d\to [0,\infty)$.  The \textit{alignment force} on the right side of \eqref{e:EA}$_2$ works to bring the velocities toward agreement and is the basis for studying \eqref{e:EA} as a \textit{flocking model}.  Global-in-time wellposedness is therefore a matter of pressing interest, since one of the main motivations for studying such models is the understanding of their long-time behavior. 
	
	For a large class of protocols $\phi\in L^1_{\loc}(\mathbb{R}^d)$,  \eqref{e:EA} constitutes a system of nonlocal hyperbolic balance laws.  As is typical in hyperbolic systems, solutions of \eqref{e:EA} associated to smooth initial data can lose regularity in finite time.   This leads to the question of how solutions can be continued after blowup, which is the primary focus of this paper.  A complete resolution of the issue in higher dimensions remains out of reach---even without the nonlocal term, it is a long-standing open problem---but we are nevertheless able to develop a comprehensive theory of existence, uniqueness, stability, and long-time behavior tailored to \textit{unidirectional} dynamics (where all velocities point in the same direction) and communication protocols that may be bounded or weakly singular. To the best of our knowledge, this work provides the first wellposedness theory for \eqref{e:EA} that admits shocks in any geometry outside of the pure 1D setting.  Substituting   $\bu(\bm{x},t) = (u(\bm{x},t),0,\ldots,0)$ into \eqref{e:EA} and writing $\bm{x} = (x,y)\in \mathbb{R}\times \mathbb{R}^{d-1}$ yields the relevant system (here $\partial_1$ denotes differentiation with respect to the first argument):
	\begin{equation}
		\label{e:EAuni}
		\left\{ 
		\begin{split}
			\partial_t \rho + \partial_1 (\rho u) & = 0,\qquad \quad (x,y)\in \mathbb{R}\times \mathbb{R}^{d-1},\,t>0, \\
			\partial_t (\rho u) + \partial_1 (\rho u^2) & = \int_{\mathbb{R}\times \mathbb{R}^{d-1}} \phi(x-\xi, y - \eta) (u(\xi, \eta) - u(x,y)) \rho(x,y) \rho(\xi, \eta) \,\dd\xi\,\dd\eta.
		\end{split}
		\right. 
	\end{equation}
	We state our full set of assumptions on $\phi$ in Section \ref{sss:phiassumptions}.   For now, we note that our $\phi\ge 0$ will always be even,  $\phi(x,y) = \phi(-x,-y)$, but we do \textit{not} require it to be radial.  We allow for weak singularities at zero: 
	If $\phi$ is Lipschitz away from the origin and there exists a nonincreasing function $\overline{\phi}\in L^1_{\loc}([0,\infty))$ such that $\phi(\bm{x})\le \overline{\phi}(|\bm{x}|)$, then $\phi$ will satisfy the hypotheses of most of our main theorems.  
	
	System \eqref{e:EAuni}  constrains trajectories to move along horizontal slices $\mathbb{R}\times \{y\}$, the mass of each of which is  conserved.  This is a boon for the analysis and allows us to conduct some of our work slice-by-slice.  As will be justified in Section \ref{s:Recovery}, an additional special structure of~\eqref{e:EAuni} allows us to capture its dynamics in a possibly infinite family of coupled scalar balance laws, one for $\nu$-a.e. $y\in \mathbb{R}^{d-1}$:
	\begin{equation} 
		\label{e:ScalarBalance}
		\left\{ 
		\begin{split}
			& \partial_t M(x,y,t) + \partial_x(A(M(x,y,t),y)) = B[\nu, M] \,\partial_x M(x,y,t), \qquad x\in \mathbb{R}, \;\;t>0,\\
			& B[\nu, M](x,y,t) := (\Phi*(\nu \partial_1 M))(x,y,t) = \int_{\mathbb{R}^{d-1}} \int_\mathbb{R} \Phi(x-\xi, y-\eta) \,\partial_1 M(\dd\xi, \eta,t)\,\nu(\dd\eta).
		\end{split} 
		\right.
	\end{equation} 
	Here $\Phi(x,y) = \int_0^x \phi(\xi,y)\,\dd\xi$ is an antiderivative of $\phi$ along horizontal slices.  The flux $A$ depends nonlocally on the initial velocity and initial density profile (the precise relationship is \eqref{e:defA} below). Finally, the time-independent probability measure $\nu$ (which we refer to as the `transverse' measure) and the function $M$ are defined via disintegration: 
	\begin{equation}
		\rho(\dx, \dy, t) = \rho_y(\dx, t)\,\nu(\dy),
		\qquad M(x,y,t) = 
		\rho_y((-\infty, x],t).
	\end{equation}
	
	Even though \eqref{e:ScalarBalance} is substantially simpler than the more general system \eqref{e:EA}, the nonlocal coupling between slices (involving a possibly unbounded kernel $\phi$) makes \eqref{e:ScalarBalance} a genuinely multidimensional problem, and one that is not covered by the classical theory, for several structural reasons detailed later on.  We also stress that our focus is on the \textit{interactions between} different slices rather than on the dynamics along each slice---though these two issues cannot, of course, be completely separated.  
	
	\subsection{Main results and contributions} Our aim in this paper is to develop a theory of wellposedness and long-time behavior for what we refer to as `entropic solutions' of the unidirectional Euler Alignment system \eqref{e:EAuni}. We describe our results here on their own terms; in Sections \ref{ss:Background} and \ref{ss:Weakformulations}, we will situate them in the literature and  provide the precise definitions, entropy conditions, and weak formulations needed for our analysis.
	
	\subsubsection{Uniqueness and Stability of Entropy Solutions} 
	We develop a stability theory for entropy solutions of \eqref{e:ScalarBalance} with respect to perturbations of the initial data $M^0$, flux $A$, and transverse measure $\nu$.  The last of these requires special care: Once two solutions $M$ and $\widetilde{M}$ carry different transverse measures $\nu \ne \widetilde{\nu}$, it is unclear how to meaningfully compare $M$ and $\widetilde{M}$ through an $L^1$-based quantity.  We resolve this by integrating over an optimal coupling  between $\nu$ and $\widetilde{\nu}$ in our measure of distance between $M$ and $\widetilde{M}$.  Under this framework, Theorem \ref{t:Mstab} establishes two complementary  estimates. The first grows linearly in time and requires Lipschitz-type control on the flux difference; the second needs only $L^\infty$ or H\"older control on the flux difference but grows exponentially in time and requires $\phi$ to be bounded.  
	
	More precisely, 
	let $M$ be an entropy solution of \eqref{e:ScalarBalance} equipped with flux $A$ and transverse measure $\nu$; define $\widetilde{M}$, $\widetilde{A}$, $\widetilde{\nu}$ analogously, and let $\rho_y$, $\widetilde{\rho}_{\widetilde{y}}$ denote slices of the associated densities $\rho$ and $\widetilde{\rho}$. Let $\bpi$ be an optimal coupling of $\nu$ and $\widetilde{\nu}$ for the cost $|y - \widetilde{y}|$, and let $W_1$ denote Wasserstein-1 distance (in any dimension).  Our stability estimates concern the quantity 
	\[
	\cD(t) = \int_{\mathbb{R}^{d-1} \times \mathbb{R}^{d-1}} \|M(\cdot, y,t) - \widetilde{M}(\cdot, \widetilde{y}, t)\|_{L^1(\mathbb{R})}\,\bpi(\dy, \dd\widetilde{y})
	= \int_{\mathbb{R}^{d-1} \times \mathbb{R}^{d-1}} W_1(\rho_y(t), \widetilde{\rho}_{\widetilde{y}}(t))\,\bpi(\dy, \dd\widetilde{y}).
	\]
	This is an $L^1$-based analog of what Peszek and Poyato refer to as the `fibered transport pseudometric' comparing $\rho$ and $\widetilde{\rho}$ (see \cite{PeszekPoyatoFiberedOT} for a systematic theory of fibered quantities of this type, as well as a discussion of earlier developments of similar quantities).  It satisfies $W_1(\rho(t), \widetilde{\rho}(t)) \le \cD(t) + W_1(\nu, \widetilde{\nu})$, so our estimates on $\cD(t)$ also provide control on the Wasserstein-1 distance between $\rho$ and $\widetilde{\rho}$. The appearance of $\bpi$ in the definition of $\cD$ is natural: 
	The contribution of $\|M(\cdot, y)-\widetilde{M}(\cdot, \widetilde{y})\|_{L^1(\mathbb{R})}$ depends on how much mass from slice $y$ is assigned to slice $\widetilde{y}$ when $\nu$ and $\widetilde{\nu}$ are coupled optimally.  Using \textit{some} sort of coupling between $\nu$ and $\widetilde{\nu}$ is arguably the only way to make a fair comparison between $M$ and~$\widetilde{M}$.  
	
	Both of our estimates double the variables $x$ and $t$ in the style of Kru\v zkov \cite{Kruzkov1970}, and (after properly accounting for the coupling $\bpi$) the flux can be handled by adapting existing methods.  The analysis of the nonlocal term, on the other hand, is subtle, and forms the technical heart of the proofs of both statements.  Let us briefly elaborate, letting  $\Delta$ denote our mollification parameter in the $x$-variable. In our first stability estimate (Theorem \ref{t:Mstab}(a)), we send $\Delta\to 0+$ before estimating the nonlocal term; leveraging the antisymmetry of $\Phi$, we show that the resulting contribution of the nonlocal term can be written as the sum of a \textit{nonpositive} term and a quantity depending on $W_1(\nu, \widetilde{\nu})$; this allows us to close our estimates even when $\phi$ is weakly singular. Our final bound depends on the Lipschitz seminorm of the flux difference (cf. Lucier \cite{Lucier1986}). In our second stability estimate (Theorem \ref{t:Mstab}(b)), we keep $\Delta$ fixed until the end of the argument, at which point we optimize---an idea we borrow from Bouchut and Perthame's work \cite{BouchutPerthame1998}.  The quantity that was nonpositive in the previous estimate is no longer sign-definite, and we require an $L^\infty$ bound on $\phi$ to control this term.  However, the resulting stability estimate needs only $L^\infty$ or H\"older control on the flux difference; this weaker demand is what allows us to include (under additional assumptions) rates of convergence of our discretization scheme even in the presence of nontrivial vacuum configurations (see Section \ref{s:Rates}).

	\subsubsection{Existence and Construction of Entropy Solutions via Sticky Particle Dynamics}  
	\label{sss:ExistenceConstruction}
	In Section \ref{s:Existence}, we construct entropy solutions of \eqref{e:ScalarBalance} for any $(M^0, A, \nu)$ satisfying the conditions of Definition \ref{def:SBIC} below.  We proceed in three stages of increasing generality, ultimately exhibiting our entropy solution as a limit of spatially discrete ones. Our final step, where we move from atomic~$\nu$ to a general compactly supported transverse measure, requires the most new work.  
	
	The first step of our construction, described in Theorem \ref{t:SPCSentropy}, is to build the discrete-in-space entropy solutions mentioned above, starting from atomic $\nu$, with $M^0(\cdot, y)$ piecewise constant and $A(\cdot, y)$ piecewise linear for each $y\in \supp\nu$.  Extending the 1D analysis of the second author and Tan in \cite{LeslieTan2023}, we show that this structure propagates in time and that the jump discontinuities follow \textit{sticky particle Cucker--Smale} dynamics,  which are developed in Section \ref{sss:SPCS1}.  This construction is essentially a front-tracking approximation (cf. \cite{Dafermos1972, HoldenRisebrobook, Dafermosbook}) of the solution we construct in the second step, which we obtain in Theorem \ref{t:sliceexist} as a limit as the number of sticky particles along each slice (of a fixed, atomic measure $\nu$) tends to infinity.
	
	In our final step, recorded in Theorem \ref{t:existencefinalstep}, we take the number of slices to infinity.  We fashion certain subsequential limits $\overline{M}(x,y,t)$ and $M_y(x,t)$ from our transverse discretization procedure and show that $\nu$-a.e. $M_y$ is an entropy solution~of  
	\begin{equation} 
		\label{e:SBfrozenintro}
		\partial_t M_y(x,t) + \partial_x(A(M_y(x,t), y)) = B[\nu, \overline{M}](x,y,t)\,\partial_x M_y(x,t), 
		\qquad (x,t)\in \mathbb{R}\times (0,\infty).
	\end{equation} 
	With $\overline{M}$ regarded as given,  \eqref{e:SBfrozenintro} is a far simpler object than even the 1D version of \eqref{e:ScalarBalance}. Indeed, \eqref{e:SBfrozenintro} is a bona fide scalar balance law with no nonlocality and a source that is linear in $\partial_x M_y$; its entropy solutions are easily seen to be unique.  This uniqueness allows us to establish that $\overline{M}(\cdot, y, \cdot) = M_y$ for $\nu$-a.e. $y$ and that $M = \overline{M}$ is an entropy solution of \eqref{e:ScalarBalance}.

	\subsubsection{Entropic Solutions of \eqref{e:EAuni}}
	
	In Section \ref{s:Recovery}, we use the entropy conditions for \eqref{e:ScalarBalance} to uniquely select weak solutions of the initial value problem for \eqref{e:EAuni}, through a process we refer to as the \textit{entropic selection principle}.  Given initial data $(\rho^0, u^0)$ for~\eqref{e:EAuni}, we generate an associated triple $(M^0, A, \nu)$, obtain a unique entropy solution $M$ of~\eqref{e:ScalarBalance}, and use it to produce $\rho$ and $u$ which solve \eqref{e:EAuni}.  Definition \ref{def:entropicselection} spells out the details, and Theorem \ref{t:entropicselection} makes precise the sense in which $(\rho, u)$ satisfies~\eqref{e:EAuni}. Importantly, an atomic initial density will propagate atomically for all time under this procedure (Proposition~\ref{p:entropicdiscrete}), and the atoms will follow sticky particle Cucker--Smale dynamics.  Finally, the density $\rho$ and momentum $\rho u$ of any weak solution arising from this procedure can be attained through a limit of sticky particle dynamics, per Corollary~\ref{c:SPCSconvgen}.  
	
	\subsubsection{Rates of Convergence for the Sticky Particle Approximation}
	
	If $(\rho^0, u^0)$ (or $(M^0, A, \nu)$) is  sufficiently regular,  then the stability estimates of Theorem \ref{t:Mstab} can be used to establish quantifiable rates of convergence for the sticky particle limit discussed just above.  If one works entirely at the level of the scalar balance law, the rates for $M$ fall out almost for free, but the situation for $\rho$ is more delicate. The issue is vacuum, which for generic $(\rho^0,u^0)$ will prevent us from passing enough regularity to $A$ to allow for use of Theorem \ref{t:Mstab}(a).  In Section \ref{s:Rates}, we provide two complementary statements on rates of convergence for the sticky particle limit towards $\rho$, which deal with the issue in different ways.  Theorem \ref{t:SPCSratestrong} essentially assumes away the vacuum; it is valid for certain favorable initial conditions and for bounded or weakly singular communication protocols $\phi$.  On the other hand, Theorem \ref{t:SPCSconvrateweak} treats more generic initial conditions, at the cost of requiring $\phi$ to be bounded, by leveraging Theorem \ref{t:Mstab}(b) and its weaker requirement on the flux.  
	
	\subsubsection{Flocking} 
	\label{sss:flockingresultsintro}
	Section \ref{s:Flocking} concerns long-time behavior for solutions of \eqref{e:EAuni} and those of its discrete counterpart, the unidirectional Cucker--Smale system of ODEs.  Under suitable assumptions on $\phi$ and on the initial configuration, we show that $\rho(t)$ converges to a traveling wave profile: For some compactly supported probability measure $\rho_\infty$, there exist $C,c>0$ such that
	\begin{equation} 
		\label{e:EAflockingintro}
		W_1(\rho(\cdot + \overline{u}t, \cdot, t), \rho_\infty) \le Ce^{-ct},
		\qquad t>0,
	\end{equation}
	where $\overline{u} = \int \rho^0 u^0$.  Convergence of this kind is one of several related long-time behaviors referred to in the literature as `flocking'.  It is well-known that sufficiently slow decay of $\phi$ at infinity is enough to establish flocking for smooth solutions (possibly even in a stronger topology), for \textit{any} initial data.  Our results extend this to weak unidirectional solutions, but they also do something more.  We observe that, under the constraint of unidirectional geometry, the emergence of a flocking state need not rely at all on communication in the direction of motion.  In fact, we may even have $\phi(x,y)\le 1_{(r,\infty)}(|y|)$ for some $r>0$, creating a cylinder $\mathbb{R}\times \overline{B(0,r)}$ where the communication is completely degenerate.  See Figure \ref{fig:degenerate-slice-communication} for a schematic depiction, and Theorems \ref{t:flockingCS} and \ref{t:flockingEA} for precise statements.  The utility of this observation is not limited to systems where communication actually degenerates.  Even if agents have no `blind spot' in their direction of motion, our results suggest that lateral communication should not be discounted as a driving mechanism for flocking dynamics.
	
	\begin{figure}[htbp]
		\resizebox{\linewidth}{!}{
			\begin{minipage}{0.39\linewidth}
				
				\resizebox{\textwidth}{!}{%
					\begin{tikzpicture}[
						>=Latex,
						agent/.style={circle, fill=black, inner sep=1.5pt}, %
						hatch/.style={pattern=north east lines, pattern color=black!50},
						note/.style={align=left, font=\small}
						]

						\begin{scope}[xshift=0cm, scale=0.9, every node/.append style={scale=0.9}]
							
							\node[font=\bfseries] at (0, 4.4) {Support of the Communication Kernel $\phi$};

							\draw[->, thick] (-5.4, 0) -- (4.0, 0) node[right] {$x$};
							\draw[->, thick] (0, -3.8) -- (0, 3.9) node[left] {$y$};

							\fill[hatch] (-3.5, 0.6) rectangle (3.5, 3.5);
							\fill[hatch] (-3.5, -3.5) rectangle (3.5, -0.6);

							\draw[thick] (-3.5, 3.5) -- (3.5, 3.5);
							\draw[thick] (-3.5, 0.6) -- (3.5, 0.6);
							\draw[thick] (-3.5, -0.6) -- (3.5, -0.6);
							\draw[thick] (-3.5, -3.5) -- (3.5, -3.5);

							\node[right, fill=white, inner sep=1.5pt, font=\small] at (0.1, 3.5) {$\mathscr{D}_2$};
							\node[right, fill=white, inner sep=1.5pt, font=\small] at (0.1, 0.6) {$r$};
							\node[right, fill=white, inner sep=1.5pt, font=\small] at (0.1, -0.6) {$-r$};
							\node[right, fill=white, inner sep=1.5pt, font=\small] at (0.1, -3.5) {$-\mathscr{D}_2$};

							\node[fill=white, inner sep=2pt, font=\small] at (2, 2.05) {$\operatorname{supp}(\phi)$};
							\node[fill=white, inner sep=2pt, font=\small] at (2, -2.05) {$\operatorname{supp}(\phi)$};

							\draw[dotted, thick, black] (-3.5, 3.5) -- (-5.2, 3.5);
							\draw[dotted, thick, black] (-3.5, 0.6) -- (-4.2, 0.6);
							\draw[dotted, thick, black] (0, 0) -- (-5.2, 0); 
							\draw[dotted, thick, black] (-3.5, -0.6) -- (-4.2, -0.6);
							\draw[dotted, thick, black] (-3.5, -3.5) -- (-5.2, -3.5);

							\draw[<->, thick, black] (-4.9, 0) -- (-4.9, 3.5) node[midway, fill=white, inner sep=1pt] {$\mathscr{D}_2$};
							\draw[<->, thick, black] (-4.9, 0) -- (-4.9, -3.5) node[midway, fill=white, inner sep=1pt] {$\mathscr{D}_2$};
							
							\draw[<->, thick, black] (-3.9, 0) -- (-3.9, 0.6) node[midway, left] {$r$};
							\draw[<->, thick, black] (-3.9, 0) -- (-3.9, -0.6) node[midway, left] {$r$};
						\end{scope}
					\end{tikzpicture}
				}
				
			\end{minipage}\hspace{6 mm}
			\begin{minipage}{0.47\linewidth}
				
				\resizebox{\textwidth}{!}{%
					\begin{tikzpicture}[
						>=Latex,
						agent/.style={circle, fill=black, inner sep=1.5pt}, %
						hatch/.style={pattern=north east lines, pattern color=black!50},
						note/.style={align=left, font=\small}
						]

						\begin{scope}[xshift=12.2cm]
							
							\node[font=\bfseries] at (0, 4.5) {Communication Domain of Agent $i$};

							\draw[->, thick] (-6.7, -0.4) -- (4.6, -0.4) node[right] {$x$};
							\draw[->, thick] (0, -3.3) -- (0, 4.2) node[left] {$y$};

							\fill[hatch] (-3.5, 1.0) rectangle (4.2, 3.9);
							\fill[hatch] (-3.5, -3.1) rectangle (4.2, -0.2);

							\draw[thick] (-3.5, 3.9) -- (4.2, 3.9);
							\draw[thick] (-3.5, 1.0) -- (4.2, 1.0);
							\draw[thick] (-3.5, -0.2) -- (4.2, -0.2);
							\draw[thick] (-3.5, -3.1) -- (4.2, -3.1);

							\foreach \y in {-2.0, -1.3, -0.6, 0.0, 0.4, 0.8, 1.5} {
								\draw[gray!30, thin] (-3.5, \y) -- (4.2, \y);
							}

							\node[agent] (a1) at (-1.0, -2.0) {}; \draw[->, thick] (a1) -- ++(0.9, 0);
							\node[agent] (a2) at (0.2, -2.0) {};  \draw[->, thick] (a2) -- ++(0.3, 0);
							\node[agent] (a3) at (1.5, -2.0) {};  \draw[->, thick] (a3) -- ++(0.7, 0);

							\node[agent] (b1) at (-1.8, -1.3) {}; \draw[->, thick] (b1) -- ++(0.4, 0);
							\node[agent] (b2) at (-0.5, -1.3) {}; \draw[->, thick] (b2) -- ++(0.8, 0);
							\node[agent] (b3) at (1.0, -1.3) {};  \draw[->, thick] (b3) -- ++(0.2, 0);
							\node[agent] (b4) at (2.2, -1.3) {};  \draw[->, thick] (b4) -- ++(1.0, 0);

							\node[agent] (c1) at (-2.2, -0.6) {}; \draw[->, thick] (c1) -- ++(0.6, 0);
							\node[agent] (c2) at (-0.8, -0.6) {}; \draw[->, thick] (c2) -- ++(0.3, 0);
							\node[agent] (c3) at (0.5, -0.6) {};  \draw[->, thick] (c3) -- ++(1.1, 0); 
							\node[agent] (c4) at (2.0, -0.6) {};  \draw[->, thick] (c4) -- ++(0.5, 0);

							\node[agent] (d1) at (-2.0, 0.0) {}; \draw[->, thick] (d1) -- ++(0.8, 0);
							\node[agent] (d2) at (-0.2, 0.0) {}; \draw[->, thick] (d2) -- ++(0.4, 0);
							\node[agent] (d3) at (1.2, 0.0) {};  \draw[->, thick] (d3) -- ++(0.9, 0);
							\node[agent] (d4) at (2.8, 0.0) {};  \draw[->, thick] (d4) -- ++(0.2, 0);

							\node[agent] (e1) at (-1.5, 0.4) {};  \draw[->, thick] (e1) -- ++(0.3, 0);
							\node[agent] (ci) at (0.8, 0.4) {};   \draw[->, thick] (ci) -- ++(0.7, 0);
							\node[agent] (e3) at (3.2, 0.4) {};   \draw[->, thick] (e3) -- ++(0.5, 0);

							\draw[thick] (ci) circle (3.5pt);
							\node[above right=3pt, font=\small, fill=white, inner sep=1pt] at (ci) {agent $i$};

							\node[agent] (f1) at (-0.5, 0.8) {};  \draw[->, thick] (f1) -- ++(0.9, 0);
							\node[agent] (f2) at (2.4, 0.8) {};   \draw[->, thick] (f2) -- ++(0.4, 0);
							\node[agent] (f3) at (3.5, 0.8) {};   \draw[->, thick] (f3) -- ++(0.6, 0);

							\node[agent] (g1) at (0.5, 1.5) {};   \draw[->, thick] (g1) -- ++(0.5, 0);
							\node[agent] (g2) at (2.0, 1.5) {};   \draw[->, thick] (g2) -- ++(0.8, 0);

							\draw[dotted, thick, black] (-3.5, 3.9) -- (-6.5, 3.9);
							\draw[dotted, thick, black] (-3.5, -3.1) -- (-6.5, -3.1);
							\draw[<->, thick, black] (-6.2, 0.4) -- (-6.2, 3.9) node[midway, fill=white, inner sep=1pt] {$\mathscr{D}_2$};
							\draw[<->, thick, black] (-6.2, 0.4) -- (-6.2, -3.1) node[midway, fill=white, inner sep=1pt] {$\mathscr{D}_2$};

							\draw[dotted, thick, black] (-6.5, 0.4) -- (0.8, 0.4);

							\draw[dotted, thick, black] (-5.5, 1.5) -- (-3.5, 1.5);   
							\draw[dotted, thick, black] (-5.5, -2.0) -- (-3.5, -2.0); 
							\draw[<->, thick, black] (-5.2, -2.0) -- (-5.2, 1.5) node[midway, fill=white, inner sep=1pt] {$\mathscr{D}_2$};

							\draw[dotted, thick, black] (-3.5, 1.0) -- (-4.5, 1.0);
							\draw[dotted, thick, black] (-3.5, -0.2) -- (-4.5, -0.2);
							\draw[<->, thick, black] (-4.2, 0.4) -- (-4.2, 1.0) node[midway, left] {$r$};
							\draw[<->, thick, black] (-4.2, 0.4) -- (-4.2, -0.2) node[midway, left] {$r$};
							
						\end{scope}
						
					\end{tikzpicture}%
				}
				
			\end{minipage}
		}
		\caption{A schematic of the flocking paradigm of Theorem \ref{t:flockingCS}. Each agent may carry a cylinder of degeneracy around itself, inside which it does not interact with other agents.  Flocking will still occur if the communication is strong enough outside this cylinder and if no two such cylinders contain all the agents.  Here $\mathscr{D}_2$ is the diameter of the  configuration of the agents in $\mathbb{R}^{d-1}$, and our assumptions on $\phi$ concern only the annulus $r\le |y|\le \mathscr{D}_2$.} 
		\label{fig:degenerate-slice-communication}
	\end{figure}

	\FloatBarrier
	
	\subsection{Background: particle systems; kinetic limits and hydrodynamic closures; sticky particle dynamics}
	
	\label{ss:Background}
	
	This section reviews the literature most  relevant to our work.  We first consider the discrete $N$-agent Cucker--Smale system from which \eqref{e:EA} arises and the flocking dynamics that motivate its study.  Next, we introduce the Vlasov-type kinetic equation obtained in the large-crowd limit $N\to \infty$.  The macroscopic density $\rho$ and momentum $\rho \bu$ obtained as moments of the solution of the kinetic equation satisfy an  \textit{Euler--Reynolds Alignment} system that reduces to \eqref{e:EA} under a closure assumption that makes the Reynolds stress vanish identically.  We discuss alternative approaches to dealing with the Reynolds stress, devoting special attention to the recent work \cite{CarrilloChoiTadmor2026} of Carrillo, Choi, and Tadmor.  
	Finally, we discuss the wellposedness theory for \eqref{e:EA} and recent progress inspired by the literature on the pressureless Euler equations, to which \eqref{e:EA} reduces when $\phi \equiv 0$. 
	In this section, we drop the boldface notation for elements of $\mathbb{R}^d$, except for the $d$-dimensional macroscopic velocity, which we continue to denote by $\bu$ when $d>1$.

	\subsubsection{Cucker--Smale-type systems and flocking dynamics}
	\label{sss:CS}
	
	The Cucker--Smale system of ODEs
	\cite{CS2007a, CS2007b}, written for $N$ agents with (time-independent) masses $m_i\ge 0$, positions $x_i = x_i(t)\in \mathbb{R}^d$, and velocities $v_i = v_i(t) \in \mathbb{R}^d$, reads
	\begin{equation}
		\label{e:CS}
		\left\{ 
		\begin{split} 
			\dot{x}_i & = v_i,\\
			\dot{v}_i & = \sum_{j\in [N]\backslash \{i\}} m_j\phi(x_i - x_j)(v_j - v_i).
		\end{split}
		\right.
	\end{equation}
	
	This system is typically equipped with a radial kernel $\phi$ that is at least Lipschitz, except possibly at the origin.  Wellposedness is then trivial for bounded protocols, and collisions are the only possible issue for singular $\phi$. As shown in \cite{CCMP2017}, agents never collide if $\phi$ blows up like $|x|^{-1}$ or faster at the origin (so the system is wellposed).  Work of Peszek \cite{Pes2014,Pes2015} treats wellposedness for \textit{weakly singular} protocols, which blow up more slowly than $|x|^{-1}$ at the origin.
	
	Cucker--Smale-type systems are notable for their compatibility with flocking dynamics.  The term `flocking' is associated with several related long-time behaviors, such as velocity alignment, 
	$\lim_{t\to \infty} \max_{i,j}|v_i(t) - v_j(t)| = 0$, uniform boundedness of the spatial configuration,   $\sup_{t\ge 0} \max_{i,j}|x_i(t) - x_j(t)| <+\infty$,
	and crystallization, i.e., existence of  
	$\lim_{t\to \infty} (x_i(t) - x_j(t))$ for each $i,j$.  These discrete characterizations carry over readily to other contexts; for example, the convergence to a traveling wave discussed in Section \ref{sss:flockingresultsintro} is the natural hydrodynamic analog of what we have just called `crystallization'.  The behavior captured by such mathematical phenomena mimics---in a simplistic but analytically tractable way---the alignment dynamics of many-agent systems in the real world, the prototypical example being a flock of birds.  At the level of the equations, the driving mechanism appears in the `alignment forcing' on the right side of equations \eqref{e:EA}$_2$ and \eqref{e:CS}$_2$.  The emergence of a global structure (in this case, flocking) from possibly small-scale, pairwise interactions is a complex phenomenon; understanding exactly which initial configurations and which communication protocols lead to flocking behavior is a subject of much active research.  We make no attempt to give a comprehensive account of the by-now-expansive literature on the subject.  Rather, we refer the reader to the survey papers \cite{MT2014, CCP2017, Tadmor2021review, ShvydkoySurvey}, the book \cite{ShvydkoyBook}, and references therein, for a more thorough overview of the state of the art, and we limit our remaining discussion here to two threads that bear most directly on our own results. 
	
	\subsubsection{Selected flocking results}
	
	One of the best-known criteria for flocking, due to Ha and Liu \cite{HaLiu2008}, says that if $\phi$ is lower-bounded everywhere by a radially nonincreasing function $\underline{\phi}$ that is nonintegrable at infinity (i.e., it is `heavy-tailed'), then flocking occurs for any initial data.  Dietert and Shvydkoy \cite{DietertShvydkoy2019} have studied \eqref{e:CS} and \eqref{e:EA} under the weaker assumption that $\phi$ is lower bounded by a heavy-tailed $\underline{\phi}$ \textit{outside a ball} centered at the origin.  Under additional technical assumptions, they prove that flocking still occurs (though not precisely in the sense described above).  Our Theorems \ref{t:flockingCS} and \ref{t:flockingEA} allow a cylindrical degeneracy similar in spirit to their spherical one, but the two results cover essentially disjoint regimes: Ours requires a special geometry while permitting more degeneracy in the communication, whereas theirs handles arbitrary initial conditions but requires the region of degeneracy to be bounded. Dietert and Shvydkoy employ a sophisticated `corrector' method to treat the degeneracy near zero, while our result falls out of the geometry after a simple adaptation of the Ha--Liu argument.

	When a limiting configuration exists (such as a time-asymptotic density profile $\rho_\infty$), it is natural to ask about its structure.  This is a subtle question, but results exist in special settings, including for 1D weak solutions \cite{LeslieTan2024, Galtung2025} and 1D or unidirectional classical solutions \cite{LS2019, LLST2020geometric,Lear2021unising} of the Euler Alignment system, and for 1D Cucker--Smale systems \cite{HaKimParkZhang2019, HaParkZhang2018}.  Though our paper does not explicitly speak to the structure of $\rho_\infty$, it opens the door to studying this question for  unidirectional weak solutions.

	\subsubsection{Kinetic limit and hydrodynamic closure}
	
	\label{sss:kinetic}
	
	Taking the number of agents $N$ in \eqref{e:CS} to infinity, through a BBGKY-type argument \cite{HT2008} or a mean-field limit \cite{HaLiu2008, carrillo2010particle, MuchaPeszek2018}, yields a Vlasov-type equation for the kinetic distribution $f(x,v,t)$:
	\begin{equation}
		\label{e:kineticCS}
		\renewcommand*{\arraystretch}{1.5}
		\left\{ \begin{array}{l}
			\p_t f + v\cdot \n_x f + \n_v \cdot \big(fF(f)\big) = 0, \;\;(x,v,t)\in \R^d \times \R^d \times (0,\infty), \\
			F(f)(x,v,t) = \displaystyle\int_{\R^d\times\R^d} f(\widetilde{x},\widetilde{v},t)\phi(x-\widetilde{x})(\widetilde{v}-v)\,\dd\widetilde{v}\,\dd\widetilde{x}.
		\end{array}\right.
	\end{equation}
	
	When $\phi$ is bounded and Lipschitz, the wellposedness of \eqref{e:kineticCS} in the space of measures is well-established; see for example \cite{HaLiu2008, CanizoCarrilloRosado2011}.  The case of singular $\phi$ (i.e., $\phi\sim |x|^{-s}$ near the origin) is more subtle, especially when $s$ is large; local and global wellposedness results for various ranges of $s$ appear in \cite{MuchaPeszek2018,CarrilloChoiHauray2014,ChoiJung2023kin, ChoiJung2024kin, ChoiJung2025kin}.

	Given a kinetic distribution $f$, the associated macroscopic density and momentum are defined through moments of $f$:
	\begin{equation} 
		\label{e:fmoments}
		\rho(x,t) = \int_{\R^d} f(x,v,t)\,\dv, 
		\qquad 
		\rho(x,t) \bu(x,t) = \int_{\R^d} vf(x,v,t)\,\dv.
	\end{equation} 
	If $f$ solves \eqref{e:kineticCS}, then $\rho$ and $\bu$ evolve according to the following Euler--Reynolds Alignment system: 
	\begin{equation}
		\label{e:EAReynolds}
		\renewcommand*{\arraystretch}{1.6}
		\left\{ \begin{array}{rcl}
			\p_t \rho + \n_x\cdot (\rho \bu) & = & 0, \qquad \qquad (x,t)\in \R^d\times (0,\infty),\\
			\p_t (\rho \bu) + \n_x \cdot (\rho \bu \otimes \bu+\cR) & = & 
			\displaystyle \int_{\R^d} \rho(x,t)\rho(\widetilde{x},t) \phi(x-\widetilde{x}) (\bu(\widetilde{x},t) - \bu(x,t))\,\dd\widetilde{x},\\
			\cR(x,t) & = & \displaystyle\int_{\R^d} (v-\bu(x,t))\otimes (v-\bu(x,t))f(x,v,t)\,\dv.
		\end{array}
		\right.
	\end{equation}
	Note that \eqref{e:EAReynolds} does not constitute a closed system for $(\rho, \bu)$, as the kinetic  distribution function $f$ still appears in the momentum equation through the Reynolds stress $\cR$.  The pressureless Euler Alignment system \eqref{e:EA} corresponds to the regime where exactly one velocity is associated to each location in the support of the density profile, i.e., the kinetic distribution function is \textit{monokinetic}, $f(x,v,t) = \rho(x,t) \delta(v - \bu(x,t))$, in which case $\cR$ vanishes identically.  The kinetic-to-pressureless hydrodynamic limit has been justified rigorously, for bounded $\phi$, in the presence of an additional strong local alignment forcing which drives the velocity distribution to concentrate; see  \cite{KangVasseur2015,FigalliKang2019,ShvydkoyBook} for details.  
	
	The analysis of the present manuscript is valid only in the monokinetic setting, where $\cR\equiv 0$, but one can consider other regimes.  When $f(x,v,t) = (2\pi)^{-\frac{d}{2}}\rho(x,t)\exp(-\frac12|v-\bu(x,t)|^2)$, the term $\n_x \cdot \cR$ becomes a pressure gradient $\n_x \rho$.  Works related to the corresponding hydrodynamic limits include \cite{KMT2013,KMT2014,KMT2015} (bounded protocols) and \cite{ChoiKim2023} (singular protocols); see also \cite{CCTT2016,ChoiEAPressure2019} for analysis of the corresponding \textit{isothermal} Euler Alignment system (and \cite{AmadoriChristoforou2022,AmadoriChristoforou2024, AmadoriChristoforou2026} when $\phi \equiv 1$).  The \textit{isentropic} Euler Alignment system, with pressure term $\n_x(\rho^\gamma),  \gamma\geq1$, was treated in \cite{ConstantinDrivasShvydkoy2020,chen2020global}; see also \cite{ChoiHwang2024} for a hydrodynamic limit (with $\phi$ singular) and \cite{CarrilloChenYuanZatorska2025} for global 1D entropy solutions via an inviscid limit. Tadmor \cite{Tadmor2021review, Tadmor2023, Tadmor2026} has studied the question of what can be said about the dynamics of the system \eqref{e:EAReynolds} equipped with any member of a class of `entropic' pressures, but without specifying a more explicit closure assumption.  Finally, Carrillo, Choi, and Tadmor \cite{CarrilloChoiTadmor2026} have studied \eqref{e:EAReynolds} with $\cR$ determined by the Lagrangian dynamics, in the case where $\phi$ is bounded.\footnote{Actually, \cite{CarrilloChoiTadmor2026} and \cite{Tadmor2023} treat more general `$p$-alignment' dynamics (not discussed further in this paper), of which \eqref{e:EAReynolds} constitutes the special case $p=2$.}  
	\subsubsection{Reynolds stress determined from Lagrangian dynamics: further discussion of \cite{CarrilloChoiTadmor2026}}
	For bounded, Lipschitz $\phi$, the Lagrangian formulation of \eqref{e:kineticCS} equipped with initial condition $f^0$ can be written as
	\begin{equation}
		\label{e:kinCSchar}
		\renewcommand*{\arraystretch}{1.6}
		\left\{ \begin{array}{l}
			\dot{X}(x,v,t) = V(x,v,t),\qquad (x,v,t)\in \mathbb{R}^d\times \mathbb{R}^d \times (0,\infty),\\
			\dot{V}(x,v,t) = \displaystyle\int_{\mathbb{R}^d\times \mathbb{R}^d}  \phi(X(x,v,t) - X(\widetilde{x},\widetilde{v},t))(V(\widetilde{x},\widetilde{v},t) - V(x,v,t))f^0(\widetilde{x},\widetilde{v})\,\dd\widetilde{x}\,\dd\widetilde{v},  \\
			X(x,v,0) = x;\;\; V(x,v,0) = v.
		\end{array}
		\right.
	\end{equation}
	If $(X,V)$ solves \eqref{e:kinCSchar} and $f(t) = (X(t),V(t))_\sharp f^0$ is the pushforward of $f^0$ by $(X(t), V(t))$, then $f$ solves \eqref{e:kineticCS}, and consequently, $\rho$ and $\bu$ defined via \eqref{e:fmoments} satisfy \eqref{e:EAReynolds}.  The paper \cite{CarrilloChoiTadmor2026} treats the situation where $f^0(x,v) = \rho^0(x)\delta(v - \bu^0(x))$ is monokinetic.  In this case, $(\rho, \bu)$ solves the \textit{pressureless} system \eqref{e:EA} precisely when the Reynolds stress $\cR$---which is \textit{defined} by \eqref{e:EAReynolds}$_3$---vanishes identically.  The authors show that 
	$\cR$ vanishes as long as the Lagrangian trajectories remain injective, and that it vanishes in the time-asymptotic limit for heavy-tailed $\phi$ (though it may be nonzero at intermediate times). Their framework thereby reduces the question of global-in-time existence for \eqref{e:EA} to the question of whether 
	$\cR$ ever becomes nontrivial. 
	
	Under their hypothesis of a bounded, Lipschitz $\phi$, trajectories that collide must have distinct velocities at the time of collision, and continuing them via \eqref{e:kinCSchar} leads to a failure of monokineticity and the appearance of a nonzero Reynolds stress. In the unidirectional setting, our construction instead continues the dynamics by merging colliding trajectories, with post-collisional velocity determined by conservation of momentum.   This is the choice selected by the entropy conditions for \eqref{e:ScalarBalance}; it guarantees the preservation of monokineticity and the vanishing of the Reynolds stress.  The two frameworks provided by \cite{CarrilloChoiTadmor2026} and by the present work thus furnish different continuations of the same initial data past blowup, and in this regime our results are essentially disjoint from theirs, since the equations satisfied by these two continuations are different: Theirs solves \eqref{e:EAReynolds} with an explicit, nonvanishing Reynolds stress, while ours solves the pressureless system \eqref{e:EA}.
	
	\subsubsection{Wellposedness results for the pressureless Euler Alignment system}
	
	\label{sss:BackgroundWP}
	
	The wellposedness theory of the Euler Alignment system~\eqref{e:EA} has seen a good deal of progress in recent years, but much remains unknown.  The most complete theory is available when~$d=1$ and additionally~$\phi$ is even and locally integrable.  
	In this case, one can predict from the initial data whether the solution will remain regular for all time or lose regularity in finite time \cite{TT2014,CCTT2016,tan2020euler,L2019CTC}.  Furthermore, if the solution loses regularity in finite time, one can still continue it as a weak solution \cite{LeslieTan2023,HaHuangWang2014wk, Galtung2025}, and it is even possible to say a fair amount about the large-time structure of the density profile  \cite{LS2019,LLST2020geometric, LeslieTan2024, Galtung2025}. In what follows, we will be concerned exclusively with the above-mentioned case where $\phi$ is locally integrable (and satisfies the assumptions of Section \ref{sss:phiassumptions}).\footnote{We note, however, that there are relevant parallel lines of research concerning the cases where $\phi$ is instead strongly singular near the origin (see for example \cite{KT,DKRT,ShvydkoyTadmorI,ShvydkoyTadmorII,ShvydkoyTadmorIII, Shvydkoy2018NearlyAligned, ArnaizCastro2019, DanchinMuchaPeszekWroblewski2018,FabisiakPeszek2024}) or depends on the density profile (i.e., `topological' protocols, see for example \cite{STtopo, ReynoldsShvydkoy2020, LearReynoldsShvydkoy2022, LesliePeszek2026}).} 
	
	Most of the literature on the 1D case leans heavily on the presence of an additional conserved quantity and on the ordering of the real line---properties that both disappear in higher dimensions.  Largely for this reason, wellposedness remains much less well-understood in general than it is on~$\mathbb{R}$ (but see for example  \cite{Shvydkoy2018NearlyAligned, tan2021eulerian, HT2016, CarrilloChoiTadmor2026, ChoiTadmor2026} for some partial results). The \textit{unidirectional} dynamics studied in this paper provide a setting that retains many of the technical advantages of the pure 1D case; on the other hand, the question of regularity in the transverse directions and the additional geometrical flexibility (especially with respect to the limiting configurations, cf. \cite{LLST2020geometric}) render it considerably richer both mathematically and phenomenologically than the 1D setting.  For bounded, Lipschitz, radial $\phi$, the existence of classical unidirectional solutions for `subcritical' initial data was first proven in \cite{LearShvydkoy2019}.  
	(See also \cite{Lear2021unising,LearShvydkoy2021, LiMiaoTanXue2024} for results in the strongly singular protocol regime.) As in 1D, solutions starting from `supercritical' unidirectional initial data will lose regularity in finite time.  However, unlike in the 1D case, the problem of continuing a solution beyond a blowup has not previously been addressed; this is the subject of the present work.  
	
	\subsubsection{Ideas from the pressureless Euler literature}
	
	\label{sss:BackgroundSP}
	The pressureless Euler system, to which \eqref{e:EA} reduces when $\phi\equiv 0$, has a longer history\footnote{It is beyond the scope of this paper to give a complete overview of the literature on the pressureless Euler system, but we flag several important results not discussed elsewhere in this section: \cite{Zeldovich1969sb,Dermoune1999,DermouneMoutsinga2003,Moutsinga2008,nguyen2008pressureless,BrenierGangboSavareWestdickenberg2013,NguyenTudorascu2015,CavallettiSedjroWestdickenberg2015,CavallettiSedjroWestdickenberg2019,  Hynd2019LagrangianSP, Hynd2020trajectory, Hynd2020probmeasureSP, carrillogaltung2025, LiuPego2025, BressanChenHuang2025a,  BressanChenHuang2025b}.} than \eqref{e:EA} itself, and its literature provides a natural source of inspiration, especially when $\phi$ is locally integrable. Bouchut \cite{Bouchut1994} observed as early as 1994 that atomic solutions of the 1D pressureless Euler system are not unique, even under natural entropy conditions on $\rho$ and $\rho u$.  Once atoms collide, there is a range of admissible post-collisional velocities.  Of all the possibilities, completely inelastic (or \textit{sticky particle}) collision rules dissipate the most kinetic energy.  These collision rules can serve as a selection principle at the atomic level but do not apply directly to generic initial densities.  The early  literature \cite{Grenier1995,ERykovSinai1996,WangDing1997,WangHuangDing1997,BouchutJames1998, BouchutJames1999, HuangWang2001stickyunique} essentially resolved this question in 1D, from several different perspectives, all of which are compatible with the sticky particle collision rules at the atomic level.  
	
	The nonlocality of \eqref{e:EA} makes it challenging to adapt the techniques of the pressureless Euler literature mentioned above, as  much of the latter relies on certain explicit formulas induced by the constant-velocity trajectories at the particle level.\footnote{One exceptional case worth noting: If $\phi$ is identically constant, the alignment forcing loses its nonlocal character and degenerates to a linear damping term. In this case, many of the techniques from the pressureless Euler literature become easier to adapt; see  \cite{HaHuangWang2014wk, Jin2015flockingdiss, Jin2016flockingdiss}.}  One useful framework, however, is due to Brenier and Grenier \cite{brenier1998sticky}, who recast the 1D pressureless Euler system as a single scalar conservation law, the wellposedness theory of which is well-understood; they also showed that their solutions could be approximated by discrete sticky particle dynamics. Recently, the second author and Tan observed in \cite{LeslieTan2023} that in 1D, the dynamics of \eqref{e:EA} can, surprisingly, be encoded in a nonlocal scalar balance law.  Drawing on the framework of \cite{brenier1998sticky} (but with significant additional challenges caused by the nonlocal term), they established wellposedness in the space of measures.  Their solutions were constructed as limits of atomic solutions subject to sticky particle rules at collision times and Cucker--Smale dynamics in between. Later, in \cite{LeslieTan2024}, they used techniques partially inspired by work of Natile and Savar\'e \cite{NatileSavare2009} to determine which Lagrangian labels eventually stick together.  Also influenced by \cite{NatileSavare2009}, Galtung \cite{Galtung2025} exhibited a gradient flow structure for the 1D Euler Alignment system and demonstrated that the solutions generated this way are equivalent to those of \cite{LeslieTan2023}.  
	
	The picture for both \eqref{e:EA} and the pressureless Euler system is much less clear in higher dimensions.  Even for finitely many particles, the data-to-solution map for the sticky particle dynamics is not in general continuous, and whether or not trajectories collide is no longer a stable phenomenon.    Bressan and Nguyen \cite{BressanNguyen2014} constructed initial conditions for the sticky particle dynamics (with constant velocities between collisions) for which uniqueness fails, and another set of initial conditions for which existence fails.  Bianchini and Daneri \cite{BianchiniDaneri2023} have constructed unique  solutions of the pressureless Euler system for a large but restricted class of initial data (excluding the configurations of \cite{BressanNguyen2014}).  Their solutions are subject to sticky collision rules but nevertheless exhibit free-flow dynamics.  Carrillo et al. \cite{Carrilloetal2017weakEA} used convex integration techniques to  generate nonunique energy-dissipating solutions for a family of equations including \eqref{e:EA} (see also the recent paper \cite{HuangShiWang2025} on the pressureless Euler system specifically). Taken together, these results illustrate the challenges of exploring the unrestricted multi-D setting.  On the other hand, the present work confirms that the structural restrictions imposed by unidirectionality are enough to restore wellposedness in this specialized but still multidimensional regime, even in the presence of a nonlocal alignment force.

	\subsection{Weak formulations, entropy conditions, and standing assumptions}
	
	\label{ss:Weakformulations}
	
	Throughout the discussion above, we have not yet made precise what we mean by `weak solution' of \eqref{e:EAuni} or `entropy solution' of \eqref{e:ScalarBalance}. We remedy this presently, and we also discuss the assumptions on $\phi$ that make our theory work.
	
	\subsubsection{Weak formulations of the unidirectional Euler Alignment system}
	We let $\cM$ denote the space of measures, and $\cP_c$ the space of compactly supported probability measures, both equipped with the weak-$*$ topology (see Section \ref{ss:notation} below).
	
	\begin{DEF}[Weak formulation of \eqref{e:EAuni}]
		\label{def:weak_EA}
		Let $\rho^0 \in \mathcal{P}_c(\mathbb{R}\times\mathbb{R}^{d-1})$ be a compactly supported probability measure, and let $u^0 \in L^\infty(\rho^0)$ be a measurable velocity field. We say that a pair $(\rho, u)$ is a \textit{weak solution} to the unidirectional Euler Alignment system \eqref{e:EAuni} on $[0,T)$ associated to the initial data $(\rho^0, u^0)$ if the following conditions hold:
		\begin{enumerate}[label=(\alph*)]
			\item $\rho \in C([0,T); \mathcal{P}_c(\mathbb{R}\times\mathbb{R}^{d-1}))$ and  $u\in L^\infty(\rho(\dx, \dy,\cdot)\,\dt)$. 
			\item As $t\to 0+$, we have $\rho(t)\stackrel{*}{\rightharpoonup}\rho^0$ and $\rho u(t)\stackrel{*}{\rightharpoonup}\rho^0 u^0$ in $\cM(\mathbb{R}^d)$.
			\item The pair $(\rho, u)$ satisfies  \eqref{e:EAuni}$_1$ in the sense of distributions on $\mathbb{R}\times \mathbb{R}^{d-1}\times (0,T)$, and satisfies \eqref{e:EAuni}$_2$ on $(0,T)$ in the following sense.  For any $\zeta\in C^\infty_c(\mathbb{R}\times \mathbb{R}^{d-1}\times (0,T))$, we have 
			\begin{small} 
			\begin{equation}
				\label{e:wkmomentumsymmetrized}
				\begin{split}
					& \int_0^T \!\!\int_{\mathbb{R}^{d-1}} \int_{\mathbb{R}} \big( (\partial_t \zeta + u \partial_1 \zeta) u \big) \,\rho(\dx, \dy, t) \,\dt \\
					& = \frac12\int_0^T\!\!\int_{\mathbb{R}^{d-1}}\!\int_{\mathbb{R}} \int_{\mathbb{R}^{d-1}}\! \int_{\mathbb{R}} \phi(x\!-\!\xi, y\!-\!\eta)(u(x,y,t)\!-\!u(\xi, \eta,t))(\zeta(x,y,t)\!-\!\zeta(\xi, \eta, t))\,\rho(\dx, \dy, t)\,\rho(\dd\xi, \dd\eta, t) \,\dt.
				\end{split}
			\end{equation}
			\end{small} 
		\end{enumerate}
		We say that $(\rho, u)$ is a global-in-time weak solution if it is a weak solution on $[0,T)$ for every $T>0$.
	\end{DEF}
	
	We will also work with the following reformulation of \eqref{e:EAuni}, which has an unambiguous distributional interpretation.
	\begin{equation}
		\label{e:EAunipsi}
		\left\{ 
		\begin{split}
			\partial_t \rho + \partial_1 (\rho u) = 0, &  \qquad (x,y)\in \mathbb{R}\times \mathbb{R}^{d-1},\,t>0,\\
			\partial_t (\rho\psi) + \partial_1 (\rho\psi u) = 0, & \qquad \psi = u + \Phi * \rho.
		\end{split}
		\right. 
	\end{equation}
	Here $\Phi(x,y) = \int_0^x \phi(\xi, y)\, \dd\xi$ as before and 
	$*$ represents convolution in the spatial variables $x$ and $y$.  (Note that $\Phi*\rho = B[\nu, M]$ if $\rho = \partial_x M \nu$.) We show in Section \ref{s:Recovery} that, for appropriate $\phi$, the distributional form of \eqref{e:EAunipsi} implies that \eqref{e:wkmomentumsymmetrized} holds.

	\begin{DEF}[Weak formulation of \eqref{e:EAunipsi}]
		\label{def:weakEAunipsi}
		Let $\rho^0 \in \mathcal{P}_c(\mathbb{R}\times\mathbb{R}^{d-1})$ be a compactly supported probability measure, and let $u^0 \in L^\infty(\rho^0)$ be a measurable velocity field. We say that $(\rho, u)$ is a \textit{weak solution} of \eqref{e:EAunipsi} on $[0,T)$, associated to the initial data $(\rho^0, u^0)$, if it satisfies requirements (a) and (b) of Definition \ref{def:weak_EA} and if \eqref{e:EAunipsi} holds in the sense of distributions on $(0,T)$.
	\end{DEF}

	\subsubsection{Definition of entropy solutions for the scalar balance law}
	The natural entropy condition for the family \eqref{e:ScalarBalance} of coupled scalar balance laws is the following inequality, formulated in the sense of distributions on $\mathbb{R}\times \mathbb{R}^{d-1}\times (0,T)$. 
	\begin{equation}
		\label{e:entropydistributional}
		\partial_t |M-\alpha| \nu  + \partial_x \big( \sgn(M-\alpha)(A(M,y) - A(\alpha, y))\big) \nu \le B[\nu, M]\,\partial_x|M-\alpha| \nu,
		\qquad \text{ for all } \alpha\in \mathbb{R}.
	\end{equation}
	More explicitly, for every $\alpha\in \mathbb{R}$ and every nonnegative $\zeta\in C^\infty_c(\mathbb{R}\times \mathbb{R}^{d-1}\times (0,T))$, an entropy solution $M$ should satisfy
	\begin{equation}\label{e:entropyintegral}
		\int_0^T\!\!\int_{\mathbb{R}^{d-1}}\int_{\mathbb{R}}
		\Big[
		|M\!-\!\alpha|\,\p_t\zeta
		+ \sgn(M\!-\!\alpha)\big(A(M,y)\!-\!A(\alpha,y)\big)\,\p_x\zeta
		+\zeta\,B[\nu, M]\,\p_x|M\!-\!\alpha|
		\Big]\dx\,\nu(\dy)\,\dt\ge0.
	\end{equation}
	Note that by taking $\zeta(x,\widetilde{y},t) = \widetilde{\zeta}(x,t)b_\e(\widetilde{y}-y)$, where $b_\e$ is a standard mollifier, and taking $\e\to 0+$, we can reduce \eqref{e:entropyintegral} to the condition that for $\nu$-a.e.  $y$, we have 
	\begin{equation}\label{e:entropyintegraly}
		\int_0^T\int_{\mathbb{R}}
		\Big[
		|M-\alpha|\,\p_t\widetilde{\zeta}
		+ \sgn(M-\alpha)\big(A(M,y)-A(\alpha,y)\big)\,\p_x\widetilde{\zeta}
		+\widetilde{\zeta}\,B[\nu, M]\,\p_x|M-\alpha|
		\Big]\dx\,\dt\ge0,
	\end{equation}
	for every nonnegative $\widetilde{\zeta}\in C_c^\infty(\RR\times(0,T))$.  We will use both formulations \eqref{e:entropyintegral} and \eqref{e:entropyintegraly} in what follows.

	\begin{DEF} 
		\label{def:SBIC}
		We say that the triple $(M^0, A, \nu)$ of initial data, flux, and transverse measure is \textit{admissible} for the system \eqref{e:ScalarBalance} if the following conditions are satisfied.  
		\begin{itemize}
			\item $\nu$ is a compactly supported probability measure on $\mathbb{R}^{d-1}$.  (We will use $R(\nu)$ to denote $\diam(\supp(\nu))$ below.)
			\item The function $M^0:\mathbb{R}\times \mathbb{R}^{d-1}\to \R$ is Borel measurable.  Furthermore, there exists $R^0>0$ such that for $\nu$-a.e. $y$, we have that $M^0(x,y)=0$ for $x<-R^0$ and $M^0(x,y)=1$ for $x>R^0$.  Finally, the function $x\mapsto M^0(x,y)$ is nondecreasing on $\mathbb{R}$ for $\nu$-a.e. $y$. 
			\item The function $A:[0,1]\times \mathbb{R}^{d-1}\to \mathbb{R}$ is Borel measurable, and normalized so that $A(0,y) = 0$ for $\nu$-a.e. $y\in \mathbb{R}^{d-1}$.  Furthermore, the Lipschitz seminorm of the map $m\mapsto A(m,y)$ is uniformly bounded in $y$.
		\end{itemize}
	\end{DEF} 
	
	\begin{DEF}
		\label{def:entropysoln}
		Fix $T>0$.  We say that a function $M\in C([0,T); L^1_{\loc}(\mathbb{R}\times \mathbb{R}^{d-1}, \dx \otimes \nu))$ is an \textit{entropy solution} of~\eqref{e:ScalarBalance} associated to the admissible triple $(M^0, A, \nu)$ if the following conditions are satisfied. 
		\begin{enumerate}[label = (\alph*)]
			\item There exists $R(T)>0$ such that for all $t\in [0,T)$ and $\nu$-a.e. $y\in \mathbb{R}^{d-1}$, we have that $M(x,y,t)=0$ for $x<-R(T)$ and $M(x,y,t) = 1$ for $x>R(T)$.  Furthermore, the function $x\mapsto M(x,y,t)$ is nondecreasing on $\mathbb{R}$ for $\nu$-a.e. $y\in \mathbb{R}^{d-1}$ and for all $t\in [0,T)$.
			\item $M$ satisfies \eqref{e:entropydistributional} in the sense of distributions on $\mathbb{R}\times \mathbb{R}^{d-1} \times (0,T)$.
			\item We have $\lim_{t\to 0+} \|M(t) - M^0\|_{L^1(\mathbb{R}\times \mathbb{R}^{d-1}, \dx \otimes \nu)} = 0$. 
		\end{enumerate}  
		We say that  $M\in C([0,\infty); L^1_{\loc}(\mathbb{R}\times \mathbb{R}^{d-1}, \dx \otimes \nu))$ is a global-in-time entropy solution of \eqref{e:ScalarBalance} if it is an entropy solution on $[0,T)$ for every $T>0$.
	\end{DEF}

	\subsubsection{Assumptions on the communication protocol}
	
	\label{sss:phiassumptions}
	
	In our analysis, we always assume that $\phi(x,y)$ is even:
	\begin{equation} 
		\label{e:phieven}
		\phi(x,y) = \phi(-x,-y),
		\qquad (x,y)\in \mathbb{R}\times \mathbb{R}^{d-1}.
	\end{equation}
	We will also need some regularity on $\phi$.  For most of our results, the following assumptions are more than enough:
	\begin{equation}
		\label{e:phiupperbd}
		\text{There exists a nonincreasing } \overline{\phi} \in L^1_{\loc}([0,\infty)) \text{ such that } \phi(\bm{x}) \le \overline{\phi}(|\bm{x}|) \text{ for all } \bm{x}\in \mathbb{R}^d,
	\end{equation}
	\begin{equation} 
		\label{e:philoccnts}
		\phi \text{ is Lipschitz on $\{\bm{x}: |\bm{x}|\ge r\}$, for every $r>0$.}  
	\end{equation} 
	The type of protocol we have in mind here is either bounded and Lipschitz, or of the form $\phi(x,y) = \frac{\widetilde{\phi}(x,y)}{(|x|^{qs_1} + |y|^{qs_2})^{1/q}} $, where $s_1,s_2\in (0,1)$, $q\ge \frac{1}{\min\{s_1, s_2\}}$, and $\widetilde{\phi}$ is Lipschitz on $\mathbb{R}^d\backslash \{0\}$ and bounded.
	
	Actually, most of our theory---everything involving the scalar balance law \eqref{e:ScalarBalance} or the reformulation \eqref{e:EAunipsi} rather than \eqref{e:EAuni} itself---depends on $\Phi$ rather than $\phi$, and for these parts of the theory, we can get by with less than what is implied by \eqref{e:phiupperbd}--\eqref{e:philoccnts} under the relationship $\Phi(x,y) = \int_0^x \phi(\xi, y)\,\dd\xi$.  The relevant assumptions on $\Phi$ are that 
	\begin{equation}
		\label{e:Phiodd}
		\Phi(x,y) = -\Phi(-x,-y),\;\;\text{ with }\Phi(0,y) = 0,
		\qquad x\in \mathbb{R}, \; y\in \mathbb{R}^{d-1},
	\end{equation}
	\begin{equation}
		\label{e:Phinondecr}
		x\mapsto \Phi(x,y) \text{ is absolutely continuous and nondecreasing for each } y\in \mathbb{R}^{d-1},
	\end{equation}
	\begin{equation}
		\label{e:Phibddcptsets}
		\Phi \text{ is measurable and bounded on compact subsets of } \mathbb{R}\times \mathbb{R}^{d-1}.
	\end{equation}
	
	For most (but not all) purposes, we also assume that
	\begin{equation}
		\label{e:Phiymodulus}
		|\Phi(x,y) - \Phi(x,\widetilde{y})| \le (|x|+1) \,\omega(|y-\widetilde{y}|),
		\qquad x\in \mathbb{R},\;\;y,\widetilde{y}\in \mathbb{R}^{d-1},
	\end{equation}
	where $\omega$ is a modulus of continuity:
	\begin{equation}
		\label{e:moc}
		\omega:[0,\infty)\to [0,\infty) \text{ satisfies }\omega(0) = 0\text{ and is increasing, continuous, and concave.}
	\end{equation}
	The assumption \eqref{e:Phiymodulus}--\eqref{e:moc} is compatible with $\phi$ of the form $\phi(x,y) = \phi_1(x)\phi_2(y)$, where $\phi_1\in L^1_{\loc}(\mathbb{R})$ is bounded on the complement of any ball centered at the origin and $\phi_2$ is bounded and uniformly continuous. Assumptions \eqref{e:phieven}--\eqref{e:philoccnts} are easily seen to imply \eqref{e:Phiodd}--\eqref{e:moc} if $\phi$ and $\Phi$ are related via $\Phi(x,y) = \int_0^x \phi(\xi, y)\,\dd\xi$.  But one can also formulate sufficient conditions on $\phi$ that are weaker than \eqref{e:phiupperbd}--\eqref{e:philoccnts} and still guarantee that \eqref{e:Phinondecr}--\eqref{e:moc} hold for such $\Phi$.  We will carefully specify our assumptions in each theorem below.  
	
	\subsection{Notation, conventions, and general preliminaries}

	\label{ss:notation}
	Most of the notation we use is standard, but let us collect a few conventions here for the convenience of the reader.  For $n\in \mathbb{N}$, we write $[n]$ for $\{1,\ldots, n\}$. Unless otherwise specified, `$C$' will denote a constant that may change from line to line but is independent of any quantities that may later be taken to zero or infinity.  All functions and measures we consider will be defined on some Euclidean space or a subset thereof, unless explicitly specified otherwise.  We will use $[g]_{C^{0,\gamma}}$ to denote the H\"older seminorm of $g$ of order $\gamma\in [0,1]$, i.e.,  $[g]_{C^{0,\gamma}} = \sup_{x_1\ne x_2} \frac{|g(x_1) - g(x_2)|}{|x_1 - x_2|^\gamma}$; when $\gamma=1$ it is the Lipschitz seminorm (and we will sometimes write $[\cdot]_{\Lip}$ instead of $[\cdot]_{C^{0,1}}$), and when $\gamma=0$ it represents the oscillation of $g$.  We use the word `measurable' to mean `Borel measurable' unless explicitly noted otherwise.  
	We will use $\cM$ to denote the space of signed Radon measures, and $\cP$ will denote the space of probability measures (with $\cP_c$ denoting probability measures with compact support).  We will always equip $\cM$ and $\cP$ with the weak-$*$ topology, defined via convergence with respect to test functions $\zeta\in C_c$ which are continuous and compactly supported.  (We will use the standard notation $\stackrel{*}{\rightharpoonup}$ to denote weak-$*$ convergence.) However, all the probability measures we will consider will belong to a family whose supports are contained in some compact set.  For such a family, weak-$*$ convergence is equivalent to weak convergence (i.e., convergence with respect to test functions which are continuous and bounded), and the weak topology is metrizable via the Wasserstein-1 metric, or equivalently the bounded Lipschitz distance.  For probability measures $\mu$ and $\widetilde{\mu}$ on $\mathbb{R}^k$, this is defined as follows:
	\begin{equation} 
		W_1(\mu, \widetilde{\mu}) = \inf_{\bpi} \int\!\!\!\int |y-\widetilde{y}|\, \bpi(\dy, \dd\widetilde{y})
		= \sup\bigg\{ \int f(y) \,\mu(\dy) - \int f(y) \,\widetilde{\mu}(\dy) :  [f]_{\Lip} \le 1\bigg\},
	\end{equation} 
	where the infimum in the first expression is taken over all couplings $\bpi$ whose first marginal is $\mu$ and whose second marginal is $\widetilde{\mu}$, i.e., all $\bpi\in \cP(\mathbb{R}^k\times \mathbb{R}^k)$ such that $\bpi(E_1\times \mathbb{R}^k) = \mu(E_1)$ and $\bpi(\mathbb{R}^k\times E_2) = \widetilde{\mu}(E_2)$ for all Borel sets $E_1, E_2$ of $\mathbb{R}^k$.  We will use $W_1$ to denote the Wasserstein-$1$ distance in any spatial dimension $k$ (though $k=1$, $k=d-1$, and $k=d$ will be the only cases we use).  The value of $k$ will be clear from context.  We also use (in Section \ref{ss:rate2}) the Wasserstein-$p$ metric for other values of $p\in [1,\infty)$, defined via $W_p(\mu, \widetilde{\mu})^p = \inf_{\bpi} \int\!\!\int |y - \widetilde{y}|^p \,\bpi(\dy, \dd\widetilde{y})$.  Note that for $\mu, \widetilde{\mu}\in \cP_c(\mathbb{R}^k)$, the map $p\mapsto W_p(\mu, \widetilde{\mu})$ is nondecreasing in $p$, by Jensen's inequality, and furthermore that $W_p(\mu, \widetilde{\mu})^p$ is convex in each of its arguments.  In particular, if $\nu_Q$ is a probability measure on a Borel set $Q$, and $y\mapsto \widetilde{\mu}_y$ is a measurable map from $Q$ into $\cP_c(\mathbb{R}^k)$, then 
	\begin{equation}
		\label{e:Wppconvex}
		W_p\big( \mu, \widetilde{\mu}\big)^p \le \int_Q W_p(\mu, \widetilde{\mu}_y)^p\,\nu_Q(\dy),
		\qquad \widetilde{\mu} = \int \widetilde{\mu}_y\, \nu_Q(\dy). 	
	\end{equation}
	
	Given $\mu\in \cM$ and a measurable function $f$, we denote by $f_{\sharp} \mu$ the \textit{pushforward} measure of $\mu$ by $f$, defined via $f_{\sharp} \mu(E) = \mu(f^{-1}(E))$.  This implies that $\int g(x) (f_\sharp \mu)(\dx) = \int g\circ f(x)\, \mu(\dx)$ for any measurable function $g$ such that $g\circ f\in L^1(\mu)$. 
	
	For a right-continuous, nondecreasing function $f:\mathbb{R}\to [0,1]$, we define its \textit{generalized inverse} via 
	\begin{equation}
		\label{e:geninvdef}
		f^{-1}(m) = \inf\{ x\in \mathbb{R} : f(x)\ge m\}.  
	\end{equation}
	If $f:\mathbb{R}\to [0,1]$ is the cumulative distribution function of a probability measure $\mu$ on $\mathbb{R}$, i.e., $f(x) = \mu((-\infty, x])$, then $\mu$ is the pushforward of Lebesgue measure on $[0,1]$ under $f^{-1}$.  If $\widetilde{\mu}$ is another probability measure on $\mathbb{R}$ with cumulative distribution function $\widetilde{f}$, then 
	\begin{equation}
		\label{e:Wpid}
		W_p(\mu, \widetilde{\mu}) = \|f^{-1} - \widetilde{f}^{-1}\|_{L^p([0,1])}, \qquad p\in [1,\infty);  
	\end{equation}
	\begin{equation}
		\label{e:W1id}
		W_1(\mu, \widetilde{\mu}) = \|f - \widetilde{f}\|_{L^1(\mathbb{R})}.  
	\end{equation}
	Furthermore, if $g$ is bounded and measurable, then 
	\begin{equation} 
		\label{e:geninvid}
		\int_{(-\infty,x]} g(\xi) \,\mu(\dd\xi) = \int_0^{f(x)} (g\circ f^{-1})(m)\,\dm,
		\qquad \text{ for all } x\in \mathbb{R}.
	\end{equation}
	Finally, for $G\in BV(\mathbb{R})$ and $f$ bounded and measurable, we will write (for example) $\int f \partial_x G\, \dx$ for the integral of $f$ against the measure $\partial_x G$.

	\section{Uniqueness and stability}
	
	\label{s:Stability}
	
	\begin{THEOREM}\label{t:Mstab} 
		Fix $T>0$.  Let $(M^0, A, \nu)$ and $(\widetilde{M}^0, \widetilde{A}, \widetilde{\nu})$ be two admissible triples, with associated entropy solutions $M$ and $\widetilde{M}$, respectively, on $[0,T)$.  Assume $\Phi$ satisfies \eqref{e:Phiodd}--\eqref{e:Phibddcptsets}.  If $\nu \ne \widetilde{\nu}$, assume $\Phi$ also satisfies \eqref{e:Phiymodulus}--\eqref{e:moc}; if $\nu = \widetilde{\nu}$, interpret $\omega(W_1(\nu, \widetilde{\nu}))$ as zero in the inequalities below. Assume that $R(T)$ is chosen to  satisfy the requirements of Definition \ref{def:entropysoln} for both $M$ and $\widetilde{M}$.  Let $\bpi$ be an optimal coupling between $\nu$ and $\widetilde{\nu}$, associated to the cost $c(y,\widetilde{y}) = |y-\widetilde{y}|$.  Denote (for $\tau\in [0,T)$, $\gamma\in [0,1]$)
		\begin{equation}
			\label{e:cDcAdef}
			\mathcal{D}(\tau):=
			\int\!\|M(\cdot,y,\tau)-\widetilde M(\cdot,\widetilde{y},\tau)\|_{L^1}\,\bpi(\dd  y,\dd  \widetilde{y}),\qquad
			\cA_\gamma:= 
			\int\![A(\cdot,y)-\widetilde A(\cdot,\widetilde{y})]_{C^{0,\gamma}([0,1])}\,\bpi(\dd  y,\dd  \widetilde{y}).
		\end{equation}   
		Then for $0\le s\le t < T$, we have the following stability estimates.
		\begin{enumerate}[label = (\alph*)]
			\item Under no additional assumptions,
			\begin{equation}\label{e:stability1}
				\cD(t) \le
				\cD(s) + (t-s) \Big( \cA_1 + 
				4 (2R(T)+1) \,\omega(W_1(\nu, \widetilde{\nu}))\Big),
			\end{equation}
			In particular, if $(M^0, A, \nu) = (\widetilde{M}^0, \widetilde{A}, \widetilde{\nu})$, then $M(t) = \widetilde{M}(t)$ in $L^1(\dx\otimes \nu)$ for all $t\in [0,T)$. 
			
			\medskip
			\item If additionally we have that $\Phi(x,y)$ is Lipschitz continuous in $x$, uniformly in $y$, then for any $\gamma\in [0,1)$,
			\begin{small} 
			\begin{equation}
				\label{e:stability2}
				\cD(t)
				\le
				\bigg( \cD(s) 
				+ C \Big( R(T)^{1-\gamma} (t-s) \cA_\gamma\Big)^{\frac{1}{2-\gamma}} 
				+ 
				C(t-s)
				\, \Big( \cA_\gamma + (2R(T)+1) \,\omega(W_1(\nu, \widetilde{\nu})) \Big) \bigg)e^{C\|\partial_x \Phi\|_{L^\infty}(t-s)}.
			\end{equation}
			\end{small} 
		\end{enumerate}
	\end{THEOREM}
	
	The remainder of this section is dedicated to the proof of Theorem \ref{t:Mstab}, which we split across three subsections.  The proofs of both \eqref{e:stability1} and \eqref{e:stability2} rely on a doubling-of-the-variables argument, which we set up in Section \ref{ss:stabsetup} before specializing to \eqref{e:stability1} in Section \ref{ss:stab1} and \eqref{e:stability2} in Section \ref{ss:stab2}.  Before we begin the proof in earnest, we pause to recall the BV chain rule (cf. \cite{AmbrosioDalMaso1990}), which we will use repeatedly in our stability estimates (mostly without comment) and also later in Section \ref{s:Recovery}.
	
	\begin{LEMMA}
		\label{l:BVChain}
		Suppose $W\in BV_{\loc}(\mathcal{U})$, where  $\mathcal{U}$ is an open subset of $\mathbb{R}^k$.  Assume also that $f$ is a real-valued Lipschitz function whose domain of definition includes the image of $W$.  Then $f\circ W\in BV_{\loc}(\mathcal{U})$, and $|\partial_i (f\circ W)| \le [f]_{\Lip} |\partial_{i} W|$ for each $i=1,\ldots, k$ in the sense of measures.  Furthermore, there exists a single bounded Borel measurable function $g$ on $\mathcal{U}$, with $|g|\le [f]_{\Lip}$, such that $\partial_i (f\circ W) = g \partial_i W$ for each $i=1,\ldots, k$ on $\mathcal{U}$.
	\end{LEMMA}

	\subsection{Doubling of the variables and preliminary estimates}
	\label{ss:stabsetup} 
	
	We begin by following the usual Kru\v zkov doubling-of-the-variables argument.  We substitute $\alpha=\widetilde M(\widetilde x,\widetilde{y},\widetilde\tau)$ and $\widetilde{\zeta}(x,\tau)=w(x,\tau,\widetilde x,\widetilde\tau)$ into \eqref{e:entropyintegraly}: 
	\begin{equation}
		\label{e:Mdif1}
		0 \le \int_0^T\int_\RR \Big[
		|M-\widetilde M|\,\p_\tau w + \sgn(M-\widetilde M)\big(A(M,y)-A(\widetilde M,y)\big)\,\p_x w + w\,B[\nu,M]\,\p_x |M-\widetilde M|
		\Big] \, \dx\, \dtau .
	\end{equation}
	
	Similarly,
	\begin{equation}
		\label{e:Mdif2}
		0 \le \int_0^T\int_\RR \Big[
		|M-\widetilde M|\,\p_{\widetilde\tau} w + \sgn(\widetilde M-M)\big(\widetilde A(\widetilde M,\widetilde{y})-\widetilde A(M,\widetilde{y})\big)\,\p_{\widetilde x} w + w\,B[\widetilde\nu,\widetilde M]\,\p_{\widetilde x} |M-\widetilde M|
		\Big] \, \dd \widetilde x\, \dd \widetilde\tau .
	\end{equation}
	
	We integrate \eqref{e:Mdif1} over $\widetilde{x}$ and $\widetilde{\tau}$ and \eqref{e:Mdif2} over $x$ and $\tau$, then add the results together.  We also add and subtract the quantity $\sgn(M-\widetilde{M})(A(M,y) - A(\widetilde{M},y))\partial_{\widetilde{x}} w$, for reasons that will become clear below.  To keep the notation as light as possible, we omit the domain of integration for all spatial integrals.  The domain of integration will always be all of $\mathbb{R}$ for integrals in $x$ and $\widetilde{x}$.  The result of these manipulations is as follows.
	\begin{equation}
		\begin{aligned}\label{eq:prior-commonest}
			0  & \le \int_0^T \!\!\int_0^T\!\! \int\!\!\!\int
			|M-\widetilde M|(\p_\tau w+\p_{\widetilde\tau}w)
			\,\dx\,\dd \widetilde x\,\dd \tau\,\dd \widetilde\tau
			\\
			& \quad +\int_0^T \!\!\int_0^T\!\! \int\!\!\!\int
			\sgn(M-\widetilde M)\big(A(M,y)-A(\widetilde M,y)\big)
			(\p_x w+\p_{\widetilde x}w)
			\,\dd  x\,\dd  \widetilde x\,\dd  \tau\,\dd  \widetilde\tau
			\\
			& \quad -\int_0^T \!\!\int_0^T\!\!\int\!\!\!\int
			\sgn(M-\widetilde M)
			\Big(
			(A(M,y)-\widetilde A(M,\widetilde{y}))
			-
			(A(\widetilde M,y)-\widetilde A(\widetilde M,\widetilde{y}))
			\Big)
			\p_{\widetilde x}w
			\,\dd  x\,\dd  \widetilde x\,\dd  \tau\,\dd  \widetilde\tau
			\\
			& \quad +\int_0^T \!\!\int_0^T\!\!\int\!\!\!\int
			w\Big(
			B[\nu,M]\p_x|M-\widetilde M|
			+
			B[\widetilde\nu,\widetilde M]\p_{\widetilde x}|M-\widetilde M|
			\Big)\,\dd  x\,\dd  \widetilde x\,\dd  \tau\,\dd  \widetilde\tau .
		\end{aligned}
	\end{equation}
	
	Now, we fix a nonnegative even function $b\in C^\infty_c(\mathbb{R})$ with $\supp(b)\subset (-1,1)$, $\int b = 1$.  Define $b_\e(x) = \frac{1}{\e} b\big( \frac{x}{\e} \big)$ for each $\e>0$, so that  $b_\e\stackrel{*}{\rightharpoonup} \delta$ as $\e\to 0+$.  Next, fix $s,t\in [0,T)$ with $s<t$.  For $0<h<\min\{t-s, T-t\}$, let $\chi_h:[0,\infty)\to \mathbb{R}$ be a function which is identically zero outside $[s,t+h]$, identically $1$ on $[s+h, t]$, and linear on $[s,s+h]$ and $[t, t+h]$.  As $h\to 0+$, we have $\chi_h\stackrel{*}{\rightharpoonup} 1_{[s,t]}$  and   $\chi_h'\stackrel{*}{\rightharpoonup}\delta(\cdot - s) - \delta(\cdot - t)$.  Define 
	\[
	w_{\Delta, \e, h}(x,\tau,\widetilde{x},\widetilde{\tau}) = b_{\Delta}(x-\widetilde{x}) b_\e(\tau-\widetilde{\tau}) \chi_h(\tau).
	\]
	Then 
	\[
	(\partial_\tau + \partial_{\widetilde{\tau}})w_{\Delta, \e, h} = b_\Delta(x-\widetilde{x}) b_\e(\tau - \widetilde{\tau}) \chi_h'(\tau),
	\]
	\[
	(\partial_x + \partial_{\widetilde{x}})w_{\Delta, \e, h} = 0.
	\]
	We substitute $w = w_{\Delta, \e, h}$ in \eqref{eq:prior-commonest} (note that the second integral on the right vanishes completely) and take $\e\to 0+$, then $h\to 0+$.  After an additional integration, \eqref{eq:prior-commonest} becomes 
	\begin{equation}\label{e:commonest}
		\mathcal{D}_{\Delta}(t)\le \mathcal{D}_{\Delta}(s)+\int_s^t \mathcal F_{\Delta}(\tau)\,\dd \tau+\int_s^t \mathcal N_{\Delta}(\tau)\,\dd \tau,
	\end{equation}
	where 
	\begin{equation}\label{e:DDeltadef}
		\mathcal{D}_{\Delta}(\tau):=
		\int\!\!\!\int\!\!\!\int\!|M(x,y,\tau)-\widetilde M(\widetilde{x},\widetilde{y},\tau)| b_{\Delta}(x-\widetilde x)\,\dd  x\,\dd \widetilde{x}\,\bpi(\dd  y,\dd  \widetilde{y}),
	\end{equation}
	\begin{equation}\label{e:FDeltadef}
		\mathcal F_{\Delta}(\tau)
		:=\!\!
		\int\!\!\!\int\!\!\!\int
		f\big(M(x,y,\tau), \widetilde{M}(\widetilde{x}, \widetilde{y}, \tau), y, \widetilde{y}\big) 
		\,b_{\Delta}'(x-\widetilde x)\,\mathrm{d}x\,\mathrm{d}\widetilde x\;\bpi(\mathrm{d}y,\mathrm{d}\widetilde{y}),
	\end{equation}
	\begin{equation}
		\label{e:fdef}
		f(m,\widetilde{m},y,\widetilde{y}) = \sgn(m-\widetilde{m})\Big( (A(m,y) - \widetilde{A}(m,\widetilde{y})) - (A(\widetilde{m},y) - \widetilde{A}(\widetilde{m},\widetilde{y})) \Big),  
	\end{equation}
	and 
	\begin{equation}
		\label{e:NL}
		\mathcal{N}_\Delta(\tau)  \!=\! \int\!\!\!\int\!\!\!\int
			\Big(
			B[\nu,M](x,y,\tau)\p_x|M\!-\!\widetilde M|
			+
			B[\widetilde\nu,\widetilde M](\widetilde x,\widetilde{y},\tau)\p_{\widetilde x}|M\!-\!\widetilde M|
			\Big)
			b_\Delta(x-\widetilde{x})
			\,\dd  x\,\dd  \widetilde x\;\bpi(\dd  y,\dd  \widetilde{y}).
	\end{equation}
	
	We refer to $\mathcal{F}_\Delta$ (or $\int_s^t \mathcal{F}_\Delta(\tau)\,\dtau$) as the `flux term' and $\mathcal{N}_\Delta$ (or $\int_s^t \mathcal{N}_\Delta(\tau)\,\dtau$) as the `nonlocal term' in what follows.
	
	\medskip
	The inequality \eqref{e:commonest} serves as a starting point for the proofs of both stability estimates.  One other fact that is used in both proofs (though in different ways) is that $f(\cdot, \widetilde{m}, y, \widetilde{y})$ is Lipschitz (uniformly in $\widetilde{m}$), and  
	\begin{equation}
		\label{e:fCgamma}
		[f(\cdot, \widetilde{m}, y, \widetilde{y})]_{C^{0,\gamma}([0,1])} \le 2^{1-\gamma} [A(\cdot, y) - \widetilde{A}(\cdot, \widetilde{y})]_{C^{0,\gamma}([0,1])},    
	\end{equation}
	for all $\widetilde{m}\in [0,1]$,   all $y,\widetilde{y}\in \mathbb{R}^{d-1}$, and all  $\gamma\in [0,1]$. 
	
	We now prove each of \eqref{e:stability1} and \eqref{e:stability2} in turn.
	
	\subsection{First stability estimate}
	
	\label{ss:stab1} 
	
	In order to deduce \eqref{e:stability1} from \eqref{e:commonest}, we would like to take $\Delta\to 0+$.  However, before we can do so, we need to move the derivative that appears on $b_\Delta$ in the definition of $\cF_\Delta$.  We integrate by parts and use the BV chain rule (and \eqref{e:fCgamma}, with $\gamma = 1$) to estimate the result.  
	\[
	\begin{aligned}
		\bigg| \int_s^t \int\!\!\!\int & 
		\sgn(M-\widetilde M)
		\Big(
		(A(M,y)-\widetilde A(M,\widetilde{y}))
		-
		(A(\widetilde M,y)-\widetilde A(\widetilde M,\widetilde{y}))
		\Big)
		b_\Delta'(x-\widetilde{x})
		\,\dd  x\,\dd  \widetilde x\,\dd  \tau \bigg| 
		\\
		&=
		\bigg| \int_s^t \int\!\!\!\int
		\p_x
		\Big[
		f\big( M(x,y,\tau), \widetilde{M}(\widetilde{x}, \widetilde{y}, \tau), y, \widetilde{y}\big) 
		\Big]
		b_\Delta(x-\widetilde x)
		\,\dd  x\,\dd \widetilde x\,\dd \tau \bigg| \\
		& \le \int_s^t \int\!\!\!\int
		[A(\cdot,y)-\widetilde A(\cdot,\widetilde{y})]_{\Lip}\,
		\p_x M(x,y,\tau)\,
		b_\Delta(x-\widetilde x)
		\,\dd  x\,\dd  \widetilde x\,\dd \tau \\
		& = (t-s)\,[A(\cdot,y)-\widetilde A(\cdot,\widetilde{y})]_{\Lip}.
	\end{aligned}
	\]
	
	We substitute this bound into \eqref{e:commonest} and take $\Delta\to 0+$, to obtain 
	\begin{equation}
		\label{e:Kruzkovstep1}
		\begin{aligned}
			& \int \|M(\cdot,y,t)-\widetilde M(\cdot, \widetilde{y},t)\|_{L^1(\mathbb{R})} \, \bpi(\dy, \dd\widetilde{y}) \\
			& \;\; \le \int\!\! \big( \|M(\cdot,y,s)\!-\!\widetilde M(\cdot, \widetilde{y},s)\|_{L^1(\mathbb{R})}
			\!+\! (t\!-\!s) [A(\cdot,y)\!-\!\widetilde A(\cdot,\widetilde{y})]_{\Lip} \big) \,\bpi(\dy, \dd\widetilde{y})
			\!+\! \limsup_{\Delta\to 0+} \int_s^t \!\mathcal{N}_\Delta(\tau) \, \dtau.
		\end{aligned}
	\end{equation}
	It remains to estimate the nonlocal term.  We turn to this presently.  From now on we will suppress dependence on the time $\tau$, since the latter plays no role in the subsequent computations. We introduce the change of variables
	\begin{equation}
		\label{e:cov}
		x_\pm:=\frac{x\pm \widetilde x}{2},
		\qquad 
		\p_x=\frac12(\p_{x_+}+\p_{x_-}),
		\qquad
		\p_{\widetilde x}=\frac12(\p_{x_+}-\p_{x_-}),
	\end{equation} 
	so that $\mathcal{N}_\Delta$ becomes
	\begin{equation}\label{eq:BP-nonlocal-xpm}
		\begin{aligned}
			\mathcal{N}_\Delta & = \frac12\int\!\!\!\int\!\!\!\int
			\Big[
			\big(
			B[\nu,M](x_+ + x_-, y)
			+
			B[\widetilde\nu,\widetilde M](x_+ - x_-,\widetilde{y})
			\big)\p_{x_+}|M-\widetilde M| 
			\\
			&  \quad \qquad \quad + 
			\big(
			B[\nu,M](x_+ + x_-,y)
			-
			B[\widetilde\nu,\widetilde M](x_+ - x_-,\widetilde{y})
			\big)\p_{x_-}|M-\widetilde M|
			\Big] 
			b_{\frac{\Delta}{2}}(x_-)\,\dd  x_+\,\dd  x_-\,\;\bpi(\dd  y,\dd  \widetilde{y}).
		\end{aligned}
	\end{equation}
	
	We use $\big|\partial_{x_-}|M - \widetilde{M}|\big| \le \partial_1 M + \partial_1 \widetilde{M}$ and let $\Delta\to0+$ in the resulting upper bound for $\cN_\Delta$. This yields
	\begin{equation}\label{eq:BP-nonlocal-diagonal-xpm}
		\limsup_{\Delta\to0^+}
		\mathcal{N}_\Delta \le
		\frac12 ( I + II ),
	\end{equation}
	where 
	\begin{equation}\label{eq:BP-I-II-def}
		\begin{aligned}
			I
			&:=
			\int\!\!\!\int
			\big(
			B[\nu,M](x,y)+B[\widetilde\nu,\widetilde M](x,\widetilde{y})
			\big)\,
			\p_1|M(x,y)-\widetilde M(x,\widetilde{y})|
			\,\dd  x\,\bpi(\dd  y,\dd  \widetilde{y}),
			\\
			II
			&:=
			\int\!\!\!\int
			\Big|
			B[\nu,M](x,y)-B[\widetilde\nu,\widetilde M](x,\widetilde{y})
			\Big|
			\big(
			\p_1M(x,y)+\p_1\widetilde M(x,\widetilde{y})
			\big)
			\,\dd  x\,\bpi(\dd  y,\dd  \widetilde{y}).
		\end{aligned}
	\end{equation}

	\medskip
	
	We start by writing 
	\[
	\begin{aligned}
		& B[\nu,M](x,y)\pm B[\widetilde\nu,\widetilde M](x,\widetilde{y}) \\
		& =
		\int\!\!\!\int
		\Big[
		\Phi(x-\xi,y-\eta)\,\p_1 M(\xi,\eta)
		\pm
		\Phi(x-\xi,\widetilde{y}-\widetilde{\eta})\,\p_1\widetilde M(\xi,\widetilde{\eta})
		\Big]
		\,\dd \xi\,\bpi(\dd \eta,\dd \widetilde{\eta})\\
		& =
		\int\!\!\!\int
		\Big(
		\Phi(x-\xi,y-\eta)-\Phi(x-\xi,\widetilde{y}-\widetilde{\eta})
		\Big)\,
		\p_1 M(\xi,\eta)
		\,\dd \xi\,\bpi(\dd \eta,\dd \widetilde{\eta})
		\\
		& \quad
		+
		\int\!\!\!\int
		\Phi(x-\xi,\widetilde{y}-\widetilde{\eta})
		\big(
		\p_1 M(\xi,\eta)\pm\p_1 \widetilde{M}(\xi,\widetilde{\eta})
		\big)
		\,\dd \xi\,\bpi(\dd \eta,\dd \widetilde{\eta}).
	\end{aligned}
	\]

	\medskip
	
	Inserting this decomposition into the expressions for $I$ and $II$, we may write 
	\[
	I = I^{a}+I^{b},
	\qquad 
	II \le II^{a} + II^{b},
	\]
	where
	\begin{small} 
	\begin{align*}
		I^a
		& = 
		\int\!\!\!\int\!\!\!\int\!\!\!\int 
		\Big(
		\Phi(x-\xi,y-\eta)-\Phi(x-\xi,\widetilde{y}-\widetilde{\eta})
		\Big) 
		\p_1M(\xi,\eta) \,\p_1|M(x,y)-\widetilde M(x,\widetilde{y})| 
		\,\dd  x\,\dd \xi\,\bpi(\dd  y,\dd  \widetilde{y})
		\,\bpi(\dd \eta,\dd \widetilde{\eta}), \\ 
		I^{b}
		& =
		\int\!\!\!\int\!\!\!
		\int\!\!\!\int
		\Phi(x-\xi,\widetilde{y}-\widetilde{\eta})
		\big(
		\p_1 M(\xi,\eta)+\p_1\widetilde M(\xi,\widetilde{\eta})
		\big) \,\p_1\big|M(x,y)-\widetilde M(x,\widetilde{y})\big| \,\dx \,\dd \xi\,\bpi(\dd  y,\dd  \widetilde{y})\,\bpi(\dd \eta,\dd \widetilde{\eta})
		,\\
		II^a
		& = 
		\int\!\!\!\int \bigg| \int\!\!\!\int 
		\Big(
		\Phi(x-\xi,y-\eta)-\Phi(x-\xi,\widetilde{y}-\widetilde{\eta})
		\Big) 
		\p_1M(\xi,\eta) \, \dd\xi \,\bpi(\dd\eta, \dd\widetilde{\eta}) \bigg| \,\big( \partial_1 M(x,y)+ \partial_1 \widetilde M(x,\widetilde{y})\big)  
		\,\dd  x \,\bpi(\dd  y,\dd  \widetilde{y}) \\ 
		II^{b}
		& =
		\int\!\!\!\int \bigg| 
		\int\!\!\!\int
		\Phi(x-\xi,\widetilde{y}-\widetilde{\eta})
		\big(
		\p_1 M(\xi,\eta)-\p_1\widetilde M(\xi,\widetilde{\eta})
		\big) \,\dd\xi \,\bpi(\dd\eta, \dd\widetilde{\eta}) \bigg| \,\big( \partial_1 M(x,y)+ \partial_1 \widetilde M(x,\widetilde{y})\big)  
		\,\dd  x \,\bpi(\dd  y,\dd  \widetilde{y}).
	\end{align*}
	\end{small} 
	We estimate $I^a$ and $II^a$ together, leveraging the assumed regularity on $\Phi$ in the second component via \eqref{e:Phiymodulus}: 
	\begin{small} 
	\[
	\begin{aligned}
		& \frac12 (I^{a} + II^{a})\\
		& \;\;\le  (2R(T)+1) \!\int\!\!\! \int \!\bigg[
		\omega\big( \big| (y \!- \!\eta) \!-\! (\widetilde{y} \!- \!\widetilde{\eta}) \big|\big) \cdot\!\int 
		\!\p_1 M(\xi,\eta) \,\dd \xi 
		\cdot \!\int \!\big( \partial_1 M(x,y) \!+\! \partial_1\widetilde M(x,\widetilde{y}) \big) \,\dx \bigg] \,\bpi(\dd \eta,\dd \widetilde{\eta})\,\bpi(\dd  y,\dd  \widetilde{y}) \\
		& \;\;\le 2(2R(T)+1) \bigg[ \int \omega(|y - \widetilde{y}|) \,\bpi(\dy, \dd\widetilde{y}) + \int \omega(|\eta - \widetilde{\eta}|) \,\bpi(\dd\eta, \dd\widetilde{\eta})\bigg]\\
		& \;\;\le 4(2R(T)+1) \,\omega(W_1(\nu, \widetilde{\nu})).
	\end{aligned}
	\]
	\end{small} 
	Next, we show that $I^b + II^b\le 0$.  Since $\Phi$ is nondecreasing in its first argument by \eqref{e:Phinondecr}, we have (integrating by parts, estimating, and then integrating by parts again)
	\[
	\bigg|\int
	\Phi(x-\xi,\cdot)
	\big(
	\p_1 M(\xi,\eta)-\p_1 \widetilde M(\xi,\widetilde{\eta})
	\big)
	\,\dd \xi\bigg|
	\le 
	\int
	\Phi(x-\xi,\cdot)
	\p_1\big|M(\xi,\eta)-\widetilde M(\xi,\widetilde{\eta})\big|
	\,\dd \xi.
	\]
	Therefore, 
	\begin{align*} 
		II^{b}
		& \le 
		\int\!\!\!\int\!\!\!\int\!\!\!\int \Phi(x-\xi,\widetilde{y}-\widetilde{\eta}) \big(
		\p_1M(x,y)+\p_1\widetilde M(x,\widetilde{y})
		\big)
		\partial_1\big|M(\xi,\eta)-\widetilde M(\xi,\widetilde{\eta})\big|  \,\dd  x\,\dd\xi\,\bpi(\dd \eta,\dd \widetilde{\eta})\,\bpi(\dd  y,\dd  \widetilde{y}).
	\end{align*} 
	After  interchanging $(x,y,\widetilde{y})\leftrightarrow(\xi,\eta,\widetilde{\eta})$, and recalling that $\Phi$ is odd by assumption \eqref{e:Phiodd}, we see that the quantity on the right is exactly $-I^b$, justifying our claim.
	
	Putting together all the estimates above, we see that  
	\[
	\mathcal{N}_\Delta(\tau) \le
	\frac12 ( I^a + II^a + I^b + II^b )
	\le
	4 (2R(T)+1) \,\omega(W_1(\nu, \widetilde{\nu})).
	\]
	This estimate (which is uniform in $\tau\in [0,T)$), together with \eqref{e:Kruzkovstep1}, completes the proof of \eqref{e:stability1}.
	
	\medskip
	
	\subsection{Second stability estimate}
	\label{ss:stab2}
	
	As we prove our second stability estimate, we will perform our analysis with fixed $\Delta>0$ rather than taking $\Delta\to 0+$.  We will ultimately be estimating the same quantity as before, namely $\cD$ as defined in \eqref{e:cDcAdef}. However, we will first work with the related quantity $\mathcal{D}_\Delta$, which is defined in \eqref{e:DDeltadef} and appears on the left side of \eqref{e:commonest}.
	The discrepancy between $\mathcal{D}$ and $\mathcal{D}_\Delta$ is of order~$\Delta$ (see Lemma \ref{l:stab2prelim} below), so we may use our estimates on $\mathcal{D}_\Delta$ to write down an inequality for $\mathcal{D}(t)$ in terms of quantities involving $\Delta$ and $\frac{1}{\Delta}$.  We will ultimately make a choice of $\Delta$ to balance these two contributions and obtain \eqref{e:stability2}; however, we will write our argument with arbitrary $\Delta>0$ in order to keep its role transparent.  The main idea of this  strategy is inspired by work of Bouchut and Perthame \cite{BouchutPerthame1998}.
	
	We collect a few basic estimates in the following Lemma before turning to the main thrust of our proof.  
	
	\begin{LEMMA}
		\label{l:stab2prelim}
		The following estimates hold:
		\begin{equation}\label{e:DDDeltadif}
			|\mathcal{D}-\mathcal{D}_\Delta| \le \Delta
			;
		\end{equation}
		\begin{equation}\label{e:MMtildeb'}
			\Bigg|
			\int\!\!\!\int
			|M(x,y)-\widetilde M(\widetilde x,\widetilde{y})|\,b_{\Delta}'(x-\widetilde x)\,\dx\,\dd \widetilde x
			\Bigg|
			\le C.
		\end{equation}
		
	\end{LEMMA}
	\begin{proof} 
		We use repeatedly the fact that if $M:\R\to \R$ is the cumulative distribution function of a probability measure, then 
		\begin{equation}
			\int_\R \big| M(x+h) - M(x)\big|\, \dx \le |h|. 
		\end{equation}
		To prove \eqref{e:DDDeltadif}, we estimate as follows:
		\begin{align*}
			|\mathcal{D}-\mathcal{D}_\Delta|
			& = \bigg| \int\!\!\!\int\!\!\!\int \big( |M(x,y) - \widetilde{M}(x,\widetilde{y})| - |M(x,y) - \widetilde{M}(\widetilde{x},\widetilde{y})|\big) b_\Delta(x-\widetilde{x}) \,\dx \, \dd \widetilde{x} \,\bpi(\dy, \dd\widetilde{y}) \bigg| \\
			& \le
			\int\!\!\!\int\!\!\!\int
			|\widetilde M(x,\widetilde{y})-\widetilde M(\widetilde x,\widetilde{y})|
			\,b_\Delta(x-\widetilde x)\,\dx\,\dd \widetilde x\;\bpi(\dd  y,\dd  \widetilde{y}) \\
			& = \int \!\!\!\int \bigg( \int \big| \widetilde{M}(\widetilde{x}+h,\widetilde{y}) - \widetilde{M}(\widetilde{x},\widetilde{y})\big| \,\dd \widetilde{x} \bigg) b_\Delta(h)\,\dh \,\bpi(\dy, \dd\widetilde{y}) \\
			& \le
			\int |h|\,b_\Delta(h)\,\dd  h \le \Delta.
		\end{align*}
		
		We turn to the proof of \eqref{e:MMtildeb'}.  We use the fact that $\int b_\Delta'(x-\widetilde{x})\,\dx = 0$ in order to pass to the second line below.
		\begin{align*} 
			& \Bigg|
			\int\!\!\!\int
			|M(x,y)-\widetilde M(\widetilde x,\widetilde{y})|\,b_{\Delta}'(x-\widetilde x)\,\dx\,\dd \widetilde x
			\Bigg|\\
			& \qquad = \Bigg|
			\int\!\!\!\int
			\big[ |M(x,y)-\widetilde M(\widetilde x,\widetilde{y})|- |M(\widetilde x,y)-\widetilde M(\widetilde x,\widetilde{y})|\big] \,b_{\Delta}'(x-\widetilde x)\,\dx\,\dd \widetilde x
			\Bigg|\\
			& \qquad \le
			\int\!\!\!\int
			|M(x,y)-M(\widetilde x,y)|\,|b'_\Delta(x-\widetilde x)|
			\,\dx\,\dd \widetilde x \\
			& \qquad \le \int |h| |b_\Delta'(h)| \,\dh = C.
		\end{align*} 
	\end{proof}

	We now begin the proof of \eqref{e:stability2} in earnest, starting by estimating $\cF_\Delta$.  Again using that $\int b_\Delta'(x-\widetilde{x})\,\dx = 0$ and leveraging \eqref{e:fCgamma}, we have 
	\begin{equation} 
		\label{e:FDeltaest}
		\begin{split} 
			|\cF_\Delta| 
			& = \bigg| \int\!\!\!\int\!\!\!\int
			\Big[ f\big(M(x,y), \widetilde{M}(\widetilde{x}, \widetilde{y}), y, \widetilde{y}\big) - f\big(M(\widetilde{x},y), \widetilde{M}(\widetilde{x}, \widetilde{y}), y, \widetilde{y}\big) \Big] 
			\,b_{\Delta}'(x-\widetilde x)\,\mathrm{d}x\,\mathrm{d}\widetilde x\;\bpi(\mathrm{d}y,\mathrm{d}\widetilde{y}) \bigg| \\
			& \le 
			2^{1-\gamma}\int  [A(\cdot, y) - \widetilde{A}(\cdot, \widetilde{y})]_{C^{0,\gamma}([0,1])} \bigg( \int\!\!\!\int |M(\widetilde{x}+h,y) - M(\widetilde{x}, y)|^\gamma \,|b_\Delta'(h)| \,\dh\,\dd\widetilde{x} \bigg) \bpi(\mathrm{d}y,\mathrm{d}\widetilde{y})\\ 
			& 
			\le 
			2^{1-\gamma} \cA_\gamma  \bigg( \int |h| |b_\Delta'(h)|\,\dh \bigg)^\gamma \bigg( \int_{-R(T)-\Delta}^{R(T)+\Delta} \int |b_\Delta'(h)| \, \dh \, \dd  \widetilde{x} \bigg)^{1-\gamma} \\
			& \le C\Big( \big(\tfrac{ R(T)}{\Delta}\big)^{1-\gamma} + 1\Big)\cA_\gamma.
		\end{split}
	\end{equation} 
	
	Most of the remainder of our proof is dedicated to estimating the nonlocal term $\cN_\Delta$. Integrating by parts in \eqref{e:NL}, we write 
	\begin{equation}\label{eq:BP-NDelta-split}
		\mathcal N_{\Delta}
		=
		\mathcal{N}_{\Delta}^{(a)}+\mathcal{N}_{\Delta}^{(b)}
		+ \mathcal{N}_{\Delta}^{(c)},
	\end{equation}
	where (writing $\phi:= \partial_x \Phi$ and noting that $\|\partial_x B[\nu,M]\|_{L^\infty} \le \|\phi\|_{L^\infty}$, which is finite under our current assumptions)
	\begin{small} 
	\begin{align*}
		\mathcal N_{\Delta}^{(a)}
		& =
		-\int\!\!\!\int\!\!\!\int
		\Big(
		(\phi*(\nu\,\p_1M))(x,y)
		+
		(\phi*(\widetilde\nu\,\p_1\widetilde M))(\widetilde x,\widetilde{y})
		\Big)
		|M(x,y)-\widetilde M(\widetilde x,\widetilde{y})|\,b_{\Delta}(x-\widetilde{x})
		\,\dx\,\dd \widetilde x\;\bpi(\dd  y,\dd  \widetilde{y}),
		\\
		\mathcal N_{\Delta}^{(b)}
		& =-\int\!\!\!\int\!\!\!\int
		\big(
		B[\nu,M](x,y)-B[\nu,M](\widetilde x,y)
		\big)
		|M(x,y)-\widetilde M(\widetilde x,\widetilde{y})|\,b'_{\Delta}(x-\widetilde{x})
		\,\dx\,\dd \widetilde x\;\bpi(\dd  y,\dd  \widetilde{y}), \\
		\mathcal N_{\Delta}^{(c)}
		& =-\int\!\!\!\int
		\big(
		B[\nu,M](\widetilde{x},y)-B[\widetilde\nu,\widetilde M](\widetilde x,\widetilde{y})
		\big) \bigg( \int 
		|M(x,y)-\widetilde M(\widetilde x,\widetilde{y})|   \,b'_{\Delta}(x-\widetilde{x})
		\,\dx\bigg) \,\dd \widetilde x\;\bpi(\dd  y,\dd  \widetilde{y}).
	\end{align*} 
	\end{small} 
	The first of these is nonpositive and can be discarded:
	\begin{equation}\label{eq:BP-NDelta-a-est}
		\mathcal N_{\Delta}^{(a)}
		\le 0.
	\end{equation}
	The second requires a bit of manipulation but is straightforward to estimate:
	\begin{align*}
		|\cN_\Delta^{(b)}| & \le \int\!\!\!\int \big( \|\phi\|_{L^\infty} |h|\big) \bigg( \int \big[ |M(\widetilde{x}+h,y)  - M(\widetilde{x},y)| + |M(\widetilde{x},y) -  \widetilde{M}(\widetilde{x},\widetilde{y})|\big]\, \dd\widetilde{x} \bigg)  |b_\Delta'(h)| \, \dh \, \bpi(\dy, \dd\widetilde{y}) \\
		& \le \|\phi\|_{L^\infty}\int\!\!\!\int \bigg( |h| + \|M(\cdot, y) - \widetilde{M}(\cdot, \widetilde{y})\|_{L^1} \bigg) |h| |b_\Delta'(h)| \,\dh \,\bpi(\dy, \dd\widetilde{y}) \\
		& \le C \|\phi\|_{L^\infty} \big( \Delta + \mathcal{D} \big) 
	\end{align*}
	
	To estimate $\cN_\Delta^{(c)}$, we note first of all that 
	\begin{small} 
	\begin{align*}
		\big|B[\nu, M](\widetilde{x},y) - B[\widetilde{\nu}, \widetilde{M}](\widetilde{x},\widetilde{y}) \big| 
		& = \bigg| \int\!\!\!\int \big[ \Phi(\widetilde{x}-\xi, y-\eta) \partial_1 M(\xi, \eta) - \Phi(\widetilde{x}-\xi, \widetilde{y}-\widetilde{\eta}) \partial_1 \widetilde{M}(\xi, \widetilde{\eta})\big] \,\dd \xi\,\bpi(\dd \eta, \dd \widetilde{\eta})\bigg| \\
		& \le \|\phi\|_{L^\infty} \mathcal{D} + (2R(T)+1)\bigg[ \omega(|y-\widetilde{y}|) + \int \omega(|\eta - \widetilde{\eta}|)\,\bpi(\dd\eta, \dd\widetilde{\eta}) \bigg].
	\end{align*}
	\end{small} 
	We also note that the integral over $x$ in $\cN_\Delta^{(c)}$ can be treated using the same trick as in the proof of \eqref{e:MMtildeb'}:
	\begin{small} 
	\begin{align*} 
		\bigg| \int 
		|M(x,y)-\widetilde M(\widetilde x,\widetilde{y})|   \,b'_{\Delta}(x-\widetilde{x})
		\,\dx\bigg| 
		& = \bigg| \int 
		\big[ |M(x,y)-\widetilde M(\widetilde x,\widetilde{y})| - |M(\widetilde{x},y) - \widetilde{M}(\widetilde{x},\widetilde{y})|\big] \,b'_{\Delta}(x-\widetilde{x})
		\,\dx\bigg| \\
		& \le \int |M(x,y) - M(\widetilde{x},y)||b_\Delta'(x-\widetilde{x})| \,\dx.
	\end{align*} 
	\end{small} 
	The integral over $\widetilde{x}$ of this last quantity is uniformly bounded, by \eqref{e:MMtildeb'}.  Therefore, we may write 
	\begin{align*}
		|\cN_\Delta^{(c)}|
		& \le C \big(\|\phi\|_{L^\infty} \mathcal{D} + (2R(T)+1)\,\omega(W_1(\nu, \widetilde{\nu})) \big).
	\end{align*}
	
	Combining the estimates for $\cN_\Delta^{(a)}$, $\cN_\Delta^{(b)}$, and $\cN_\Delta^{(c)}$, and using \eqref{e:DDDeltadif}, we get
	\begin{equation}\label{e:NDeltacombined}
		\mathcal N_{\Delta}(\tau)
		\le
		C\Big( \|\phi\|_{L^\infty}\big(\mathcal{D}_\Delta(\tau) + \Delta \big) 
		+ (2R(T)+1)\,\omega(W_1(\nu, \widetilde{\nu})) \Big) .
	\end{equation}
	
	We are now in a position to finish the argument.  Substituting the estimates \eqref{e:FDeltaest} and \eqref{e:NDeltacombined} into \eqref{e:commonest} yields
\begin{small} 
	\begin{equation} 
		\label{e:DDeltapreGronwall}
		\mathcal{D}_{\Delta}(t)
		\le
		\mathcal{D}_{\Delta}(s)
		+
		C (t-s) \bigg[ 
		\Big( \big(\tfrac{ R(T)}{\Delta}\big)^{1-\gamma} + 1\Big) \cA_\gamma
		+ \|\phi\|_{L^\infty} \Delta 
		+ (2R(T)+1)\,\omega(W_1(\nu, \widetilde{\nu})) \bigg] 
		\!+ C\|\phi\|_{L^\infty}\! \int_s^t \!\mathcal{D}_{\Delta}(\tau) \,\dtau.
	\end{equation} 
	\end{small}

	Then, applying Gr\"onwall's inequality gives
	\begin{small} 
	\[
	\begin{aligned}
		\mathcal{D}_{\Delta}(t)
		&\le
		\mathcal{D}_{\Delta}(s)e^{C\|\phi\|_{L^\infty}(t-s)}
		+
		\bigg[ \Big( \big(\tfrac{ R(T)}{\Delta}\big)^{1-\gamma} \!+\! 1\Big) \cA_\gamma
		+\|\phi\|_{L^\infty}\,\Delta
		+ (2R(T)+1)\,\omega(W_1(\nu, \widetilde{\nu}))
		\bigg]
		\frac{e^{C\|\phi\|_{L^\infty}(t-s)}-1}{\|\phi\|_{L^\infty}}\\
		& \le e^{C\|\phi\|_{L^\infty}(t\!-\!s)}\bigg( \!\mathcal{D}_{\Delta}(s)
		\!+\!
		C(t\!-\!s)\bigg[ \Big( \big(\tfrac{ R(T)}{\Delta}\big)^{1-\gamma}\!+\!1\Big) \cA_\gamma
		\!+\! (2R(T)\!+\!1)\,\omega(W_1(\nu, \widetilde{\nu}))
		\bigg]\bigg) \!+\! ( e^{C\|\phi\|_{L^\infty}(t\!-\!s)}\!-\!1)\Delta .
	\end{aligned}
	\]
	\end{small} 
	(If $\phi\equiv 0$, this last bound is just \eqref{e:DDeltapreGronwall}.) Using \eqref{e:DDDeltadif}, we obtain
	\[
	\mathcal{D}(t)
	\le e^{C\|\phi\|_{L^\infty}(t-s)} \bigg( 
	\mathcal{D}(s)
	+ 2 \Delta 
	+ C(t-s) \bigg[ \Big( \big(\tfrac{ R(T)}{\Delta}\big)^{1-\gamma} + 1\Big)  \,\mathcal A_\gamma + (2R(T)+1)\,\omega(W_1(\nu, \widetilde{\nu})) \bigg] \bigg).
	\]
	
	Finally, we obtain \eqref{e:stability2} after choosing $\Delta = \big(R(T)^{1-\gamma} (t-s) \cA_\gamma \big)^{\frac{1}{2-\gamma}}$ and adjusting the constant $C$ appropriately.

	\section{Existence and construction of entropy solutions}
	
	\label{s:Existence}
	
	We now turn to the construction of solutions, via the procedure already previewed in Section \ref{sss:ExistenceConstruction}.  The full construction involves several steps and is spread out among Theorem \ref{t:SPCSentropy} in Section \ref{ss:Existdisc} (completely spatially discrete solutions), Theorem~\ref{t:sliceexist} in Section~\ref{ss:Existslice} (discrete-in-$y$ solutions), and Theorem~\ref{t:existencefinalstep} in Section \ref{ss:Existgen} (general solutions).  Theorem \ref{t:SPCSapproxM} in Section~\ref{ss:Existgen} also demonstrates that any entropy solution can be approximated directly by fully discrete dynamics.  Most of the analysis of Sections \ref{ss:Existdisc} and \ref{ss:Existslice} runs parallel to that of the second author's previous work \cite{LeslieTan2023} with Tan; the main new analytical contribution is the content of Section \ref{ss:Existgen}.  Though the complete formulations of the construction and approximation will have to wait until the statements of Theorems \ref{t:existencefinalstep} and \ref{t:SPCSapproxM} below, the  existence part of our end result is easy to state; we record it here for the convenience of the reader before beginning the construction.  
	\begin{THEOREM}
		\label{t:existence} 
		Assume $\Phi$ satisfies \eqref{e:Phiodd}--\eqref{e:Phibddcptsets} and \eqref{e:Phiymodulus}--\eqref{e:moc}. Let $(M^0, A, \nu)$ be an admissible triple, and fix $T>0$.  There exists a unique entropy solution $M\in C([0,T); L^1_{\loc}(\mathbb{R}\times \mathbb{R}^{d-1}, \dx\otimes \nu))$ of \eqref{e:ScalarBalance} on $[0,T)$ associated to $(M^0, A, \nu)$.    
	\end{THEOREM}

	\subsection{Technical preliminaries}
	In order to streamline our construction below of entropy solutions of \eqref{e:ScalarBalance}, we isolate two ingredients before getting started.  The first is a lemma that will provide convergence of various discretizations of the nonlocal term $B[\nu,M]\,\partial_x M$; we invoke it a total of five times in our proofs below.  The second ingredient we single out in this subsection is a simplification of the Kru\v zkov entropy condition \eqref{e:entropydistributional} in the case of piecewise constant solutions;  this simplification is completely standard for scalar conservation laws, but we believe it is worth spelling out in the context of our nonlocal system of scalar balance laws.

	\subsubsection{A convergence lemma on the nonlocal term}
	
	\begin{LEMMA}
		\label{l:NLprelim}
		Assume $\nu\in \cP_c(\mathbb{R}^{d-1})$.  Let $(\sigma^n)_{n=1}^\infty$ be a sequence of measurable maps $\sigma^n:\mathbb{R}^{d-1}\to \mathbb{R}^{d-1}$, and put $\nu^n = (\sigma^n)_\sharp \nu$.  Fix $R>0$; let $\overline{M}$ and $(M_n)_{n=1}^\infty$ be measurable functions from $\mathbb{R}\times \mathbb{R}^{d-1}$ to $\mathbb{R}$; assume that for $\nu$-a.e. $y\in \mathbb{R}^{d-1}$, $\overline{M}$ is nondecreasing in its first argument, with $\overline{M}(x,y) = 0$ for $x< -R$ and $\overline{M}(x,y) = 1$ for $x>R$.  Assume that each $M_n$ satisfies the same properties for $\nu^n$-a.e. $y\in \mathbb{R}^{d-1}$.  Assume that $\Phi$ satisfies \eqref{e:Phiodd}--\eqref{e:Phibddcptsets} and that additionally, we have either 
		\begin{enumerate}
			\item [(A)]   $\sigma^n=\id_{\mathbb{R}^{d-1}}$, $\nu$-a.e. in $\mathbb{R}^{d-1}$, for each $n\in \N$, or 
			\item [(B)] $\|\sigma^n - \id_{\mathbb{R}^{d-1}}\|_{L^\infty(\nu)}\to 0$ and $\Phi$ satisfies \eqref{e:Phiymodulus}--\eqref{e:moc}.  
		\end{enumerate}
		Finally, assume $M_n(\cdot, \sigma^n(\cdot))\stackrel{*}{\rightharpoonup} \overline{M}$ weak-$*$ in $L^\infty(\dx \otimes \nu)$.	Then as $n\to \infty$, we have  
		\begin{equation} 
			\label{e:BnuMunifconv}
			\sup_{x\in [-R,R]} |B[\nu^n,M_n](x,\sigma^n(y)) - B[\nu,\overline{M}](x,y)|\to 0, \qquad \text{ for } \nu\text{-a.e. } y\in \mathbb{R}^{d-1}.
		\end{equation} 
		Now, suppose  additionally that $(G_n)_{n=1}^\infty$ is a sequence of nondecreasing functions on $\mathbb{R}$, with $G_n(x)=0$ for $x<-R$ and $G_n(x) = 1$ for $x>R$, and that there exists $G:\mathbb{R}\to \mathbb{R}$ such that $G_n- G\to 0$ in $L^1(\mathbb{R})$.  Then for any $\alpha\in \mathbb{R}$, we have
		\begin{equation}
			\label{e:NLconvlemma}
			B[\nu^n, M_n](\cdot, \sigma^n(y)) \partial_x |G_n-\alpha|  \stackrel{*}{\rightharpoonup} B[\nu,\overline{M}](\cdot, y)\partial_x |G-\alpha|, \quad \text{ weak-$*$ in $\cM(\mathbb{R})$,  for  $\nu$-a.e. $y\in \mathbb{R}^{d-1}$.} 
		\end{equation}
	\end{LEMMA}
	\begin{proof} We start by writing
		\[
		B[\nu^n,M_n](x,\sigma^n(y)) - B[\nu, \overline{M}](x,y)
		= I_n(x,y) + II_n(x,y), 
		\]    
		where $I_n(x,y)$ and $II_n(x,y)$ are defined for $x\in [-R, R]$ by 
		\begin{equation}
			I_n(x,y):= \int_{\mathbb{R}^{d-1}} \int_{-R}^{R} \Big[ \Phi(x - \xi, \sigma^n(y) - \sigma^n(\eta)) - \Phi(x-\xi, y-\eta) \Big]  \partial_1 M_n(\dd\xi, \sigma^n(\eta)) \,\nu(\dd\eta),    
		\end{equation}
		\begin{equation} 
			II_n(x,y) = \int_{\mathbb{R}^{d-1}} \int_{-R}^{R} \partial_1 \Phi(x-\xi, y-\eta) \big( M_n(\xi, \sigma^n(\eta)) - \overline{M}(\xi, \eta) \big)\,\dd\xi\, \nu(\dd\eta).
		\end{equation} 
		In case (A), we have that $I_n(\cdot, y) \equiv 0$ for $\nu$-a.e. $y$; in case (B), we can use \eqref{e:Phiymodulus} to obtain 
		\begin{equation}
			\label{e:Inbd}
			\sup_{x\in [-R,R]}|I_n(x,y)| \le (2R + 1) \cdot \omega(2\|\sigma^n - \id_{\mathbb{R}^{d-1}}\|_{L^\infty(\nu)}),
			\qquad \text{ for }\nu\text{-a.e. } y\in \mathbb{R}^{d-1};
		\end{equation} 
		the right side clearly tends to zero.  As for $II_n$, we get pointwise convergence $II_n(x, y)\to 0$ from the assumed \mbox{weak-$*$} convergence $M_n(\cdot, \sigma^n(\cdot))\stackrel{*}{\rightharpoonup} \overline{M}$, as $(\xi, \eta)\mapsto \partial_1\Phi(x-\xi, y-\eta) 1_{[-R, R]}(\xi)$ belongs to $L^1(\dd\xi\otimes \nu)$ for each $(x,y)\in [-R,R]\times \mathbb{R}^{d-1}$.  To obtain \textit{uniform} convergence in $x$, it suffices to prove that the family $\{x\mapsto II_n(x,y)\}_{n=1}^\infty$ is equicontinuous in $x$ on $[-R, R]$, for fixed $y$.  To this end, we note that for $x,x+h\in [-R, R]$, we have 
		\begin{equation} 
			\label{e:IInequi}
			\begin{split} 
				|II_n(x+h,y) - II_n(x,y)|
				& \le \int_{\mathbb{R}^{d-1}} \int_{-R}^{R} |\partial_1\Phi(x+h - \xi, y-\eta) - \partial_1\Phi(x-\xi, y-\eta)|\,\dd\xi\,\nu(\dd\eta) \\
				& \le \int_{\mathbb{R}^{d-1}} \int_{-2R}^{2R} |\partial_1\Phi(z+h, y-\eta) - \partial_1\Phi(z, y-\eta)|\,\dz\,\nu(\dd\eta).
			\end{split} 
		\end{equation} 
		The inner integral over $z$ tends to zero with $h$; dominated convergence then allows passage to the limit in the outer integral.  The point is that the right side of \eqref{e:IInequi} is independent of $n$, so the fact that it tends to zero with $h$ implies that $II_n(\cdot,y)$ is  equicontinuous on $[-R, R]$.  Uniform convergence of $(II_n(\cdot,y))_{n=1}^\infty$ to zero, for fixed $y$, then follows. 
		
		To prove the convergence \eqref{e:NLconvlemma}, we write (for any $\widetilde{\zeta}\in C^\infty_c(\mathbb{R})$)
		\begin{align*}
			& \bigg|  \int_{\mathbb{R}} \widetilde{\zeta} B[\nu^n,M_n](x,\sigma^n(y)) \,\partial_x|G_n - \alpha| \,\dx -  \int_{\mathbb{R}} \widetilde{\zeta} B[\nu, \overline{M}](\cdot, y)\,\partial_x|G - \alpha| \,\dx\bigg|\\
			& \quad \le  \int_{\mathbb{R}} |\widetilde{\zeta}| |B[\nu^n,M_n](x,\sigma^n(y)) - B[\nu, \overline{M}](x,y)| \,\big|\partial_x |G_n - \alpha|\big| \,\dx \\
			& \qquad +  \bigg| \int_{\mathbb{R}} \widetilde{\zeta} B[\nu, \overline{M}](x,y) \,\partial_x |G_n - \alpha|\,\dx - \int_{\mathbb{R}} \widetilde{\zeta} B[\nu, \overline{M}](x,y) \,\partial_x |G - \alpha|\,\dx \bigg| .
		\end{align*}
		As the total variation of each $G_n$ is exactly $1$, it follows that the first term tends to zero for any $y$ along which \eqref{e:BnuMunifconv} holds.  The total variation bound, together with the convergence $G_n-G\to 0$ in $L^1(\mathbb{R})$, also allows us to conclude that $\partial_x |G_n - \alpha| \stackrel{*}{\rightharpoonup} \partial_x |G - \alpha|$ in $\cM(\mathbb{R})$; thus, noting that $x\mapsto B[\nu,\overline{M}](x,y)$ is continuous, the second term tends to zero as well.
	\end{proof}

	\subsubsection{Entropy conditions for piecewise constant solutions}
	
	We now recast the Kru\v zkov entropy condition \eqref{e:entropydistributional} in a form that will be easier to verify for the piecewise constant solutions we generate using sticky particle Cucker--Smale dynamics (introduced in the next subsection).  Equation \eqref{e:RH} and inequality \eqref{e:Oleinik} below together constitute the simplified criteria and will be referred to as `Rankine--Hugoniot' and `Oleinik' conditions, respectively, on our piecewise constant solutions.
	
	Suppose the measure $\nu$ is atomic; that is, assume it takes the form $\nu = \sum_{\ell=1}^L \mu_\ell \delta(\cdot - y_\ell)$.  Suppose additionally that for each $\ell\in [L]$, the function $M(x,y_\ell,t)$ (assumed to be nondecreasing in $x$, as usual, and equal to either $0$ or $1$ for $|x|$ large enough) is piecewise constant, with discontinuity curves that are absolutely continuous.  Such a function must take the form 
	\begin{equation}
		M(x,y_\ell, t) = \frac{1}{\mu_\ell}\sum_{i=1}^{I_\ell} m_{i,\ell} H(x - x_{i,\ell}(t)),
		\qquad 
		\sum_{i=1}^{I_\ell} m_{i,\ell} = \mu_\ell, \;\;m_{i,\ell}>0 \text{ for each }i,\ell,
	\end{equation}
	where $H:\mathbb{R}\to \mathbb{R}$ is the Heaviside function ($H(x) = 1_{[0,\infty)}(x)$), and each $x_{i,\ell}:[0,\infty)\to \mathbb{R}$ is  absolutely continuous.

	Write $\eta_\alpha(m) = |m-\alpha|$ and $q_\alpha(m, y_\ell) = \sgn(m-\alpha)(A(m,y_\ell) - A(\alpha, y_\ell))$. Then the  Kru\v zkov entropy condition \eqref{e:entropydistributional} reduces to the requirement that for each $\ell\in [L]$ and each $i\in [I_\ell]$, we have 
	\begin{equation}
		\label{e:entropypiecewise}
		\big( \dot{x}_{i,\ell}(t) + B[\nu,M](x_{i,\ell}(t), y_\ell, t) \big) [[\eta_\alpha(M(\cdot, y_\ell, t))]] \ge [[q_\alpha(M(\cdot, y_\ell, t), y_\ell)]],  
	\end{equation}
	for any time $t$ at which $x_{i,\ell}$ is differentiable, and any $\alpha\in \mathbb{R}$.  Here $[[f]]$ denotes the jump in $f$ across $x_{i,\ell}$, i.e., $[[f]] = f(x_{i,\ell}(t)+) - f(x_{i,\ell}(t)-)$.
	
	Denote $m_\pm = M(x_{i,\ell}(t)\pm, y_\ell, t)$.  If we take $\alpha \le m_-$ then  \eqref{e:entropypiecewise} takes the form
	\begin{equation}
		\label{e:entropypiecewise1}
		\big( \dot{x}_{i,\ell}(t) + B[\nu,M](x_{i,\ell}(t), y_\ell, t) \big) [[M(\cdot, y_\ell, t)]] \ge [[A(M(\cdot, y_\ell, t), y_\ell)]],  
	\end{equation}
	whereas taking $\alpha\ge m_+$ gives the opposite inequality.  Combining both cases gives us (for a.e. $t$)
	\begin{equation}
		\label{e:RH}
		\dot{x}_{i,\ell}(t) + B[\nu,M](x_{i,\ell}(t), y_\ell, t) = \frac{[[A(M(\cdot, y_\ell, t), y_\ell)]]}{[[M(\cdot, y_\ell, t)]]}.  
	\end{equation}
	The final case, namely $m_-<\alpha<m_+$, yields 
	\begin{small} 
	\begin{equation}
		\label{e:entropypiecewise2}
		\big( \dot{x}_{i,\ell}(t) + B[\nu,M](x_{i,\ell}(t), y_\ell, t) \big) \big( [[M(\cdot, y_\ell, t)]] - 2(\alpha - m_-)\big) \ge [[A(M(\cdot, y_\ell, t), y_\ell)]] - 2(A(\alpha, y_\ell) - A(m_-, y_\ell)).  
	\end{equation} 
	\end{small} 
	Substituting \eqref{e:RH} into \eqref{e:entropypiecewise2} and rearranging gives
	\begin{equation}
		\label{e:Oleinik}
		\frac{[[A(M(\cdot, y_\ell, t), y_\ell)]]}{[[M(\cdot, y_\ell, t)]]} \le \frac{A(\alpha, y_\ell) - A(m_-, y_\ell)}{\alpha - m_-},
		\qquad m_-<\alpha<m_+.
	\end{equation}
	We will refer to \eqref{e:RH} as the `Rankine--Hugoniot' condition and \eqref{e:Oleinik} as the `Oleinik entropy condition' for these kinds of solutions.  Together, \eqref{e:RH} and \eqref{e:Oleinik} are equivalent to \eqref{e:entropydistributional} for piecewise constant solutions of \eqref{e:ScalarBalance}.

	\subsection{Sticky Particle Cucker--Smale entropy solutions of the scalar balance law}
	\label{ss:Existdisc}
	
	We now construct the unique entropy solution associated to any admissible triple $(M^0, A, \nu)$ for which $\nu$ is atomic, $M^0$ is piecewise constant, and $A$ is piecewise linear.  As we will show in Theorem \ref{t:SPCSentropy}, the entropy solution $M$ remains piecewise constant for all time, and the locations of its discontinuities evolve according to the \textit{sticky particle Cucker--Smale} dynamics, which we introduce presently.
	
	\subsubsection{Sticky particle Cucker--Smale (SPCS) dynamics: definition and basic properties}
	
	\label{sss:SPCS1}
	
	We characterize the sticky particle Cucker--Smale dynamics through  Definition \ref{d:SP} below.  The unidirectionality of the velocities allows us to write the equations of motion (between collisions) as a first-order system, see part (c) of the definition, as well as Remark \ref{r:SPCS}. After this, we include a few  properties of the SPCS dynamics that will be used crucially in our proofs of Theorems \ref{t:SPCSentropy} and \ref{t:sliceexist}.
	
	\begin{DEF} 
		\label{d:SP}
		Assume $\Phi:\mathbb{R}\times \mathbb{R}^{d-1}\to \mathbb{R}$ is measurable and  satisfies \eqref{e:Phiodd} and \eqref{e:Phinondecr}.
		Let 
		\begin{equation} 
			\label{e:SPIC}
			\{\bm{x}_{i,\ell}^0 = (x_{i,\ell}^0, y_\ell)\in \mathbb{R}\times \mathbb{R}^{d-1},i\in [I_\ell], \ell \in [L]\},
			\qquad 
			\{\bm{v}_{i,\ell}^0 = (v_{i,\ell}^0,0)\in \mathbb{R}\times \mathbb{R}^{d-1},i\in [I_\ell], \ell \in [L]\}
		\end{equation} 
		denote initial positions and unidirectional velocities of $\sum_{\ell=1}^L I_\ell$ agents distributed among $L$ slices $\mathbb{R}\times \{y_\ell\}$ ($\ell\in [L]$), with 
		\begin{equation} 
			\label{e:SPICorder}
			x_{1,\ell}^0\le x_{2,\ell}^0 \le \cdots \le x_{I_\ell, \ell}^0,\qquad \ell\in [L].
		\end{equation} 
		
		For each $i$ and $\ell$, let $m_{i,\ell}>0$ denote the mass of agent $(i,\ell)$, and assume that $\sum_{\ell=1}^L \sum_{i=1}^{I_\ell} m_{i,\ell} = 1$.  
		We associate to each agent a right-continuous velocity $t\mapsto v_{i,\ell}(t)\in \mathbb{R}$ and a trajectory $t\mapsto x_{i,\ell}(t) = x_{i,\ell}^0 + \int_0^t v_{i,\ell}(\tau)\,\dd\tau$ along the slice $\mathbb{R}\times \{y_\ell\}$.   
		Denote 
		\[
		\bm{x}_{i,\ell}(t) = (x_{i,\ell}(t), y_\ell),
		\qquad \bm{v}_{i,\ell}(t) = (v_{i,\ell}(t), 0).
		\]
		We say that $(\bm{x}_{i,\ell}(t), \bm{v}_{i,\ell}(t))_{i\in [I_\ell], \ell\in [L]}$ satisfies the \textit{Sticky Particle Cucker--Smale} dynamics (SPCS) associated to the initial configuration $(m_{i,\ell}, \bm{x}_{i,\ell}^0, \bm{v}_{i,\ell}^0)_{i\in [I_\ell], \ell\in [L]}$ if the following hold for all $i\in [I_\ell]$, $\ell\in [L]$:
		\begin{enumerate}[label = (\alph*)]
			\item (Stickiness of the dynamics) For any time $t\ge 0$ we have 
			\begin{equation} 
				\label{e:Jnested}
				\cJ_{i,\ell}(s)\subseteq \cJ_{i,\ell}(t),
				\qquad \text{ whenever } 0\le s \le t,
			\end{equation} 
			where $\cJ_{i,\ell}(\tau)$ denotes the indices of all agents that have the same location as $\bm{x}_{i,\ell}(\tau) = (x_{i,\ell}(\tau), y_\ell)$ at time $\tau$:
			\begin{equation} 
				\label{e:Jnotation} \mathcal{J}_{i,\ell}(\tau) := \{j: x_{j,\ell}(\tau) = x_{i,\ell}(\tau)\}.
			\end{equation}
			\item (Conservation of momentum across collisions) We have 
			\begin{equation}
				\label{e:collisioncons}
				\sum_{j\in \cJ_{i,\ell}(t)} m_{j,\ell} v_{j,\ell}(t) = \sum_{j\in \cJ_{i,\ell}(t)} m_{j,\ell} v_{j,\ell}(t-),
				\qquad t>0, \,i\in [I_\ell],\,\ell\in [L].
			\end{equation}
			\item (Cucker--Smale dynamics at noncollisional times) Let us refer to a \textit{collision time} as a time $t$ where some $\cJ_{i,\ell}(t)$ increases in cardinality, and a time interval $(s,t)$ as \textit{noncollisional} if it contains no collision times. On any noncollisional time interval $(s,t)$, we have 
			\begin{equation}
				\label{e:CSuni}
				\dot{x}_{i,\ell}(\tau) = v_{i,\ell}(\tau) = \psi_{i,\ell}(s) - \sum_{p = 1}^L \sum_{j=1}^{I_p} m_{j,p} \Phi(x_{i,\ell}(\tau) - x_{j,p}(\tau), y_\ell - y_p),\quad \tau\in (s,t).
			\end{equation}
			where $\psi_{i,\ell}$ is the discrete analog of the quantity $\psi$ from \eqref{e:EAuni},  defined via 
			\begin{equation}
				\label{e:psiielldef}
				\psi_{i,\ell}(\tau) := v_{i,\ell}(\tau) + \sum_{p=1}^L \sum_{j=1}^{I_p} m_{j,p} \Phi(x_{i,\ell}(\tau) - x_{j,p}(\tau), y_\ell - y_p),
				\qquad \tau\ge 0.
			\end{equation}
			\item (Initial Conditions) Denoting $\cJ_{i,\ell}^0 = \{j: x_{j,\ell}^0 = x_{i,\ell}^0\}$, we have 
			\[
			\sum_{j\in \cJ_{i,\ell}(0)} m_{j,\ell} v_{j,\ell}(0) = \sum_{j\in \cJ_{i,\ell}^0} m_{j,\ell} v_{j,\ell}^0,\qquad i\in [I_\ell], \,\ell\in [L].
			\]
		\end{enumerate}
	\end{DEF}

	\begin{REMARK} 
		\label{r:SPCS}
		Rule (a) guarantees that only finitely many collisions may occur, so that \eqref{e:CSuni} holds except at these finitely many collision times. Differentiating \eqref{e:CSuni} in $\tau$ recovers the standard form of the Cucker--Smale velocity equation (assuming unidirectionality). 
		Finally, we make frequent use below of the following relationship: 
		\begin{equation}
			\label{e:collisioncons2}
			v_{i,\ell}(t) = \frac{\sum_{j\in \cJ_{i,\ell}(t)} m_{j,\ell} v_{j,\ell}(t)}{\sum_{j\in \cJ_{i,\ell}(t)} m_{j,\ell}} = \frac{\sum_{j\in \cJ_{i,\ell}(t)} m_{j,\ell} v_{j,\ell}(t-)}{\sum_{j\in \cJ_{i,\ell}(t)} m_{j,\ell}},\qquad t\ge 0,\,i\in [I_\ell], \,\ell\in [L],
		\end{equation}
		under the convention $v_{i,\ell}(0-) = v_{i,\ell}^0$.  This is seen to be valid for $t>0$ by combining (a) and (b), and at $t=0$ by (a) and (d). 
	\end{REMARK}

	\begin{PROP}[Maximum Principle]
		\label{prop:SPCSmaxprinc}
		Using the notation above, the unidirectional sticky particle Cucker--Smale dynamics $(x_{i,\ell}(t), v_{i,\ell}(t))_{i\in [I_\ell], \ell\in [L]}$ associated to the initial configuration $(m_{i,\ell},\bm{x}_{i,\ell}^0, \bm{v}_{i,\ell}^0)_{i\in [I_\ell], \ell\in [L]}$, with $\bm{x}_{i,\ell}^0 = (x_{i,\ell}^0, y_\ell)$, $\bm{v}_{i,\ell}^0 = (v_{i,\ell}^0, 0)$, satisfies
		\begin{equation}
			\min_{j,p} v_{j,p}(s) \le v_{i,\ell}(t)\le \max_{j,p} v_{j,p}(s),
			\qquad 0\le s \le t,
		\end{equation}
		\begin{equation}
			|x_{i,\ell}(t)| \le \max_{j,p} |x_{j,p}^0| + t \max_{j', p'} |v_{j', p'}^0|,
			\qquad t\ge 0.
		\end{equation}
	\end{PROP}
	We omit the proof of this proposition; it is a well-known feature of the Cucker--Smale system, and the presence of collisions can only tighten the bounds in question.

	\begin{PROP}[Properties of $\psi_{i,\ell}$] The following statements hold.
		\begin{enumerate}[label = (\alph*)]
			\item $\psi_{i,\ell}$ satisfies the  following conservation property:
			\begin{equation}
				\label{e:collisionconspsi}
				\psi_{i,\ell}(t) = \frac{\sum_{j\in \cJ_{i,\ell}(t)} m_{j,\ell} \psi_{j,\ell}(t)}{\sum_{j\in \cJ_{i,\ell}(t)} m_{j,\ell}} = \frac{\sum_{j\in \cJ_{i,\ell}(t)} m_{j,\ell} \psi_{j,\ell}(s)}{\sum_{j\in \cJ_{i,\ell}(t)} m_{j,\ell}},
				\qquad 0\le s\le t.
			\end{equation}
			\item Denote 
			\begin{equation} 
				\label{e:istardef}
				i^*(t,\ell) = \max \cJ_{i,\ell}(t),
				\qquad 
				i_*(t,\ell) = \min \cJ_{i,\ell}(t).
			\end{equation} 
			Then for any $k\in \cJ_{i,\ell}(t)$, we have 
			\begin{equation}
				\label{e:psibary}
				\frac{\sum_{j=i_*(t,\ell)}^k m_{j,\ell} \psi_{j,\ell}(s)}{\sum_{j=i_*(t,\ell)}^k m_{j,\ell}} 
				\ge \psi_{i,\ell}(t) 
				\ge \frac{\sum_{j=k}^{i^*(t,\ell)} m_{j,\ell} \psi_{j,\ell}(s)}{\sum_{j=k}^{i^*(t,\ell)} m_{j,\ell}},
				\qquad 0\le s \le t.
			\end{equation}
		\end{enumerate}
	\end{PROP}
	
	\begin{proof}
		When $t=0$, both claims follow from \eqref{e:collisioncons2} (and the inequalities in \eqref{e:psibary} are equalities); we argue for $t>0$.  Requirements \eqref{e:Jnested} and \eqref{e:collisioncons} can be combined to conclude that 
		\begin{equation} 
			\label{e:collisioncons0st}
			\sum_{j\in \cJ_{i,\ell}(t)} m_{j,\ell} v_{j,\ell}(s) = \sum_{j\in \cJ_{i,\ell}(t)} m_{j,\ell} v_{j,\ell}(s-),
			\qquad 0< s\le t.
		\end{equation} 
		Since the trajectories $t\mapsto x_{i,\ell}(t)$ are continuous in time and since $\Phi$ is continuous in its first argument, it is clear from the definition of $\psi_{i,\ell}$ that these quantities share the property analogous to \eqref{e:collisioncons0st} for positive time:
		\begin{equation}
			\label{e:collisioncons0stpsi}
			\sum_{j\in \cJ_{i,\ell}(t)} m_{j,\ell} \psi_{j,\ell}(s) = \sum_{j\in \cJ_{i,\ell}(t)} m_{j,\ell} \psi_{j,\ell}(s-),
			\qquad 0< s\le t.
		\end{equation} 
		Then, combining \eqref{e:collisioncons0stpsi} with the fact that each $\psi_{i,\ell}$ is constant on noncollisional time intervals,  \eqref{e:collisionconspsi} follows.
		
		Next, we note the following basic property: If a collection of agents is stuck together at some time $t>0$ (either because a collision occurred at time $t$ or because the agents were already stuck together at some previous time), then their velocities must be ordered accordingly:
		\begin{equation}
			v_{i_*(t,\ell),\ell}(t-)\ge v_{i_*(t,\ell)+1,\ell}(t-) \ge \cdots \ge v_{i^*(t,\ell),\ell}(t-),
			\qquad t > 0.
		\end{equation}
		Leveraging again the continuity of $\Phi$ in the first argument, we obtain the analogous property for the $\psi_{i,\ell}$'s:
		\begin{equation}
			\label{e:psicollisionorder}
			\psi_{i_*(t,\ell),\ell}(t-)\ge \psi_{i_*(t,\ell)+1,\ell}(t-) \ge \cdots \ge \psi_{i^*(t,\ell),\ell}(t-),
			\qquad t> 0.
		\end{equation}
		Combining \eqref{e:psicollisionorder} with \eqref{e:collisionconspsi} immediately yields 
		\[
		\frac{\sum_{j=i_*(t,\ell)}^k m_{j,\ell} \psi_{j,\ell}(t-)}{\sum_{j=i_*(t,\ell)}^k m_{j,\ell}} 
		\ge \psi_{i,\ell}(t) 
		\ge \frac{\sum_{j=k}^{i^*(t,\ell)} m_{j,\ell} \psi_{j,\ell}(t-)}{\sum_{j=k}^{i^*(t,\ell)} m_{j,\ell}}, \qquad t> 0,
		\]
		which is the special case of \eqref{e:psibary} where $s=t-$.  But \eqref{e:psicollisionorder} also implies that for $k\in \cJ_{i,\ell}(t)$, the functions
		\[
		s\mapsto \frac{\sum_{j=i_*(t,\ell)}^k m_{j,\ell} \psi_{j,\ell}(s)}{\sum_{j=i_*(t,\ell)}^k m_{j,\ell}},
		\qquad 
		s\mapsto \frac{\sum_{j=k}^{i^*(t,\ell)} m_{j,\ell} \psi_{j,\ell}(s)}{\sum_{j=k}^{i^*(t,\ell)} m_{j,\ell}},
		\]
		are nonincreasing and nondecreasing, respectively, on the interval $s\in [0,t]$.  Combining with the previous step then immediately implies \eqref{e:psibary}.
	\end{proof}

	\subsubsection{Entropy solutions generated by SPCS dynamics}
	
	Equipped with the tools developed above, we can now complete our verification that unidirectional SPCS dynamics generate entropy solutions of \eqref{e:ScalarBalance} and make the relationship  precise. 
	\begin{THEOREM}
		\label{t:SPCSentropy}
		Assume $\Phi$ satisfies \eqref{e:Phiodd}--\eqref{e:Phibddcptsets}.  Fix an initial configuration of masses $m_{i,\ell}$, positions $\bm{x}_{i,\ell}^0 = (x_{i,\ell}^0, y_\ell)$, and velocities $\bm{v}_{i,\ell}^0 = (v_{i,\ell}^0, 0)$ as in Definition \ref{d:SP}, with $i\in [I_\ell]$, $\ell\in [L]$.  Assume as usual that $m_{i,\ell}>0$ and $\sum_{\ell=1}^L \sum_{i=1}^{I_\ell} m_{i,\ell} = 1$.  Assume $(x_{i,\ell}(t), v_{i,\ell}(t))$ satisfy the sticky particle Cucker--Smale dynamics associated to this initial data.  For each $\ell\in [L]$, define 
		\begin{equation}
			\label{e:thetaielldef}
			\mu_\ell = \sum_{i=1}^{I_\ell} m_{i,\ell},
			\qquad\qquad \theta_{i,\ell} = \frac{1}{\mu_\ell}\sum_{j=1}^i m_{j,\ell},\quad i\in [I_\ell], 
			\qquad \qquad \theta_{0,\ell} = 0.
		\end{equation}
		Then 
		\begin{equation}
			\label{e:MSPCS}
			M(x,y_\ell,t) = \frac{1}{\mu_{\ell}}\sum_{i=1}^{I_\ell} m_{i,\ell} H(x - x_{i,\ell}(t))
		\end{equation}
		defines the unique entropy solution of \eqref{e:ScalarBalance} associated to the admissible triple $(M^0, A, \nu)$, where 
		\begin{equation} 
			\label{e:nuSPCS}
			\nu = \sum_{\ell=1}^L \mu_\ell \delta(\cdot - y_\ell),
		\end{equation} 
		\begin{equation}
			\label{e:M0SPCS}
			M^0(x,y_\ell) = \frac{1}{\mu_{\ell}}\sum_{i=1}^{I_\ell} m_{i,\ell} H(x - x^0_{i,\ell}),    
		\end{equation}
		and $A(\cdot, y_\ell)$ is the piecewise linear function determined by
		\begin{equation}
			\label{e:ASPCS}
			A(\theta_{i,\ell},y_\ell) = \frac{1}{\mu_{\ell}}\sum_{j=1}^i m_{j,\ell} \psi^0_{j,\ell},
			\qquad \qquad
			\partial_m A(\theta, y_\ell) = \psi^0_{i,\ell} \quad \text{ for } \quad \theta\in (\theta_{i-1, \ell}, \theta_{i,\ell}),   
		\end{equation}
		\[
		\psi_{i,\ell}^0 := \frac{\sum_{j\in \cJ_{i,\ell}^0} m_{j,\ell} v_{j,\ell}^0}{\sum_{j\in \cJ_{i,\ell}^0} m_{j,\ell}} + \sum_{p=1}^L \sum_{j=1}^{I_p} m_{j,p} \Phi(x_{i,\ell}^0 - x_{j,p}^0, y_\ell - y_p).
		\]
	\end{THEOREM}
	\begin{REMARK}
		Note that, under the notation introduced in Definition \ref{d:SP}, we have $\psi^0_{i,\ell} = \psi_{i,\ell}(0)$, by \eqref{e:collisioncons2}.
	\end{REMARK}
	\begin{proof}[Proof of Theorem \ref{t:SPCSentropy}]
		It suffices to check \eqref{e:RH} and \eqref{e:Oleinik}, with 
		\[
		m_- = M(x_{i,\ell}(t)-, y_\ell,t) = \theta_{i_*(t,\ell)-1,\ell},
		\qquad 
		m_+ = M(x_{i,\ell}(t)+, y_\ell,t) = \theta_{i^*(t,\ell),\ell}.
		\]    
		
		For the first verification, we note that
		\begin{align*} 
			[[A(M(\cdot, y_\ell, t), y_\ell)]] & = A(\theta_{i^*(t,\ell),\ell}, y_\ell) - A(\theta_{i_*(t,\ell)-1,\ell}, y_\ell) 
			= \frac{1}{\mu_\ell}\sum_{j\in \cJ_{i,\ell}(t)} m_{j,\ell} \psi^0_{j,\ell} \\
			& = \frac{1}{\mu_\ell}\bigg( \sum_{j\in \cJ_{i,\ell}(t)} m_{j,\ell}\bigg) \psi_{i,\ell}(t) \\
			& = [[M(\cdot, y_\ell, t)]]\cdot \big( \dot{x}_{i,\ell}(t) + B[\nu,M](x_{i,\ell}(t), y_\ell, t) \big).
		\end{align*} 
		In passing to the second line, we have used \eqref{e:collisionconspsi}.  To check the Oleinik condition, we argue similarly, using \eqref{e:psibary}. 
		Note that piecewise linearity of $A(\cdot, y_\ell)$ means that we need only check \eqref{e:Oleinik} at the breakpoints $\theta_{k,\ell}$, with \mbox{$k\in \cJ_{i,\ell}(t)\backslash \{i_*(t,\ell)-1, i^*(t,\ell)\}$}.  To this end, we have 
		\[
		\frac{A(\theta_{k,\ell}, y_\ell)- A(\theta_{i_*(t,\ell)-1,\ell}, y_\ell)}{\theta_{k,\ell}- \theta_{i_*(t,\ell)-1,\ell}}
		= \frac{\sum_{j=i_*(t,\ell)}^k m_{j,\ell} \psi^0_{j,\ell}}{\sum_{j=i_*(t,\ell)}^k m_{j,\ell}} 
		\ge \psi_{i,\ell}(t) = \dot{x}_{i,\ell}(t) + B[\nu,M](x_{i,\ell}(t), y_\ell, t). 
		\]
	\end{proof}
	With Theorem \ref{t:SPCSentropy} in hand, we set the following terminology: 
	\begin{DEF}
		\label{def:SPCSentropysoln}
		We say that $(M^0, A, \nu)$ is an \textit{SPCS admissible triple} if $\nu$ takes the form \eqref{e:nuSPCS}, $M^0$ takes the form \eqref{e:M0SPCS}, and $A$ takes the form \eqref{e:ASPCS}.  An \textit{SPCS entropy solution} of \eqref{e:ScalarBalance} is an entropy solution $M$ of the form \eqref{e:MSPCS} associated to an SPCS admissible triple $(M^0, A, \nu)$.
	\end{DEF}

	\subsection{Entropy solutions supported on finitely many slices}
	
	\label{ss:Existslice}
	
	The following Theorem establishes the existence of entropy solutions when the associated transverse measure $\nu$ is atomic.  It also clarifies the sense in which such a solution can be approximated by SPCS dynamics.  
	
	\begin{THEOREM}
		\label{t:sliceexist}
		Assume $\Phi$ is measurable and satisfies \eqref{e:Phiodd}--\eqref{e:Phibddcptsets}. 
		Fix $T>0$.  Let $(M^0, A, \nu)$ be an admissible triple such that $\nu$ is atomic: $\nu = \sum_{\ell=1}^L \mu_\ell \delta(\cdot - y_\ell)$, with $\mu_\ell>0$ for each $\ell$.  Then there exists a unique entropy solution $M$ of \eqref{e:ScalarBalance} on $[0,T)$ associated to this triple, and $M(\cdot, y_\ell, \cdot)\in BV(\mathbb{R}\times (0,T))$ for each $\ell\in [L]$. Moreover, there exists a sequence $(M_N)_{N=1}^\infty$ of SPCS entropy solutions (with associated admissible triples $(M^0_N, A_N, \nu)$) such that 
		\begin{equation}
			\label{e:MNtoMdisc}
			(M_N - M)(\cdot, y_\ell, \cdot)\to 0 \quad \text{ in } \; C([0,T);L^1(\mathbb{R})),\qquad \text{ for all } \ell\in [L].
		\end{equation}
		For each $M_N$, the number $R(T)$ in the definition of entropy solution can be taken as 
		\begin{equation}
			\label{e:RTslice}
			R(T) = R^0 +1 + T \Big( \max_\ell [A(\cdot, y_\ell)]_{\Lip} + \sup_{|y|\le R(\nu)} \Phi(2R^0+1, y)\Big),
		\end{equation} 
		where $R^0>0$ is chosen large enough so that $\{x:M^0(x, y_\ell)\notin \{0,1\}\}\subset [-R^0,R^0]$ for each $\ell\in [L]$, and $R(\nu) = \diam(\supp(\nu))$. Finally,  
		\begin{equation}
			\label{e:ptMNtoptMdisc}
			\partial_t M_N(\cdot, y_\ell, t) \stackrel{*}{\rightharpoonup} \partial_t M(\cdot, y_\ell, t) \qquad \text{ in } \cM(\mathbb{R}),
			\qquad \text{ for all } \ell\in [L],\;t\in (0,T),
		\end{equation}
		where the measure $\partial_t M(\cdot, y_\ell, t)\in \cM(\mathbb{R})$ is understood as $-\partial_x (A(M(\cdot, y_\ell,t),y_\ell)) + B[\nu, M]\partial_x M(\cdot, y_\ell, t)$, and similarly for $\partial_t M_N(\cdot, y_\ell, t)$.
	\end{THEOREM}
	
	\begin{proof}
		We begin by constructing discretized initial data and flux $(M_N^0, A_N)$, then using SPCS dynamics (and Theorem \ref{t:SPCSentropy}) to obtain the corresponding entropy solution $M_N$.  We then extract a subsequential limit through a compactness argument and show that it satisfies the required entropy conditions, finishing the proof of existence.  The convergence \eqref{e:MNtoMdisc} then follows by uniqueness of solutions, and \eqref{e:ptMNtoptMdisc} is obtained by taking limits in the other terms of the scalar balance law.
		
		For each $N\in \mathbb{N}$, fix $\{m_{i,\ell,N}\}_{i\in [N],\ell\in [L]}$ satisfying the following requirements.  First, $m_{i,\ell,N}>0$ for each $i,\ell$; second, $\sum_{i=1}^N m_{i,\ell,N} = \mu_{\ell}$ for each $\ell,N$; and finally, $\max_{i,\ell}  \big(\frac{m_{i,\ell, N}}{\mu_\ell}\big) \stackrel{N\to \infty}{\longrightarrow} 0$.  (For example, we may take $m_{i,\ell,N} = \frac{\mu_\ell}{N}$ for each $i,\ell,N$, but the construction allows for more flexibility than this.) Define $\theta_{i,\ell,N}$ as in \eqref{e:thetaielldef} (with $I_\ell = N$ for each $\ell$).  Define 
		\begin{equation} 
			x_{i,\ell,N}^0 = \inf\{x: M^0(x,y_\ell) \ge \theta_{i,\ell,N}\} + \frac{i}{N^2}, \qquad i\in [N],\;\; \ell\in [L], \;\; N\in \mathbb{N}.
		\end{equation} 
		(The shift of $\frac{i}{N^2}$ is introduced to ensure that $x^0_{i,\ell,N}$, $i\in [N]$, are distinct for fixed $\ell$.)
		Define $M_N^0$ via $M_N^0(x,y_\ell) = \frac{1}{\mu_{\ell}}\sum_{i=1}^N m_{i,\ell,N} H(x - x_{i,\ell,N}^0)$ and let $A_N(\cdot, y_\ell)$ be the piecewise linear function satisfying  $A_N(\theta_{i,\ell,N}, y_\ell) = A(\theta_{i,\ell,N}, y_\ell)$ for $0\le i\le N$.  For use below, note that these definitions readily imply that
		\begin{equation} 
			\label{e:MN0conv}
			\max_{\ell\in [L]}\|M_N^0(\cdot, y_\ell)-M^0(\cdot, y_\ell)\|_{L^1(\mathbb{R})}\le 2R^0 \max_{i\in [N], \ell\in [L]} \Big(\frac{m_{i,\ell,N}}{\mu_\ell}\Big) + \frac{1}{N}\stackrel{N\to \infty}{\longrightarrow} 0,
		\end{equation}
		\begin{equation} 
			\label{e:ANconv}
			\max_{\ell\in [L]}\|A_N(\cdot, y_\ell) - A(\cdot, y_\ell)\|_{L^\infty([0,1])}\le \Big(\max_{\ell\in [L]}[A(\cdot, y_\ell)]_{\Lip([0,1])}\Big)\cdot \max_{i\in [N], \ell\in [L]} \Big(\frac{m_{i,\ell,N}}{\mu_\ell}\Big) \stackrel{N\to \infty}{\longrightarrow} 0.
		\end{equation} Finally, define $\psi_{i,\ell,N}^0$ and then $v_{i,\ell,N}^0$ via \[
		m_{i,\ell,N} \psi_{i,\ell,N}^0 = \mu_\ell\big( A(\theta_{i,\ell, N},y_\ell) - A(\theta_{i-1,\ell, N},y_\ell) \big), 
		\quad \;
		v_{i,\ell,N}^0 = \psi_{i,\ell,N}^0 - \sum_{p=1}^L \sum_{j=1}^N  m_{j,p,N} \Phi(x_{i,\ell,N}^0 - x_{j,p,N}^0, y_\ell - y_p).
		\]
		Starting from the initial conditions $((m_{i,\ell,N}, x_{i,\ell,N}^0, v_{i,\ell,N}^0))_{i\in [N], \ell\in [L]}$, let $((x_{i,\ell,N}(t), v_{i,\ell,N}(t)))_{i\in [N], \ell\in [L]}$ be the associated SPCS dynamics. Then Theorem \ref{t:SPCSentropy} fashions a unique entropy solution associated to $(M_N^0, A_N, \nu)$ given by 
		\[
		M_N(x,y_\ell, t) = \frac{1}{\mu_\ell}\sum_{i=1}^N m_{i,\ell,N} H(x - x_{i,\ell,N}(t)).
		\]
		Next, we note that 
		\[ 
		-R^0\le x_{i,\ell,N}^0\le R^0+1,\qquad \qquad |v_{i,\ell,N}^0| \le \max_\ell [A(\cdot, y_\ell)]_{\Lip} + \sup_{|y|\le R(\nu)} \Phi(2R^0+1,y),
		\]
		and therefore, by Proposition \ref{prop:SPCSmaxprinc}, we have for all $t\ge 0$ that 
		\begin{equation} 
			\label{e:SPCSmaxprincappv}
			|v_{i,\ell,N}(t)|\le \max_\ell [A(\cdot, y_\ell)]_{\Lip} + \sup_{|y|\le R(\nu)} \Phi(2R^0+1,y),
		\end{equation} 
		\begin{equation} 
			\label{e:SPCSmaxprincappx}
			\qquad |x_{i,\ell,N}(t)|\le R(T): = R^0 + 1 + T\Big( \max_\ell [A(\cdot, y_\ell)]_{\Lip} + \sup_{|y|\le R(\nu)} \Phi(2R^0+1,y) \Big).
		\end{equation} 
		Using \eqref{e:SPCSmaxprincappv} and the formula for $M_N(t)$, one can conclude that 
		\[
		\max_\ell \|M_N(\cdot, y_\ell, t) - M_N(\cdot, y_\ell, s)\|_{L^1(\mathbb{R})} \le (t-s) \, \Big( \max_\ell [A(\cdot, y_\ell)]_{\Lip} + \sup_{|y|\le R(\nu)} \Phi(2R^0+1,y) \Big).
		\]
		A standard compactness and diagonal argument then fashions us with a subsequence (which we continue to denote by $(M_{N}(\cdot, y_\ell, \cdot))_{N=1}^\infty$) and a function $M(\cdot, y_\ell, \cdot) \in C([0,T),L^1(\mathbb{R}))\cap BV(\mathbb{R}\times (0,T))$, such that $x\mapsto M(\cdot, y_\ell, t)$ is nondecreasing for each $t$, such that $M(x,y_\ell,t) = 0$ for $x<-R(T)$ and $M(x,y_\ell,t) = 1$ for $x>R(T)$, and such that 
		\begin{equation}
			\label{e:MNconv}
			M_{N}(\cdot, y_\ell, \cdot) - M(\cdot, y_\ell, \cdot)\stackrel{N\to \infty}{\longrightarrow}0, \qquad \text{ in }C([0,T),L^1(\mathbb{R})).    
		\end{equation}  We may assume without loss of generality that the same subsequence can be used for each $\ell\in [L]$.  Note that \eqref{e:MN0conv} and \eqref{e:MNconv} together imply that $M(\cdot, y_\ell, 0+) = M^0(\cdot, y_\ell)$, for each $\ell$.   
		
		It remains to show that the entropy condition \eqref{e:entropyintegraly} satisfied by each $M_N$ is preserved for the limit $M$. Now, for each $\ell$, 
		\[
		\int_0^T \int_{\mathbb{R}} |M_N(x, y_\ell, t) - \alpha|\,\partial_t \widetilde{\zeta}\,\dx\,\dt \stackrel{N\to \infty}{\longrightarrow} \int_0^T \int_{\mathbb{R}} |M(x, y_\ell, t) - \alpha|\,\partial_t \widetilde{\zeta}\,\dx\,\dt.
		\]
		Recalling that $m\mapsto \sgn(m-\alpha)(A(m) - A(\alpha))$ is Lipschitz whenever $A$ is, with the same Lipschitz constant as $A$, we may also conclude from \eqref{e:ANconv} and \eqref{e:MNconv} that
		\begin{small} 
		\[
		\int_0^T \!\int_{\mathbb{R}} \sgn(M_N - \alpha)\big( A_N(M_N, y_\ell) - A_N(\alpha, y_\ell) \big)\,\partial_x \widetilde{\zeta}\,\dx\,\dt \stackrel{N\to \infty}{\longrightarrow} 
		\int_0^T \!\int_{\mathbb{R}} \sgn(M - \alpha)\big( A(M, y_\ell) - A(\alpha, y_\ell) \big)\,\partial_x \widetilde{\zeta}\,\dx\,\dt.
		\]
		\end{small} 
		Our established convergence of $M_N\to M$ along a subsequence is more than enough to apply Lemma \ref{l:NLprelim} to the nonlocal term at fixed time.  Combining it with an application of the dominated convergence theorem in time, we have
		\[
		\int_0^T \int_{\mathbb{R}} \widetilde{\zeta} B[\nu, M_N] \partial_x |M_N - \alpha|(x,y_\ell,t) \,\dx \,\dt \stackrel{N\to \infty}{\longrightarrow} \int_0^T \int_{\mathbb{R}} \widetilde{\zeta} B[\nu, M] \partial_x |M - \alpha|(x,y_\ell,t) \, \dx\,\dt.
		\]
		We conclude that $M$ satisfies \eqref{e:entropyintegraly}; therefore, by the uniqueness of entropy solutions provided by Theorem \ref{t:Mstab}, the convergence \eqref{e:MNtoMdisc} holds along the \textit{entire} sequence $(M_N)_{N=1}^\infty$, not just the subsequence we worked with to establish \eqref{e:entropyintegraly}. 
		
		It remains to treat the convergence \eqref{e:ptMNtoptMdisc}; we have already done most of the work.  The entropy inequalities satisfied by $M_N$ and $M$ imply that, for fixed $t\in (0,T)$, we have (as equalities in $\cM(\mathbb{R})$)
		\begin{equation}
			\label{e:ptMNdist}
			\partial_t M_N(\cdot, y_\ell, t) = -\partial_x (A_N(M_N(\cdot, y_\ell, t), y_\ell)) + B[\nu, M_N]\,\partial_x M_N(\cdot, y_\ell, t),
			\qquad \ell\in [L],\,N\in \N,
		\end{equation}
		\begin{equation}
			\label{e:ptMdist}
			\partial_t M(\cdot, y_\ell, t) = -\partial_x (A(M(\cdot, y_\ell, t), y_\ell)) + B[\nu, M]\,\partial_x M(\cdot, y_\ell, t),
			\qquad \ell\in [L].
		\end{equation}
		The convergences \eqref{e:ANconv} and \eqref{e:MNconv} together imply that $\partial_x (A_N(M_N(\cdot, y_\ell, t), y_\ell))\stackrel{*}{\rightharpoonup}\partial_x (A(M(\cdot, y_\ell, t), y_\ell))$ weak-$*$ in $\cM(\mathbb{R})$, and Lemma \ref{l:NLprelim} guarantees $B[\nu, M_N]\,\partial_x M_N(\cdot, y_\ell, t)\stackrel{*}{\rightharpoonup}B[\nu, M]\,\partial_x M(\cdot, y_\ell, t)$.  Combining these convergences with \eqref{e:ptMNdist} and \eqref{e:ptMdist} immediately yields \eqref{e:ptMNtoptMdisc}.
	\end{proof}

	\subsection{Construction and approximation of general entropy solutions}
	\label{ss:Existgen}
	
	We are now ready to construct the entropy solution associated to an arbitrary admissible triple $(M^0, A, \nu)$. We begin by specifying our scheme for discretizing in $y$ in Section~\ref{sss:transversdiscretization};  then, Theorem \ref{t:sliceexist}, which we have just established, provides a sequence $(M^k)_{k=1}^\infty$ of entropy solutions for the discrete-in-$y$ systems.  Theorem \ref{t:existencefinalstep} in Section \ref{sss:ExistenceConstruction} uses these two ingredients and takes limits to generate an entropy solution for the original admissible triple $(M^0, A, \nu)$.  Two natural subsequential limits arise (denoted $\overline{M}$ and $M_y$ below), and one of the main steps in the proof is to show that they represent the same object.  With existence and uniqueness then in hand, Theorem \ref{t:SPCSapproxM} shows that any entropy solution can be approximated by fully discrete-in-space SPCS entropy solutions.

	\subsubsection{Transverse discretization scheme}
	
	\label{sss:transversdiscretization}
	
	The following lemma records the dyadic discretization  that we use throughout the remainder of this section (and again in Section \ref{s:Rates} when we discuss rates of convergence).
	
	\begin{LEMMA}
		\label{l:ydiscretization}
		Let $(M^0, A, \nu)$ be an admissible triple.  There exists a sequence $(\nu^k)_{k=1}^\infty$ of atomic measures $\nu^k$, each of the form $\nu^k = (\s^k)_{\sharp} \nu$ for some map $\s^k:\mathbb{R}^{d-1}\to \mathbb{R}^{d-1}$ with finite image, and a corresponding sequence $((M^{0,k}, A^k, \nu^k))_{k=1}^\infty$ of admissible triples, such that 
		\begin{equation} 
			\label{e:M0kconv}
			\widehat{M}^{0,k} - M^0 \to 0 \quad \text{ in } L^1(\dx \otimes \nu),
		\end{equation} 
		\begin{equation} 
			\label{e:Akconv}
			\bigl\| \widehat{A}^k(\cdot,y) - A(\cdot,y) \bigr\|_{L^\infty([0,1])} \to 0, \qquad \text{ for }\nu\text{-a.e. } y\in \mathbb{R}^{d-1},
		\end{equation}
		\begin{equation}
			\label{e:nuknu}
			\|\sigma^k-\id_{\mathbb{R}^{d-1}}\|_{L^\infty(\nu)}\le \sqrt{d-1} \cdot 2^{-k},
			\qquad 
			\text{ and therefore }
			\qquad 
			W_1(\nu^k, \nu) \le \sqrt{d-1}\cdot 2^{-k}, 
		\end{equation} 
		where $\widehat{M}^{0,k}:\mathbb{R}\times \mathbb{R}^{d-1}\to \mathbb{R}$ and $\widehat{A}^k:[0,1]\times \mathbb{R}^{d-1}\to \mathbb{R}$ are defined via
		\begin{equation}
			\label{e:M0hatAhatdef}
			\widehat{M}^{0,k}(x,y) = M^{0,k}(x,\s^k(y)),
			\qquad 
			\widehat{A}^k(m,y) = A^k(m,\s^k(y)). 
		\end{equation}
	\end{LEMMA}
	
	\begin{proof} 
		Let $\cQ^k$ denote the collection of all cubes of the form $2^{-k}\mathbf{z} + (0, 2^{-k}]^{d-1}$, where $\mathbf{z}\in \Z^{d-1}$. Each cube in $\cQ^k$ is half-open, has edges parallel to the coordinate axes in $\mathbb{R}^{d-1}$, and has side length $2^{-k}$. For fixed $k$, all the cubes in $\cQ^k$ are disjoint and have union all of $\mathbb{R}^{d-1}$.  Let $\{Q^k_\ell\}_{\ell=1}^{L_k}$ be an enumeration of those cubes in $\cQ^k$ which have strictly positive $\nu$-measure, and let $y^k_\ell$ denote the center of $Q^k_\ell$.  Note that $L_k<+\infty$ for each $k$, since $\nu$ is compactly supported. Define $\s^k:\mathbb{R}^{d-1}\to \mathbb{R}^{d-1}$ by setting $\s^k(y) = y^k_\ell$ whenever $y\in Q^k_\ell$.  (For $y$ not belonging to any $Q^k_\ell$, the value of $\s^k(y)$ is inconsequential; we can set it to $0\in \mathbb{R}^{d-1}$, for instance.)   
		Define $(M^{0,k}, A^k, \nu^k)$ via
		\begin{equation} 
			\label{e:nuk}
			\nu^k:=(\s^k)_\#\nu=\sum_{\ell=1}^{L_k}\nu(Q^k_\ell)\,\delta(\cdot - y^k_\ell),
		\end{equation}
		\begin{equation}\label{e:M0kAkdef}
			M^{0,k}(x,y^k_\ell)
			:=
			\frac{1}{\nu(Q^k_\ell)}
			\int_{Q^k_\ell} M^0(x,\eta)\,\nu(\dd \eta),
			\qquad 
			A^k(m,y^k_\ell)
			:=
			\frac{1}{\nu(Q^k_\ell)}
			\int_{Q^k_\ell} A(m,\eta)\,\nu(\dd \eta).
		\end{equation}
		The inequalities \eqref{e:nuknu} then hold by construction.  Note also that each $A^k(\cdot, y^k_\ell)$ is Lipschitz, with 
		\[
		[A^k(\cdot, y^k_\ell)]_{\Lip}\le \sup_y [A(\cdot, y)]_{\Lip}.
		\]
		The other requirements of Definition \ref{def:SBIC} are also clearly satisfied; it thus remains to justify \eqref{e:M0kconv} and \eqref{e:Akconv}.  We omit the (standard) verification of \eqref{e:M0kconv}; we justify \eqref{e:Akconv} as follows.  For each fixed $m\in [0,1]$, we have $\widehat{A}^k(m,y)\to A(m,y)$ for $\nu$-a.e. $y\in \mathbb{R}^{d-1}$; let $Z_m$ denote the exceptional $\nu$-null set, which may depend on $m$.  Put $Z = \bigcup_{q\in \mathbb{Q}\cap [0,1]} Z_{q}$.  We claim that $\|\widehat{A}^k(\cdot, y)- A(\cdot, y)\|_{L^\infty([0,1])}\to 0$ for all $y\notin Z$.  To this end, assume without loss of generality that $\sup_y[A(\cdot, y)]_{\Lip}>0$;  fix $y\notin Z$, and given $\e>0$, choose $n>4\e^{-1}\sup_y[A(\cdot, y)]_{\Lip}$. Choose $K$ large enough that $|\widehat{A}^k(\frac{i}{n},y)-A(\frac{i}{n},y)|<\frac{\e}{2}$ for $i\in \{1, \ldots, n\}$, whenever $k\ge K$.  Fix $m$, and choose $i$ to minimize $|m-\frac{i}{n}|$.  Then 
		\[
		|\widehat{A}^k(m,y) - A(m,y)|\le 2\sup_y [A(\cdot, y)]_{\Lip} |m - \tfrac{i}{n}| + |\widehat{A}^k(\tfrac{i}{n},y) - A(\tfrac{i}{n},y)|<\e.
		\]
		The convergence \eqref{e:Akconv} follows.
	\end{proof}

	\subsubsection{Construction of solutions}

	\begin{THEOREM}
		\label{t:existencefinalstep}
		Assume $\Phi$ satisfies \eqref{e:Phiodd}--\eqref{e:Phibddcptsets} and \eqref{e:Phiymodulus}--\eqref{e:moc}. Let $(M^0, A, \nu)$ be an admissible triple, and fix $T>0$. Let $((M^{0,k}, A^k, \nu^k))_{k=1}^\infty$ be the sequence of admissible triples fashioned by Lemma \ref{l:ydiscretization} (with corresponding transport map $\s^k:\mathbb{R}^{d-1}\to \mathbb{R}^{d-1})$, which satisfy \eqref{e:M0kconv}, \eqref{e:Akconv}, \eqref{e:nuknu}, \eqref{e:M0hatAhatdef}.  For each $k\in \N$, let $M^k$ denote the unique entropy  solution of \eqref{e:ScalarBalance} on $[0,T)$ associated to $(M^{0,k}, A^k, \nu^k)$.  ($M^k$ exists, by Theorem \ref{t:sliceexist}.)  Then there exists $M\in C([0,T);L^1_{\loc}(\mathbb{R}\times \mathbb{R}^{d-1}, \dx\otimes \nu))$ such that
		\begin{equation} 
			\label{e:MTkconv}
			\sup_{t\in [0,T)}\bigl\| \widehat{M}^{k}(t) - M(t) \bigr\|_{L^1(\mathbb{R}\times \mathbb{R}^{d-1}, \dx \otimes \nu)}\to 0,
		\end{equation}
		where $\widehat{M}^k:\mathbb{R}\times \mathbb{R}^{d-1}\times (0,T)\to \mathbb{R}$ is defined by 
		\begin{equation}
			\label{e:Mhatdef}
			\widehat{M}^{k}(x,y,t) = M^{k}(x,\s^k(y),t),
		\end{equation}
		and $M$ is the unique entropy solution associated to the admissible triple $(M^0, A, \nu)$.  Moreover, the convergence in \eqref{e:MTkconv} holds slicewise:
		\begin{equation}
			\label{e:MTkconvslice}
			\sup_{t\in[0,T)}\|\widehat{M}^{k}(\cdot,y,t) - M(\cdot,y,t)\|_{L^1(\mathbb{R})}\to 0,
			\qquad \text{for }\nu\text{-a.e. } y\in\mathbb{R}^{d-1}.
		\end{equation}
		Finally,  
		for all $t\in (0,T)$, we have 
		\begin{equation}
			\label{e:partialtMkconv}
			\nu^k \partial_t M^k(t) \stackrel{*}{\rightharpoonup} \nu \partial_t M(t),\qquad \text{ in } \cM(\mathbb{R}\times \mathbb{R}^{d-1}),
		\end{equation}
		interpreted as in \eqref{e:ptMNtoptMdisc}.
	\end{THEOREM}

	\begin{proof}
		We begin by recording, and manipulating, the entropy condition satisfied by $M^k$.  Replacing $M$, $A$, and $\nu$ with $M^k$, $A^k$, and $\nu^k$ in \eqref{e:entropyintegraly} yields
		\begin{small} 
		\[
		\begin{split} 
			0 & \le\!\!\int_0^T\!\!\int_{\mathbb{R}^{d-1}}\!\int_{\mathbb{R}}\!
			\Big[
			|M^k\!-\!\alpha|\,\p_t\zeta
			\!+\!\sgn(M^k\!-\!\alpha)\big(A^k(M^k\!,y)\!-\!A^k(\alpha,y)\big)\,\p_x\zeta
			\!+\!\zeta\,B[\nu^k, M^k]\,\p_x|M^k\!-\!\alpha|
			\Big]\dx\,\nu^k(\dy)\,\dt \\
			& = \!\int_0^T\!\!\int_{\mathbb{R}^{d-1}}\!\int_{\mathbb{R}}\!
			\Big[
			|\widehat{M}^k-\alpha|\,\p_t\widehat{\zeta}^k
			+ \sgn(\widehat{M}^k-\alpha)\big(\widehat{A}^k(\widehat{M}^k,y)-\widehat{A}^k(\alpha,y)\big)\,\p_x\widehat{\zeta}^k
			+\widehat{\zeta}^k\,\widehat{B}^k\,\p_x|\widehat{M}^k-\alpha|
			\Big]\dx\,\nu(\dy)\,\dt,
		\end{split} 
		\]
		\end{small} 
		where 
		\begin{equation}
			\label{e:Bkhatzetakhatdef}
			\widehat{B}^k(x,y,t) = 
			B[\nu^k, M^k](x,\s^k(y),t),
			\qquad 
			\widehat{\zeta}^k(x,y,t) = \zeta(x,\s^k(y),t).
		\end{equation} 
		We can rewrite this as a slice-wise statement analogous to \eqref{e:entropyintegraly}: For all $k\in \N$, for $\nu$-a.e. $y\in \mathbb{R}^{d-1}$ and for every nonnegative $\widetilde{\zeta}\in C_c^\infty(\mathbb{R}\times(0,T))$, we have
		\begin{equation}
			\label{e:Mkentropy}
			\int_0^T\int_{\mathbb{R}}
			\Big[
			|\widehat{M}^k-\alpha|\,\p_t\widetilde{\zeta}
			+ \sgn(\widehat{M}^k-\alpha)\big(\widehat{A}^k(\widehat{M}^k,y)-\widehat{A}^k(\alpha,y)\big)\,\p_x\widetilde{\zeta}
			+\widetilde{\zeta}\,\widehat{B}^k\,\p_x|\widehat{M}^k-\alpha|
			\Big]\dx\,\dt\ge0.
		\end{equation}

		Now, we extract a subsequence $(k_i)_{i=1}^\infty$ along which 
		\begin{equation}
			\label{e:MhattoMbar}
			\widehat{M}^{k_i}(t)\stackrel{*}{\rightharpoonup} \overline{M}(t),
			\qquad \text{ weak-$*$ in $L^\infty(\dx\otimes \nu)$, for all $t\in [0,T)$, and}
		\end{equation}
		\begin{equation}
			\label{e:M0yICconv}	
			\widehat{M}^{0,k_i}(\cdot, y) - M^0(\cdot, y)\to 0, \qquad \text{ in $L^1(\mathbb{R})$, for $\nu$-a.e. $y\in \mathbb{R}^{d-1}$.}	
		\end{equation}
		Then, for $\nu$-a.e. $y$, we extract a further subsequence $(k_{i_j}(y))_{j=1}^\infty$ along which \begin{equation} 
			\label{e:MhattoMy}
			\widehat{M}^{k_{i_j}(y)}(\cdot, y, \cdot) - M_y\to 0, 
			\qquad \text{ in }C([0,T),L^1(\mathbb{R}))
		\end{equation}
		for some $M_y\in C([0,T),L^1(\mathbb{R})) \cap BV(\mathbb{R}\times (0,T))$ satisfying $M_y(\cdot,0) = M^0(\cdot, y)$ as elements of $L^1_{\loc}(\mathbb{R})$.  We would like to define our solution $M$ by the formula $M(x,y,t) = M_y(x,t)$, but without additional arguments, we cannot even guarantee that this function is measurable, much less that it satisfies the required entropy conditions.  However, inserting both  \eqref{e:MhattoMbar} and \eqref{e:MhattoMy} into \eqref{e:Mkentropy} and taking limits along the subsequence $(k_{i_j}(y))_{j=1}^\infty$ allows us to conclude (for $\nu$-a.e. $y$) that $M_y$ satisfies
		\begin{equation}
			\label{e:entropyfrozendrift}
			\int_0^T\!\!\int_{\mathbb{R}}
			\Big[
			|M_y\!-\!\alpha|\,\p_t\widetilde{\zeta}
			+ \sgn(M_y\!-\!\alpha)\big(A(M_y,y)\!-\!A(\alpha,y)\big)\,\p_x\widetilde{\zeta}
			+\widetilde{\zeta}\,B[\nu, \overline{M}](x,y,t)\,\p_x|M_y\!-\!\alpha|
			\Big]\dx\,\dt\ge0.
		\end{equation}
		Indeed, convergence of the limits leading to the first two terms in \eqref{e:entropyfrozendrift} poses no difficulty, while convergence of the nonlocal term follows from  Lemma \ref{l:NLprelim} and an application of the dominated convergence theorem in time. The inequality \eqref{e:entropyfrozendrift}
		is the natural entropy condition associated to the scalar balance law
		\begin{equation}
			\label{e:SBfrozen}
			\partial_t M_y + \partial_x A(M_y,y) = B[\nu,\overline{M}](\cdot, y, \cdot)\, \partial_x M_y,
		\end{equation}
		and \eqref{e:entropyfrozendrift} singles out a \textit{unique} solution of \eqref{e:SBfrozen} associated to a  given $M^0$ (see Remark \ref{r:uniquenessSBfrozen}).
		
		The uniqueness of $M_y$ implies that $\widehat{M}^{k_i}(\cdot, y, \cdot) - M_y\to 0$ in $C([0,T),L^1(\mathbb{R}))$ holds along the original subsequence $(k_i)_{i=1}^\infty$ rather than merely along its $y$-dependent further subsequence $(k_{i_j}(y))_{j=1}^\infty$. Now that we know this, we can use the inequality 
		\[
		\sup_{t\in [0,T)}\|\widehat{M}^{k_i}(t) - \widehat{M}^{k_{i+p}}(t)\|_{L^1(\dx\otimes \nu)} \le \int_{\mathbb{R}^{d-1}} \sup_{t\in [0,T)} \|\widehat{M}^{k_i}(\cdot, y, t) - \widehat{M}^{k_{i+p}}(\cdot, y, t) \|_{L^1(\mathbb{R})} \,\nu(\dy), 
		\]
		to conclude that the sequence $(\widehat{M}^{k_i})_{i=1}^\infty$ is Cauchy in $C([0,T);L^1_{\loc}(\dx\otimes \nu))$ and therefore converges in this space to a function $M$.  It is easily seen that this stronger convergence implies that $M = \overline{M}$ and $M(\cdot, y, \cdot) = M_y$ for $\nu$-a.e. $y\in \mathbb{R}^{d-1}$.  The entropy condition associated to \eqref{e:SBfrozen} then becomes exactly the same as \eqref{e:entropyintegraly}, so that we may conclude that $M = \overline{M}$ is the unique entropy solution of \eqref{e:ScalarBalance} associated to the given admissible triple, and furthermore, that all convergences above (including both \eqref{e:MTkconv} and \eqref{e:MTkconvslice}) occur on the entire sequence $(\widehat{M}^k)_{k=1}^\infty$.  This completes the proof, except for the verification of \eqref{e:partialtMkconv}, which we turn to presently; it involves steps similar to our proof of \eqref{e:ptMNtoptMdisc} in Theorem \ref{t:sliceexist}.
		
		The entropy inequality \eqref{e:entropyintegral} implies that the scalar balance law \eqref{e:ScalarBalance} holds in the sense of distributions; therefore, to establish~\eqref{e:partialtMkconv}, it suffices to justify the natural weak-$*$ limits of the terms on the right side of the equation 
		\begin{equation}
			\label{e:ptMkdist}
			\partial_t M^k(t) \nu^k = - \partial_x(A^k(M^k(t),\cdot)) \nu^k + B[\nu^k,M^k]  \partial_x M^k(t) \nu^k.
		\end{equation} 
		Fix $t\in (0,T)$ and $\zeta\in C^\infty_c(\mathbb{R}\times \mathbb{R}^{d-1})$. Then by \eqref{e:Akconv}, \eqref{e:MTkconv}, and the uniform convergence $\zeta(\cdot, \sigma^k(\cdot))\to \zeta$, we have 
		\[
		\int_{\mathbb{R}^{d-1}} \int_{\mathbb{R}} \big[ \partial_x \zeta(x,\sigma^k(y)) \widehat{A}^k(\widehat{M}^k(x,y,t),y) - \partial_x \zeta(x,y) A(M(x,y,t),y) \big] \,\dx\,\nu(\dy) \stackrel{k\to \infty}{\longrightarrow} 0,
		\]
		which implies that $\partial_x(A^k(M^k(t),\cdot)) \nu^k\stackrel{*}{\rightharpoonup} \partial_x (A(M(t),\cdot))\nu$ weak-$*$ in $\cM(\mathbb{R}\times \mathbb{R}^{d-1})$. As for the nonlocal term, Lemma~\ref{l:NLprelim} guarantees that $(\widehat{B}^k \partial_x \widehat{M}^k)(\cdot, y,t)\stackrel{*}{\rightharpoonup} (B[\nu,M]\partial_x M)(\cdot, y,t)$ weak-$*$ in $\cM(\mathbb{R})$, for $\nu$-a.e. $y$.  Combined with the uniform convergence $\zeta(\cdot, \sigma^k(\cdot))\to \zeta$, this implies
		\[
		\int_{\mathbb{R}} \zeta(x,\sigma^k(y)) \widehat{B}^k \,\partial_x \widehat{M}^k(x,y,t)\,\dx 
		\stackrel{k\to \infty}{\longrightarrow} \int_{\mathbb{R}} \zeta(x,y) B[\nu,M]  \,\partial_x M(x,y,t)\,\dx,
		\qquad \text{ for $\nu$-a.e. $y\in \mathbb{R}^{d-1}$.}
		\]
		Integrating in $\nu(\dy)$ and applying the dominated convergence theorem (keeping in mind that $\Phi$ satisfies \eqref{e:Phibddcptsets}) finishes the verification of weak-$*$ convergence of the nonlocal term in $\cM(\mathbb{R}\times \mathbb{R}^{d-1})$ and completes the proof.
	\end{proof}
	
	\begin{REMARK}[Uniqueness of entropy solutions of \eqref{e:SBfrozen}]
		\label{r:uniquenessSBfrozen}
		We briefly discuss the uniqueness of entropy solutions to \eqref{e:SBfrozen}, treating $\overline{M}$ as given. Unlike \eqref{e:ScalarBalance}, equation \eqref{e:SBfrozen} is \textit{not} a system of nonlocally coupled scalar balance laws---it is just one scalar balance law, posed on $\mathbb{R}$, for each $y\in \mathbb{R}^{d-1}$.  Since $\partial_x B[\nu, \overline{M}]$ has a favorable sign, the uniqueness of entropy solutions could arguably be considered standard; however, we can obtain the conclusion quickly as a byproduct of some of the more general analysis we have already carried out, so we sketch the short justification for completeness.  
		We follow the doubling-of-the-variables procedure laid out in Section \ref{ss:stabsetup}, but with $\widetilde{y} = y$ (and, of course, without integrating in $y$---this analysis takes place on a fixed horizontal slice).  We take $\widetilde{A} = A$ and $\widetilde{\nu} = \nu$, and we replace $B[\nu,M]$ and $B[\nu, \widetilde{M}]$ with $B[\nu, \overline{M}]$.  Other appearances of $M$ and $\widetilde{M}$ are replaced by $M_y$ and $\widetilde{M}_y$.  With these modifications, the procedure yields 
		\begin{equation}\label{e:DDeltauniqueness}
			\mathfrak{D}_{\Delta}(t)\le \mathfrak{D}_{\Delta}(s)+\int_s^t \mathfrak{N}_{\Delta}(\tau)\,\dd \tau,
			\qquad 0\le s\le t<T,
		\end{equation}
		where 
		\begin{equation}\label{e:DDeltafrak}
			\mathfrak{D}_{\Delta}(\tau):=
			\int\!\!\!\int\!|M_y(x,\tau)-\widetilde{M}_y(\widetilde{x},\tau)| b_{\Delta}(x-\widetilde x)\,\dd  x\,\dd \widetilde{x},
		\end{equation}
		\begin{equation}
			\label{e:NLfrak}
			\mathfrak{N}_\Delta(\tau) = \int\!\!\!\int
			\Big( B[\nu,\overline{M}](x,y,\tau)\, \partial_x |M_y-\widetilde{M}_y| + 
			B[\nu,\overline{M}](\widetilde{x},y,\tau)\, \partial_{\widetilde{x}} |M_y-\widetilde{M}_y|\Big) \,
			b_\Delta(x-\widetilde{x})
			\,\dd  x\,\dd  \widetilde{x}. 
		\end{equation}
		Making the change of variables \eqref{e:cov} in $\mathfrak{N}_\Delta(\tau)$ and sending $\Delta\to 0+$ yields 
		\[
		\limsup_{\Delta\to 0+} \mathfrak{N}_\Delta \le \int B[\nu, \overline{M}](x,y) \,\partial_x |M_y(x) - \widetilde{M}_y(x)|\,\dx = -\int \partial_1 B[\nu,\overline{M}](x,y)|M_y(x) - \widetilde{M}_y(x)|\,\dx  \le 0, 
		\]
		since $\partial_1 B[\nu, \overline{M}]\ge 0$.  It follows that \begin{equation}
			\|M_y(t) - \widetilde{M}_y(t)\|_{L^1(\mathbb{R})} 
			\le \|M_y(s) - \widetilde{M}_y(s)\|_{L^1(\mathbb{R})},
			\qquad 0\le s\le t<T.
		\end{equation}
		Uniqueness (for fixed initial data) follows immediately.

	\end{REMARK}

	\subsubsection{Approximation by sticky particle dynamics}
	
	The following theorem shows that every entropy solution of \eqref{e:ScalarBalance} arises as a limit of solutions generated by SPCS dynamics; we will translate its conclusion to the unidirectional Euler Alignment system~\eqref{e:EAuni} in Corollary \ref{c:SPCSconvgen} below.
	
	\begin{THEOREM}
		\label{t:SPCSapproxM}
		Assume $\Phi$ satisfies \eqref{e:Phiodd}--\eqref{e:Phibddcptsets} and \eqref{e:Phiymodulus}--\eqref{e:moc}.  Let $(M^0, A, \nu)$ be an admissible triple, and let $M$ be the unique entropy solution of \eqref{e:ScalarBalance} that it generates on the time interval $[0,T)$.  There exists a sequence $(M_n)_{n=1}^\infty$ of SPCS entropy solutions of \eqref{e:ScalarBalance} (in the sense of Definition~\ref{def:SPCSentropysoln}), with each $M_n$ generated by an SPCS admissible triple $(M_n^0, A_n, \nu^n)$, such that $W_1(\nu^n,\nu)\to 0$ as $n\to \infty$, and such that the following holds.  There exists a sequence of maps $(\s^n)_{n=1}^\infty$ on $\mathbb{R}^{d-1}$ such that $\nu^n = (\s^n)_\sharp \nu$ for each $n$, and, writing $\widehat{M}_n(x,y,t) = M_n(x,\s^n(y),t)$, we have the following convergences as $n\to \infty$: 
		\begin{equation}
			\label{e:Mnhatconv}
			\sup_{t\in [0,T)} \|\widehat{M}_n(\cdot, y,t) - M(\cdot, y,t)\|_{L^1(\mathbb{R})} \to 0,\qquad \text{ for $\nu$-a.e. $y\in \mathbb{R}^{d-1}$},    
		\end{equation}
		\begin{equation}
			\label{e:ptMnnutoptMnu}
			\nu^n \partial_t M_n(t)\stackrel{*}{\rightharpoonup} \nu \partial_t M(t)\qquad  \text{ in } \cM(\mathbb{R}\times \mathbb{R}^{d-1}), \text{ for all }t\in (0,T).
		\end{equation}

	\end{THEOREM}
	
	\begin{proof}
		Generate $((M^{0,k}, A^k, \nu^k))_{k=1}^\infty$, $(L_k)_{k=1}^\infty$, $(y^k_\ell)_{k\in \N, \ell\in [L_k]}$, and $(\sigma^k)_{k=1}^\infty$ from $(M^0, A, \nu)$ as in Lemma \ref{l:ydiscretization}.  Then, for each $k\in \N$, generate $(M^{0,k}_N, A^k_N)_{N=1}^\infty$ from $(M^{0,k}, A^k, \nu^k)$ as in Theorem \ref{t:sliceexist}.  For each $k$, we may choose $N_k$ large enough (on account of \eqref{e:MNtoMdisc}, \eqref{e:ANconv}) so that 
		\begin{equation}
			\label{e:MkNkAkNk}
			\max_\ell \sup_{t\in [0,T)}
			\|M^k_{N_k}(\cdot, y_\ell^k, t) - M^k(\cdot, y_\ell^k, t)\|_{L^1(\mathbb{R})} < \frac1k,
			\qquad 
			\max_\ell\|A^k_{N_k}(\cdot, y^k_\ell) - A^k(\cdot, y^k_\ell)\|_{L^\infty([0,1])} < \frac1k.
		\end{equation} 
		Put $M_n:=M^n_{N_n}$, $A_n:=A^n_{N_n}$,  for each $n\in \N$.  Then we have for $\nu$-a.e. $y\in \mathbb{R}^{d-1}$ that 
		\[
		\sup_{t\in [0,T)}\|\widehat{M}_n(\cdot, y, t) - M(\cdot, y,t)\|_{L^1(\mathbb{R})} \le \frac1n  
		+ \sup_{t\in [0,T)} \|M^n(\cdot, \sigma^n(y), t) - M(\cdot, y,t)\|_{L^1(\mathbb{R})}. 
		\]
		The second term tends to zero by \eqref{e:MTkconvslice}.  Thus \eqref{e:Mnhatconv} holds.  Furthermore, the second inequality in \eqref{e:MkNkAkNk} (together with \eqref{e:Akconv}) implies that 
		\begin{equation}
			\label{e:Anhatconv}
			\|\widehat{A}_n(\cdot, y) - A(\cdot, y)\|_{L^\infty([0,1])}\to 0,\qquad \text{ for $\nu$-a.e. $y\in \mathbb{R}^{d-1}$.}   
		\end{equation}
		With \eqref{e:Mnhatconv} and \eqref{e:Anhatconv} in hand, we can justify \eqref{e:ptMnnutoptMnu} using the same strategy as in our proof of \eqref{e:partialtMkconv} in Theorem~\ref{t:existencefinalstep}; in  particular, the flux term is treated in exactly the same way, while the convergence of the nonlocal term involves another application of Lemma \ref{l:NLprelim}.  
	\end{proof}
	
	\section{Weak solutions of the unidirectional Euler Alignment system}
	
	\label{s:Recovery}
	
	Our work in Sections \ref{s:Stability} and \ref{s:Existence} above, though motivated by our study of the Euler Alignment system, is ultimately a theory for the system \eqref{e:ScalarBalance} of nonlocal scalar balance laws.  Most of the remainder of the paper concerns the payoff for the theory of \eqref{e:EAuni}, starting with the existence and uniqueness results of this section.  Through a process we refer to as the \textit{entropic selection principle}, we  generate a uniquely determined weak solution of \eqref{e:EAuni} for given initial data, using the wellposedness theory for \eqref{e:ScalarBalance} and translating our previous results to the Euler Alignment framework.  We also demonstrate that the entropic selection principle inherits  compatibility with the sticky particle Cucker--Smale dynamics from~\eqref{e:ScalarBalance}.  More specifically, the procedure reproduces the SPCS dynamics when the initial data are atomic, and in general its solutions arise as limits of these atomic solutions. 
	
	\subsection{The entropic selection principle and approximation by SPCS dynamics}
	
	\label{ss:entropicselection}
	
	We are now in a position to state the entropic selection principle (Definition \ref{def:entropicselection} below) and to spell out some of its consequences.  Theorem \ref{t:entropicselection} clarifies the sense in which the pair $(\rho, u)$ it generates solves the Euler Alignment system, while Proposition~\ref{p:entropicdiscrete} and Corollary \ref{c:SPCSconvgen} describe the compatibility of the entropic selection principle with SPCS dynamics.

	\begin{DEF}[Entropic Selection Principle]
		\label{def:entropicselection}
		Assume $\Phi$ satisfies \eqref{e:Phiodd}--\eqref{e:Phibddcptsets} and \eqref{e:Phiymodulus}--\eqref{e:moc}.   Suppose $\rho^0\in \cP_c(\mathbb{R}\times \mathbb{R}^{d-1})$ and $u^0\in L^\infty(\rho^0)$.  We refer to the following procedure, which generates a density/velocity pair $(\rho, u)$  associated to the initial data $(\rho^0, u^0)$, as the \textit{entropic selection principle}.
		\begin{enumerate}[label = \arabic*.]
			\item (Construct the admissible triple) Let $p_2:\mathbb{R}\times \mathbb{R}^{d-1}\to \mathbb{R}$ denote projection onto the second component: $p_2(x,y) = y$, and denote $\nu = (p_2)_{\sharp} \rho^0$.  Disintegrate $\rho^0$ via $\rho^0(\dx, \dy) = \rho^0_y(\dx) \nu(\dy)$ and define $M^0:\mathbb{R}\times \mathbb{R}^{d-1}\to [0,1]$ via $M^0(x,y) = \rho^0_y((-\infty,x])$.  Define $\psi^0 = u^0 + \Phi*\rho^0$ and $A:[0,1]\times \mathbb{R}^{d-1}\to \mathbb{R}$ via  
			\begin{equation}
				\label{e:defA}
				A(m,y) = \int_0^m \psi^0(X^0(\widetilde{m},y),y)\,\dd\widetilde{m},
				\qquad X^0(m,y) = \inf\{x\in \mathbb{R}: M^0(x,y)\ge m\} = (M^0(\cdot, y))^{-1}(m).
			\end{equation}
			\item (Generate the solution to the scalar balance law) Let $M$ denote the unique entropy solution of \eqref{e:ScalarBalance} associated to the admissible triple $(M^0,A,\nu)$.
			\item (Generate the solution of the Euler Alignment system) Define $\rho(t)$ via $\rho(\dx, \dy, t) = \partial_x M(\dx, y, t)\nu(\dy)$ and $P(t)$ via $P(\dx, \dy, t) = -\partial_t M(\dx, y, t)\nu(\dy)$, and let $u(t)\in L^\infty(\rho(t))$ denote the Radon-Nikodym derivative of $P(t)$ with respect to $\rho(t)$; that is, $u(t) = \frac{\dd P(t)}{\dd\rho(t)}$. (Note that $P(t)\ll \rho(t)$ because of the scalar balance law and the BV chain rule; see the proof of Theorem \ref{t:entropicselection}(a) below for details.)
		\end{enumerate} 
	\end{DEF}

	\begin{THEOREM}
		\label{t:entropicselection} 
		
		Suppose $\rho^0\in \cP_c(\mathbb{R}\times \mathbb{R}^{d-1})$ and $u^0\in L^\infty(\rho^0)$.  Let $(\rho, u)$ be the pair generated from $(\rho^0, u^0)$ through the entropic selection principle.  
		\begin{enumerate}[label = (\alph*)]
			\item Assume $\Phi$ satisfies \eqref{e:Phiodd}--\eqref{e:Phibddcptsets} and \eqref{e:Phiymodulus}--\eqref{e:moc}.  Then $(\rho, u)$ is a weak solution of \eqref{e:EAunipsi} in the sense of Definition \ref{def:weakEAunipsi}. 
			\item If $\phi$ satisfies  \eqref{e:phieven}--\eqref{e:philoccnts} and $\Phi$ is given by $\Phi(x,y):=\int_0^x \phi(\xi, y)\,\dd\xi$, then $(\rho, u)$ is a weak solution of \eqref{e:EAuni} in the sense of Definition~\ref{def:weak_EA}.
		\end{enumerate}
	\end{THEOREM}
	
	\begin{REMARK}
		Recall that the hypotheses \eqref{e:phieven}--\eqref{e:philoccnts} and the relationship $\Phi(x,y) = \int_0^x \phi(\xi, y)\,\dd\xi$ imply that $\Phi$ satisfies \mbox{\eqref{e:Phiodd}--\eqref{e:Phibddcptsets}} and \eqref{e:Phiymodulus}--\eqref{e:moc} for some modulus of continuity $\omega$.  For the remainder of the manuscript (except for the proof of Theorem \ref{t:entropicselection}(a)), we will work with the assumptions \eqref{e:phieven}--\eqref{e:philoccnts} on $\phi$ and leave it to the reader to supply the corresponding statements under the (weaker) hypotheses on $\Phi$.    
	\end{REMARK}
	
	We give the proof of Theorem \ref{t:entropicselection} in the next subsection.  First, we state the following two results.  
	\begin{PROP}
		\label{p:entropicdiscrete}
		Assume $\phi$ satisfies \eqref{e:phieven}--\eqref{e:philoccnts}.  Assume $\rho^0\in \cP_c(\mathbb{R}^d)$ is atomic: $\rho^0 = \sum_{i=1}^N m_i \delta(\cdot - \bm{x}_i^0)$ and fix $u^0\in  L^\infty(\rho^0)$ (real-valued), so that $\rho^0 u^0$ is of the form $\rho^0 u^0 = \sum_{i=1}^N m_i v_i^0 \delta(\cdot - \bm{x}_i^0)$.  Let $(m_i, \bm{x}_i(t), v_i(t))_{i=1}^N$ denote the unidirectional sticky particle Cucker--Smale dynamics associated to $(m_i, \bm{x}_i^0, v_i^0)_{i=1}^N$.  Then the solution $(\rho, u)$ of \eqref{e:EAuni} generated by the entropic selection principle in Definition  \ref{def:entropicselection} is given by 
		\begin{equation} 
			\label{e:EASPCS}
			\rho(t) = \sum_{i=1}^N m_i \delta(\cdot - \bm{x}_i(t)),
			\qquad 
			\rho u(t) = \sum_{i=1}^N m_i v_i(t) \delta(\cdot - \bm{x}_i(t)),
			\qquad t\ge 0.
		\end{equation}
	\end{PROP}
	
	Proposition \ref{p:entropicdiscrete} says that atomic initial data $(\rho^0, u^0)$ give rise to atomic solutions $(\rho, u)$ under the entropic selection principle, with the atoms following SPCS dynamics.  We refer to such atomic solutions as \textit{sticky particle weak solutions} of~\eqref{e:EAuni}.  The proof of Proposition~\ref{p:entropicdiscrete}, which we omit, is a straightforward modification of Proposition 6.4 in \cite{LeslieTan2023}. 
	
	\begin{COROL}
		\label{c:SPCSconvgen}
		Assume $\phi$ satisfies \eqref{e:phieven}--\eqref{e:philoccnts}.
		Suppose $\rho^0\in \cP_c(\mathbb{R}^{d})$ and $u^0\in L^\infty(\rho^0)$, and fix $T>0$.  Let $(\rho, u)$ be the solution of \eqref{e:EAuni} on $[0,T)$ generated from $(\rho^0, u^0)$ by the entropic selection principle in Definition \ref{def:entropicselection}.  There exists a sequence $(\rho_n, u_n)_{n=1}^\infty$ of sticky particle weak solutions such that $\rho_n(t)\stackrel{*}{\rightharpoonup}\rho(t)$ and $\rho_n u_n(t)\stackrel{*}{\rightharpoonup}\rho u(t)$ for all $t\in (0,T)$. 
	\end{COROL}
	
	Corollary~\ref{c:SPCSconvgen} says that any weak solution $(\rho, u)$ of \eqref{e:EAuni} arising from the entropic selection principle can be approximated by sticky particle weak solutions of \eqref{e:EAuni}.  It follows as a direct consequence of Proposition \ref{p:entropicdiscrete}, Theorem \ref{t:SPCSapproxM} (which is the analogous statement at the level of the scalar balance law), and the inequality 
	\begin{equation} 
		\label{e:W1rhonuineq}
		W_1(\rho, \widetilde{\rho}) \le \int_{\mathbb{R}^{d-1}} \|M(\cdot, y) - \widetilde{M}(\cdot, \widetilde{y})\|_{L^1(\mathbb{R})}\,\bpi(\dy, \dd\widetilde{y}) + W_1(\nu, \widetilde{\nu}),
	\end{equation} 
	where $\rho(\dx, \dy) = \partial_1 M(\dx, y)\,\nu(\dy)$ and similarly for $\widetilde{\rho}$, and $\bpi$ is an optimal coupling between $\nu$ and $\widetilde{\nu}$.  We will refine the conclusion of Corollary \ref{c:SPCSconvgen} below in Section \ref{s:Rates} by providing rates of convergence under additional assumptions.

	\subsection{Verification of the weak formulations} 
	
	We now prove Theorem \ref{t:entropicselection}, verifying the requirements of Definitions \ref{def:weakEAunipsi} and \ref{def:weak_EA} in turn.  The crucial insight in part (a) is that the key quantity $\psi$ can be defined through the BV chain rule.  Equation~\eqref{e:EAunipsi}$_2$ then falls out almost for free.  Moreover, the relationship $\psi = u + \Phi*\rho$ \textit{is} the family of scalar balance laws \eqref{e:ScalarBalance}, translated in terms of $\rho = \partial_x M$ and $\rho u = -\partial_t M$.  Once part (a) of the theorem is proved, the only additional  ingredient needed for part (b) is the verification of \eqref{e:wkmomentumsymmetrized}, which is the most subtle part of the proof.  Our strategy is to insert a test function involving $\Phi(x,y) = \int_0^x \phi(\xi,y)\,\dd\xi$ into a doubled version of the continuity equation; this requires an approximation argument, since $\phi$ may be weakly singular at the origin.  Our assumption \eqref{e:phiupperbd} on $\phi$ is designed with this in mind: it is weak enough to admit the singularities of interest, while allowing us to exploit the extra factor of 
	$|x-\xi|$ offered by the symmetrization in~\eqref{e:wkmomentumsymmetrized}.  We also note that our proof of Theorem \ref{t:entropicselection}(b) patches a small gap in the proof of Theorem 6.3 in \cite{LeslieTan2023} by the second author and Tan.
	
	\medskip
	\textit{Step 1: Proof of (a).}
	The continuity equation $\partial_t \rho + \partial_x (\rho u) = 0$ follows simply from the definitions of $P$ and $u$:
	\[
	\partial_t \rho + \partial_x (\rho u) = \partial_t \partial_x M \nu + \partial_x P = 0.
	\]
	The BV chain rule (cf. Lemma \ref{l:BVChain}) guarantees that we can define $\psi(t)\in L^\infty(\rho(t))$ via 
	\begin{equation}
		\partial_x (A(M,\cdot)) =  \psi \partial_x M,
		\qquad 
		\partial_t (A(M,\cdot)) = \psi\partial_t M.
	\end{equation}
	Therefore, 
	\[
	\partial_x(A(M,y)) \nu = \rho \psi,
	\qquad \partial_t (A(M,y)) \nu = -P\psi = -\rho u \psi.
	\]
	This immediately implies that $\partial_t(\rho \psi) + \partial_x (\rho u \psi) = 0$ (which is \eqref{e:EAunipsi}$_2$).  
	
	To establish the formula $\psi = u + \Phi*\rho$ directly from the procedure outlined in Definition \ref{def:entropicselection}, we use the distributional form of the scalar balance law:
	\[
	\rho u = P = -\partial_t M \nu = \partial_x (A(M,y))\nu - (\Phi*(\nu\partial_x M))\, \partial_x M \nu = (\psi - \Phi*\rho) \rho.
	\]

	We now consider the initial conditions and take $t\to 0+$.  The convergence $\rho(t)\stackrel{*}{\rightharpoonup} \rho^0$ follows from the corresponding convergence of $M(t)$ to $M^0$.  Continuity of $\Phi*\rho$ then implies that $(\Phi*\rho(t))\rho(t)\stackrel{*}{\rightharpoonup} (\Phi*\rho^0)\rho^0$. Next, we have
	\[
	\rho \psi(t) = \partial_x (A(M(t),\cdot))\nu \stackrel{*}{\rightharpoonup} \partial_x (A(M^0, \cdot)) \nu = \partial_x \bigg( \int_{(-\infty, x]} \psi^0 \partial_x M^0 \bigg)\nu = \rho^0 \psi^0,
	\]
	where the penultimate identity follows from the definition \eqref{e:defA} of $A$ and the identity \eqref{e:geninvid}.  Combining the two convergences $\rho \psi(t)\stackrel{*}{\rightharpoonup} \rho^0\psi^0$ and $(\Phi*\rho)\rho(t)\stackrel{*}{\rightharpoonup} (\Phi*\rho^0)\rho^0$ with the relationships $\psi = u + \Phi*\rho$ and $\psi^0 = u^0 + \Phi*\rho^0$ gives us the desired convergence $\rho u(t)\stackrel{*}{\rightharpoonup} \rho^0 u^0$.

	\medskip
	\textit{Step 2: Outline of the proof of (b).}
	It remains to recover \eqref{e:wkmomentumsymmetrized}.  Our strategy is as follows.  First, we define $\varrho\in C([0,T),\cP_c((\mathbb{R}\times \mathbb{R}^{d-1})^2))$ and an $\mathbb{R}^2$-valued, measurable, and $\varrho$-a.e. defined function $U$ via 
	\[
	\varrho(\dx,\dy, \dd\xi, \dd\eta, t) :=\rho(\dx, \dy, t)\rho(\dd\xi, \dd\eta, t),
	\qquad 
	U(x,y,\xi, \eta,t) = (u(x,y,t), u(\xi, \eta, t)).
	\]
	Using the distributional form of \eqref{e:EAuni}$_1$ only, it is easily seen that the pair $(\varrho, U)$ satisfies the following `doubled' form of the continuity equation  on $(\mathbb{R}\times \mathbb{R}^{d-1})^2 \times (0,T)$, in the sense of distributions:
	\begin{equation} 
		\label{e:continuitydoubled}
		\partial_t \varrho + \n_{x,\xi} \cdot (\varrho U) = 0.
	\end{equation} 
	We will test \eqref{e:EAunipsi}$_2$ against some $\zeta\in C^\infty_c(\mathbb{R}\times \mathbb{R}^{d-1}\times (0,T))$, symmetrize the result, and use \eqref{e:continuitydoubled} with the test function 
	\begin{equation}
		\label{e:Thetadef}
		\Theta(x,y,\xi,\eta, t) = \Phi(x-\xi, y-\eta)(\zeta(x,y,t) - \zeta(\xi, \eta, t))
	\end{equation} 
	to recover \eqref{e:wkmomentumsymmetrized}.  The assumptions \eqref{e:phiupperbd} and \eqref{e:philoccnts} are used to verify that $\Theta$ can be legitimately used as a test function, i.e., that  
	\begin{equation} 
		\label{e:Thetatest}
		\int_0^T\!\!\int\!\!\!\int\!\!\!\int\!\!\!\int (\partial_t + U\cdot \n_{x,\xi}) \Theta(x,y,\xi,\eta, t)\,\varrho(\dx, \dy, \dd\xi, \dd\eta, t)\,\dt = 0.
	\end{equation} 
	Once we have done this, the strategy above can be carried out with a few straightforward manipulations. 
	
	\medskip
	\textit{Step 3: Verification of \eqref{e:Thetatest}.}
	We mollify $\phi$ in its first argument and take limits.  Let $b\in C_c^\infty(\mathbb{R})$ be nonnegative and even, with $\supp(b)\subset (-1,1)$ and $\int b = 1$; define $b_\e(x) = \frac{1}{\e} b(\frac{x}{\e})$.  Define $\phi_\e$, $\Phi_\e$, and $\Theta_\e$ via 
	\[
	\phi_\e(x,y) = \int \phi(x-\xi, y)b_\e(\xi)\,\dd\xi,
	\qquad 
	\Phi_\e(x,y) = \int_0^x \phi_\e(\xi, y)\,\dd\xi,
	\]
	\[
	\Theta_\e(x,y,\xi,\eta,t) = \Phi_\e(x-\xi, y - \eta)(\zeta(x,y,t) - \zeta(\xi, \eta, t)).
	\]
	Then 
	\begin{equation} 
		\label{e:Thetaepsilontest}
		\int_0^T\!\!\int\!\!\!\int\!\!\!\int\!\!\!\int (\partial_t + U\cdot \n_{x,\xi}) \Theta_\e(x,y,\xi,\eta, t)\,\varrho(\dx, \dy, \dd\xi, \dd\eta, t)\,\dt = 0,
	\end{equation} 
	for every $\e>0$. We wish to pass to the limit as $\e\to 0+$, to obtain \eqref{e:Thetatest}.  To this end, note that $\partial_t \Theta_\e$, $\partial_x \Theta_\e$, and $\partial_\xi \Theta_\e$ can all be legitimately substituted into the weak form of \eqref{e:continuitydoubled}; explicitly, they are given by 
	\[
	\partial_t \Theta_\e(x,y,\xi,\eta, t) = \Phi_\e(x-\xi, y - \eta)(\partial_t \zeta(x,y,t) -  \partial_t \zeta(\xi,\eta,t)),
	\]
	\begin{equation} 
		\label{e:partialxThetaepsilon}
		\partial_x \Theta_\e(x,y,\xi,\eta, t) = \phi_\e(x-\xi, y - \eta)(\zeta(x,y,t) -  \zeta(\xi,\eta,t)) +  \Phi_\e(x-\xi, y - \eta)\partial_x \zeta(x,y,t),
	\end{equation} 
	\[
	\partial_\xi \Theta_\e(x,y,\xi,\eta, t) = -\phi_\e(x-\xi, y - \eta)(\zeta(x,y,t) -  \zeta(\xi,\eta,t)) -  \Phi_\e(x-\xi, y - \eta)\partial_\xi \zeta(\xi,\eta,t).
	\]
	On the diagonal $(x,y) = (\xi, \eta)$, all three of these functions are equal to zero for all $\e>0$ and all $t$.  Off the diagonal, these three functions converge, as $\e\to 0+$, to $\partial_t \Theta$, $\partial_x \Theta$, and $\partial_\xi \Theta$, respectively.  (Note that we use \eqref{e:philoccnts} in order to conclude that the convergence is pointwise everywhere off the diagonal.)  Therefore, to pass to the limit and obtain \eqref{e:Thetatest}, it suffices to show that $\partial_t \Theta_\e$, $\partial_x \Theta_\e$, and $\partial_\xi \Theta_\e$ admit
	$\e$-independent upper bounds (for small $\e>0$), uniform on compact sets, so that the dominated convergence theorem can be applied.  This is trivial for $\partial_t \Theta_\e$.  We treat the upper bound for $\partial_x \Theta_\e$ only; the bound for $\partial_\xi \Theta_\e$ is identical. It is clear that the second term on the right side of \eqref{e:partialxThetaepsilon} admits an upper bound of the claimed form; we focus on the first term.  We have 
	\begin{align*}
		\phi_\e(x-\xi, y - \eta) |\zeta(x,y,t) - \zeta(\xi, \eta, t)| 
		& \le \|\n_{x,y}\zeta\|_{L^\infty} |(x,y) - (\xi, \eta)| \phi_\e(x-\xi, y - \eta). 
	\end{align*}
	Therefore, it suffices to estimate the quantity $|(x,y)| \phi_\e(x,y)$.  If $|(x,y)|\ge 2\e$, then 
	\[
	|(x,y)|\phi_\e(x,y)
	= |(x,y)| \int_\mathbb{R} \phi(x - \xi,y)b_\e(\xi)\,\dd\xi \le |(x,y)|\,\overline{\phi}(\tfrac{|(x,y)|}{2}) \le 2 \int_0^{|(x,y)|/2} \overline{\phi}(r)\,\dr;
	\]
	on the other hand, if $|(x,y)|<2\e$, we have 
	\[
	|(x,y)|\phi_\e(x,y)\le 2\e \int_{[-3\e,3\e]} \phi(\widetilde{\xi}, y) \big( \tfrac1\e \|b\|_{L^\infty} \big) \,\dd\widetilde{\xi} \le 4\|b\|_{L^\infty} \int_0^{3\e} \overline{\phi}(r)\,\dr.
	\]
	Altogether, the estimates above allow us to pass to the limit in \eqref{e:Thetaepsilontest}.  
	
	\medskip
	\textit{Step 4: Recovery of \eqref{e:wkmomentumsymmetrized}.}
	With \eqref{e:Thetatest} in hand, we may now proceed with the recovery of \eqref{e:wkmomentumsymmetrized}.  Note that 
	\begin{equation}
		\label{e:Thetatransport}
		\begin{split} 
			(\partial_t + U\cdot \n_{x,\xi}) \Theta(x,y,\xi,\eta, t) = & \;\Phi(x-\xi, y - \eta)\big[ \big( \partial_t \zeta + u \partial_1 \zeta\big)(x,y,t) \!-\! \big( \partial_t \zeta + u \partial_1 \zeta\big)(\xi,\eta,t) \big] \\ 
			& + \phi(x\!-\!\xi, y \!-\! \eta) \big( u(x,y,t) - u(\xi, \eta, t)\big)\big(\zeta(x,y,t) - \zeta(\xi, \eta, t)\big).
		\end{split} 
	\end{equation}
	We use \eqref{e:Thetatransport} in passing to the last line in the following computation.
	\begin{align*} 
		& \int_0^T \!\!\int\!\!\!\int \big( (\partial_t \zeta + u \partial_1 \zeta) u \big) \,\rho(\dx, \dy, t) \,\dt \\
		& = \underbrace{\int_0^T \!\!\int\!\!\!\int \big( (\partial_t \zeta + u \partial_1 \zeta) \psi \big) \,\rho(\dx, \dy, t) \,\dt}_{=0, \text{ by \eqref{e:EAunipsi}}_2}
		- \int_0^T \!\!\int\!\!\!\int \big( (\partial_t \zeta + u \partial_1 \zeta) (\Phi*\rho) \big) \,\rho(\dx, \dy, t) \,\dt \\
		& = -\frac12 \int_0^T\!\!\int\!\!\!\int\!\!\!\int\!\!\!\int \Phi(x\!-\!\xi, y \!-\! \eta)\big[ \big( \partial_t \zeta \!+\! u \partial_1 \zeta\big)(x,y,t) \!-\! \big( \partial_t \zeta \!+\! u \partial_1 \zeta\big)(\xi,\eta,t)\big]\,\rho(\dd\xi, \dd\eta,t)\,\rho(\dx,\dy, t)\,\dt \\
		& = \frac12 \int_0^T\!\!\int\!\!\!\int\!\!\!\int\!\!\!\int \phi(x-\xi, y - \eta) \big( u(x,y,t) - u(\xi, \eta, t)\big)\big(\zeta(x,y,t) - \zeta(\xi, \eta, t)\big) \rho(\dd\xi, \dd\eta,t)\,\rho(\dx,\dy, t)\,\dt \\
		& \qquad - \frac12 \underbrace{\int_0^T\!\!\int\!\!\!\int\!\!\!\int\!\!\!\int (\partial_t + U\cdot \n_{x,\xi}) \Theta(x,y,\xi,\eta, t)\,\varrho(\dx, \dy, \dd\xi, \dd\eta, t)\,\dt.}_{=0, \text{ by } \eqref{e:Thetatest}}
	\end{align*}
	This completes the proof.

	\section{Rates of convergence for the sticky particle approximation}
	
	\label{s:Rates}

	Under additional assumptions on our initial conditions $(\rho^0, u^0)$, we can guarantee a rate of convergence of a sequence of sticky particle weak solution density profiles $(\rho^k_{N(k)})_{k=1}^\infty$ (with $N(k)$ explicitly specified) to the density profile $\rho$ of the weak solution generated by Theorem \ref{t:entropicselection}.  We formulate two results toward this end; the differences between them are driven by what is allowed by each of our two stability estimates \eqref{e:stability1}, \eqref{e:stability2} from Theorem \ref{t:Mstab} after translating $(\rho^0, u^0)$ to an admissible triple $(M^0, A, \nu)$ as described in Definition \ref{def:entropicselection}.
	
	The hypotheses of Theorem \ref{t:SPCSratestrong} are designed so that \eqref{e:stability1}, which allows for weakly singular $\phi$, can be applied.  Unfortunately, the Lipschitz control on the flux difference that \eqref{e:stability1} requires is too much to ask of our discretization scheme unless the initial density profile is very favorable, specifically, with respect to vacuum.  The issue is that $A(\cdot, y)$ is defined in terms of the generalized inverse $X^0(\cdot, y)$ of $M^0(\cdot, y)$, which contains a jump discontinuity corresponding to each vacuum interval of $\rho^0_y$.  Our rather restrictive assumptions on $\rho^0$ and its support in Theorem \ref{t:SPCSratestrong} are formulated almost entirely for the purpose of ensuring appropriate regularity of $X^0$, which will in turn allow for Lipschitz control of the flux term in \eqref{e:stability1}.
	
	Theorem \ref{t:SPCSconvrateweak} avoids the strong hypotheses on the initial configuration exhibited by Theorem \ref{t:SPCSratestrong} by leveraging \eqref{e:stability2} rather than \eqref{e:stability1} when handling the transverse discretization; the former requires much less regularity on the flux difference than the latter does.  However, this comes at the cost of requiring the protocol $\phi$ to be bounded, and also a slower rate.  With some care, we can still leverage \eqref{e:stability1} to handle the horizontal discretization; we accomplish this by using the more sophisticated discretization of Lemma \ref{l:horizdiscrete} below rather than the naive scheme presented in Section \ref{ss:Existslice}.  (However, note that Lemma \ref{l:horizdiscrete} only helps us overcome the vacuum issue in the horizontal discretization and is therefore not advantageous in the context of Theorem \ref{t:SPCSratestrong}.)

	\subsection{First rate: regular initial data, possibly singular protocols}
	\label{ss:rate1}

	\begin{THEOREM}
		\label{t:SPCSratestrong}
		Assume $\phi$ satisfies \eqref{e:phieven}--\eqref{e:philoccnts} and the following two additional conditions, for some $s\in [0,1)$ and some $C>0$:
		\begin{equation}
			\label{e:phiupperbdrate}
			\phi(\bm{x})\le C\max\{|\bm{x}|^{-s},1\},\qquad \bm{x}\in \mathbb{R}^d;
		\end{equation}
		\begin{equation} 
			\label{e:Phiymocrate}
			\bigg|\int_0^x \big( \phi(\xi,y) - \phi(\xi, \widetilde{y})\big) \,\dd\xi \bigg| \le C(|x|+1) |y-\widetilde{y}|^{1-s},
			\qquad x\in \mathbb{R},\;y,\widetilde{y}\in \mathbb{R}^{d-1}. 	
		\end{equation} 	
		(The inequality \eqref{e:phiupperbdrate} is a strengthened version of \eqref{e:phiupperbd}, while \eqref{e:Phiymocrate} is a strengthened version of \eqref{e:Phiymodulus}--\eqref{e:moc}, written in terms of $\phi$.) Assume that $\rho^0\in \cP_c(\mathbb{R}^d)$ and $u^0\in L^\infty(\rho^0)$, and additionally that $\rho^0$ and $u^0$ satisfy the following properties:
		\begin{enumerate}[label = (\roman*), itemsep = 0.4em]
			\item We assume $\rho^0$ is absolutely continuous with respect to Lebesgue measure; we also use $\rho^0$ to denote a fixed (everywhere-defined) representative of its density with respect to Lebesgue measure.  We define the function $\nu:\mathbb{R}^{d-1}\to \mathbb{R}$ via $\nu(y) = \int_\mathbb{R} \rho^0(x,y)\dx$, and we assume that $\supp(\rho^0)$ has the special form $\supp(\rho^0) = \{(x,y)\in \mathbb{R}\times \supp(\nu): a(y)\le x\le b(y)\}$, where $a(y)$ and $b(y)$ are continuous functions on $\supp \nu$, with $a(y)<b(y)$ for all $y\in \supp \nu$.  Finally, we assume there exist $\underline{\rho}, \overline{\rho}\in (0,\infty)$ such that $\underline{\rho}\le \rho^0\le \overline{\rho}$ on $\supp \rho^0$.

			\item For $y\in \supp \nu$, define $\rho_y^0:\mathbb{R}\to \mathbb{R}$ via $\rho^0_y(x) = \frac{\rho^0(x,y)}{\nu(y)}$.  We assume that for some $C>0$ and some $\alpha\in (0,1]$, we have $\|\rho_y^0 - \rho_{\widetilde{y}}^0\|_{L^1(\mathbb{R})} \le C|y - \widetilde{y}|^\alpha$, for $y,\widetilde{y}\in \supp\nu$.
			\item We assume there exists an extension of $u^0$ to all of $\mathbb{R}^d$ (still denoted $u^0$)  such that, for some $C>0$ and some $\beta_1,\beta_2\in (0,1]$, we have $|u^0(x, y) - u^0(\widetilde{x}, \widetilde{y})|\le C\big(|x - \widetilde{x}|^{\beta_1} + |y - \widetilde{y}|^{\beta_2}\big)$, for $(x,y), (\widetilde{x}, \widetilde{y})\in \mathbb{R}\times \mathbb{R}^{d-1}$.
		\end{enumerate}
		Let $(\rho, u)$ denote the solution generated from $(\rho^0, u^0)$ via the entropic selection principle on some time interval $[0,T)$, and denote $\delta = \min\{\beta_1, 1-s\}$ and $\e = \min\{\alpha \delta, \beta_2\}$.  There exists a sequence of explicitly constructed sticky particle solutions $(\rho_k, u_k)_{k=1}^\infty$, the $k$th of which consists of $O(2^{k(d-1 + \frac{\e}{\delta})})$ agents, such that 
		\begin{equation}
			\sup_{t\in [0,T)} W_1(\rho_k(t), \rho(t)) \le C(1+T) 2^{-k\e}. 
		\end{equation}
	\end{THEOREM}
	\begin{proof} Choose $R^0$ large enough so that $a(y),b(y)\in [-R^0, R^0]$ for all $y\in \supp\nu$.  Denote $\Phi(x,y) = \int_0^x \phi(\xi, y)\,\dd\xi$, let $(M^0, A, \nu)$ be the admissible triple corresponding to $(\rho^0, u^0)$, and let $M$ denote the unique entropy solution of \eqref{e:ScalarBalance} generated from $(M^0,A,\nu)$.  We begin by establishing regularity properties of $M^0$ and $A$; this first step is where most of the work occurs.  Next, we discretize in the transverse directions using the scheme in Section \ref{sss:transversdiscretization} and then in the horizontal direction using the scheme from the proof of Theorem \ref{t:sliceexist}.  We use the stability bound \eqref{e:stability1}, twice, to obtain bounds on the difference between $M$ and its discretization.  Finally, we translate back to the level of the density profile, using \eqref{e:W1rhonuineq}.  
		
		\emph{Step 1: Regularity of $M^0$ and $A$.} 
		Assumptions (i) and (ii) guarantee that 
		\begin{equation}
			\label{e:rho0ylwrbd}
			\rho^0_y(x) \ge \frac{\underline{\rho}}{2R^0 \overline{\rho}}>0,
			\qquad x\in [a(y),b(y)], \; y\in \supp \nu.
		\end{equation}
		From \eqref{e:rho0ylwrbd}, we have 
		\[
		\|\rho^0_y - \rho^0_{\widetilde{y}}\|_{L^1(\mathbb{R})} \ge 	\bigg( \frac{\underline{\rho}}{2R^0 \overline{\rho}} \bigg) \big[ |a(y) - a(\widetilde{y})| + |b(y) - b(\widetilde{y})| \big],
		\qquad y,\widetilde{y}\in \supp\nu,
		\]
		and therefore $a,b\in C^{0,\alpha}(\supp\nu)$, by hypothesis (ii).  Furthermore, 
		\begin{equation}
			\label{e:M0Holdery}
			|M^0(x, y) - M^0(x, \widetilde{y})| \le \|\rho^0_y - \rho^0_{\widetilde{y}}\|_{L^1(\mathbb{R})} \le C|y-\widetilde{y}|^\alpha,	\qquad x\in \mathbb{R},\; y,\widetilde{y}\in \supp \nu,
		\end{equation}
		and \eqref{e:rho0ylwrbd} implies 
		\begin{equation}
			\label{e:M0incr}
			|M^0(x_2,y) \!-\! M^0(x_1,y)| = \bigg|\int_{x_1}^{x_2} \rho^0_y(\xi)\,\dd\xi\bigg| \ge \bigg( \frac{\underline{\rho}}{2R^0 \overline{\rho}} \bigg)|x_2-x_1|,
			\quad 
			x_1,x_2\in [a(y),b(y)],\;y\in \supp \nu. 	
		\end{equation}
		
		We have now established all the regularity we need for $M^0$; the rest of Step 1 is devoted to obtaining H\"older bounds on $\partial_m A(m,y) = \psi^0(X^0(m,y),y)$, where $\psi^0 = u^0 + \Phi*\rho^0$ and $X^0(\cdot, y)$ is the generalized inverse of $M^0(\cdot, y)$, as in Definition~\ref{def:entropicselection}.  However, because of \eqref{e:M0incr}, the function $x\mapsto M^0(x,y)$ is a bijection from $[a(y),b(y)]$ to $[0,1]$, so $X^0(\cdot,y)$ is a bona fide inverse of $M^0(\cdot, y)$ for $y\in \supp\nu$; it  satisfies 
		\begin{equation}
			|X^0(m,y) - X^0(\widetilde{m},y)| \le \bigg( \frac{2R^0 \overline{\rho}}{\underline{\rho}} \bigg) |m-\widetilde{m}|,	\qquad m,\widetilde{m}\in [0,1],\qquad y\in \supp \nu,
		\end{equation}
		simply by inverting \eqref{e:M0incr}.  The function $X^0$ also inherits H\"older regularity from $\rho^0$ in the transverse directions.  Indeed, fix $y,\widetilde{y}\in \supp\nu$ and $m\in [0,1]$.  Denote $x = X^0(m,y)$ and $\widetilde{x} = X^0(m,\widetilde{y})$; this implies $M^0(x,y) = m = M^0(\widetilde{x}, \widetilde{y})$.  In what follows, we would like to use \eqref{e:M0incr}, with $x_1 = x$ and $x_2 = \widetilde{x}$; however, $\widetilde{x}$ might not lie in $[a(y), b(y)]$, so we must make a small detour.  Instead, we will take $x_2$ equal to $\widetilde{x}_*$, which we define as the closest point of $[a(y),b(y)]$ to $\widetilde{x}$; this could be $\widetilde{x}$ itself, or $a(y)$, or $b(y)$.  In any case, we certainly have $|\widetilde{x} - \widetilde{x}_*|\le |a(y) - a(\widetilde{y})| + |b(y) - b(\widetilde{y})| \le C|y - \widetilde{y}|^\alpha$.  Note furthermore that $M^0(\widetilde{x}_*,y) = M^0(\widetilde{x},y)$ in each of these three cases.  Therefore, 
		\begin{align*}
			|X^0(m,y) - X^0(m,\widetilde{y})|
			& = |x - \widetilde{x}| \le |x - \widetilde{x}_*| + |\widetilde{x}_* - \widetilde{x}| \\
			& \le \bigg( \frac{2R^0 \overline{\rho}}{\underline{\rho}} \bigg) |M^0(x,y) - M^0(\widetilde{x}_*,y)| + C|y-\widetilde{y}|^\alpha \\
			& =  \bigg( \frac{2R^0 \overline{\rho}}{\underline{\rho}} \bigg) |M^0(\widetilde{x},\widetilde{y}) - M^0(\widetilde{x},y)| + C|y-\widetilde{y}|^\alpha \\	
			& \le C|y - \widetilde{y}|^\alpha.
		\end{align*}
		Our assumptions on $\phi$ guarantee that 
		\begin{equation} 
			|\Phi*\rho^0(x,y) - \Phi*\rho^0(\widetilde{x},\widetilde{y})| \le C\big( |x - \widetilde{x}|^{1-s} + |y - \widetilde{y}|^{1-s} \big),
			\qquad x,\widetilde{x}\in [-R^0, R^0], \,\,y, \widetilde{y}\in \mathbb{R}^{d-1},  
		\end{equation} 	
		and therefore (combining the above with hypothesis (iii)) that $\psi^0$ satisfies 
		\begin{equation}
			\label{e:psi0Holder}
			|\psi^0(x,y) - \psi^0(\widetilde{x},\widetilde{y})|
			\le C\big(|x - \widetilde{x}|^{\min\{\beta_1, 1-s\}} + |y - \widetilde{y}|^{\min\{\beta_2, 1 - s\}}\big),
			\qquad (x,y), (\widetilde{x}, \widetilde{y})\in \supp \rho^0.
		\end{equation}
		We may finally combine the estimates on $\psi^0$ and $X^0$ above.  Recalling the definitions $\delta = \min\{\beta_1, 1-s\}$ and $\e = \min\{\alpha \delta, \beta_2\}$, we write
		\begin{align*}
			|\partial_m A(m,y) - \partial_m A(\widetilde{m}, \widetilde{y})|
			& = |\psi^0(X^0(m,y),y) - \psi^0(X^0(\widetilde{m},\widetilde{y}),\widetilde{y})|\\
			& \le C\big(|X^0(m,y) - X^0(\widetilde{m}, \widetilde{y})|^\delta + |y - \widetilde{y}|^{\min\{\beta_2, 1-s\}}\big) \\
			& \le C \big[ \big( |m-\widetilde{m}|  + |y - \widetilde{y}|^\alpha \big)^\delta + |y - \widetilde{y}|^{\min\{\beta_2, 1-s\}} \big] \\
			& \le C\big(|m-\widetilde{m}|^\delta + |y - \widetilde{y}|^\e\big),
			\qquad  
			\qquad \qquad m,\widetilde{m}\in [0,1],\;\;y,\widetilde{y} \in \supp \nu.
		\end{align*}
		(In passing to the last line, we have used that $\alpha \delta \le \delta \le 1-s$.)  In particular, we have 
		\begin{equation}
			\label{e:ALiprate}
			[A(\cdot, y) - A(\cdot, \widetilde{y})]_{\Lip([0,1])} \le C|y - \widetilde{y}|^\e, \qquad y,\widetilde{y}\in \supp \nu.	
		\end{equation}

		\emph{Step 2: Transverse discretization and estimates.} Starting from $(M^0, A, \nu)$, define $(M^{0,k}, A^k, \nu^k)$, $L_k$, $(y_\ell^k)_{\ell=1}^{L_k}$, $\sigma^k$, and $M^k$ as in the proof of Lemma \ref{l:ydiscretization}.	Then for $x\in \mathbb{R}$ and $y\in \supp\nu$, we have (using \eqref{e:M0Holdery}) that
		\begin{equation}
			|M^0(x,y) - M^{0,k}(x,\sigma^k(y))| \le \sup_{\{\eta\in \supp\nu:\sigma^k(\eta) = \sigma^k(y)\}} |M^0(x,y) - M^0(x,\eta)| \le C 2^{-\alpha k}.
		\end{equation}	
		Similarly, using \eqref{e:ALiprate},
		\begin{equation}
			\big[A(\cdot,y) - A^k(\cdot,\sigma^k(y))\big]_{\Lip([0,1])} \le \sup_{\{\eta\in \supp \nu:\sigma^k(\eta) = \sigma^k(y)\}} [A(\cdot, y) - A(\cdot, \eta)]_{\Lip([0,1])} \le C2^{-\e k}.	
		\end{equation}
		Applying \eqref{e:stability1} (with $R(T) = C(1+T)$ for a suitable constant $C$, cf. \eqref{e:RTslice}) therefore yields directly that 
		\begin{equation}
			\label{e:transdiscrate1}
			\begin{split} 
			& \int_{\mathbb{R}^{d-1}} \|M(\cdot, y,t) - M^k(\cdot, \sigma^k(y),t)\|_{L^1(\mathbb{R})} \nu(\dy) \\
			& \quad \le C(2R^0) 2^{-\alpha k} + C 2^{-\e k} + C(2R(T)+1) 2^{-(1-s)k} \\
			& \quad \le C(1+T)2^{-\e k}.
			\end{split} 
		\end{equation}
		
		\emph{Step 3: Horizontal discretization and estimates.} Denote $\theta_{i,N} = \frac{i}{N}$, for all $N\in \mathbb{N}$ and $i\in \{0,1, \ldots, N\}$.  For each $k,N\in \N$ and $\ell\in [L_k]$, define 
		\[
		x^{0,k}_{i,\ell, N} = \inf\{x\in \mathbb{R}: M^{0,k}(x, y^k_\ell)\ge \theta_{i,N}\},
		\qquad m_{i,\ell,N} = \nu^k(\{y^k_\ell\})N^{-1}.
		\]
		Denote 
		\[
		M_N^{0,k}(x,y^k_\ell) = \frac1N \sum_{i=1}^N H(x - x^{0,k}_{i,\ell,N}), \qquad x\in \mathbb{R}, 
		\]
		and let $X_N^{0,k}(\cdot, y^k_\ell)$ denote the generalized inverse of $M_N^{0,k}(\cdot, y^k_\ell)$.  Let $A_N^k(\cdot, y^k_\ell)$ be the piecewise linear function, with breakpoints at each $\theta_{i,N}$, such that $A_N^k(\theta_{i,N}, y^k_\ell) = A^k(\theta_{i,N}, y^k_\ell)$ for each $i,k,\ell,N$.  Let $M^k_N$ denote the unique entropy solution of \eqref{e:ScalarBalance} generated from the admissible triple $(M^{0,k}_N, A^k_N, \nu^k)$.  Note that (cf. \eqref{e:MN0conv})
		\begin{equation} 
			\label{e:M0kNrateconv}
			\|M^{0,k}_N(\cdot, y^k_\ell) - M^{0,k}(\cdot, y^k_\ell)\|_{L^1(\mathbb{R})} \le \frac{2R^0}{N}.
		\end{equation} 
		We now extract a rate of convergence on the flux.  Note first of all that each function $\partial_m A^k(\cdot, y^k_\ell)$ inherits H\"older continuity from the maps $\partial_m A(\cdot, y)$, since the former is defined as an average of the latter.  Furthermore, since $A^k(\cdot, y^k_\ell)$ and $A^k_N(\cdot, y^k_\ell)$ agree at each breakpoint, we can argue as follows.  For each $i\in [N]$, there exists $\theta^*\in (\theta_{i-1,N}, \theta_{i,N})$ such that 
		\[
		\partial_m A^k(\theta^*,y^k_\ell) = N \big( A^k(\theta_{i,N},y^k_\ell) - A^k(\theta_{i-1,N}, y^k_\ell)\big) = \partial_m A^k_N(m,y^k_\ell),
		\qquad \text{ for all } m\in (\theta_{i-1,N}, \theta_{i,N}).
		\]   
		Thus for $m\in (\theta_{i-1, N}, \theta_{i,N})$, 
		\[
		|\partial_m A^k(m,y^k_\ell) - \partial_m A^k_N(m,y^k_\ell)| 
		= |\partial_m A^k(m,y^k_\ell) - \partial_m A^k(\theta^*,y^k_\ell)|\le C|m - \theta^*|^\delta \le CN^{-\delta}.
		\]
		Consequently, an application of \eqref{e:stability1} yields
		\begin{equation}
			\label{e:horizdiscrrate1}
			\int_{\mathbb{R}^{d-1}} \|M^k_N(\cdot, y, t) - M^k(\cdot, y,t)\|_{L^1(\mathbb{R})} \nu^k(\dy) \le C(N^{-1} + tN^{-\delta})\le C(1 + T)N^{-\delta}.		
		\end{equation}
		
		\emph{Step 4: Conclusion.} Putting together \eqref{e:transdiscrate1} and \eqref{e:horizdiscrrate1} yields 
		\begin{equation}
			\label{e:rate1M}	
			\int_{\mathbb{R}^{d-1}} \|M(\cdot, y, t) - M^k_N(\cdot, \sigma^k(y),t)\|_{L^1(\mathbb{R})} \nu(\dy) \le C(1+T)(2^{-\e k} + N^{-\delta}).	
		\end{equation}
		Choose $N = N(k)\sim 2^{\frac{\e}{\delta} k}$.  Let $(\rho^0_k, u^0_k)$ denote the initial pair generated from $(M^{0,k}_{N(k)}, A^k_{N(k)}, \nu^k)$, and let $(\rho_k, u_k)$ denote the associated solution.  The desired conclusion then follows from  \eqref{e:W1rhonuineq}.
	\end{proof}

	\subsection{Second rate: bounded protocols, more flexible initial data}
	
	\label{ss:rate2}

	\begin{THEOREM}
		\label{t:SPCSconvrateweak}
		Assume $\phi$ is bounded and Lipschitz, and that it satisfies \eqref{e:phieven}. Assume that $\rho^0\in \cP_c(\mathbb{R}^d)$ and define $\nu = (p_2)_{\sharp} \rho^0$.  Assume that $u^0$ satisfies hypothesis (iii) of Theorem \ref{t:SPCSratestrong}.  Fix $\gamma\in [0,1)$, and assume that there is a family $\{\rho^0_y\}_{y\in \mathbb{R}^{d-1}}\subset \cP_c(\mathbb{R})$ such that $\rho^0(\dx, \dy) = \rho^0_y(\dx)\nu(\dy)$, such that $\supp \rho^0_y\subset [-R^0, R^0]$ for all $y\in \mathbb{R}^{d-1}$, for some $R^0>0$, and such that the map $y\mapsto \rho^0_y$ satisfies the following H\"older-type estimate, for some $\alpha\in (0,1]$ and all $y,\widetilde{y}\in \mathbb{R}^{d-1}$.
		\begin{equation}
			\label{e:rho0Holdery}
			W_p(\rho^0_y, \rho^0_{\widetilde{y}}) \le C |y - \widetilde{y}|^{\alpha},
			\qquad p = \max\bigg\{ 1, \frac{\beta_1}{1-\gamma}\bigg\}.
		\end{equation}
		Let $(\rho, u)$ denote the solution generated from $(\rho^0, u^0)$ via the entropic selection principle. Then there exists a sequence of sticky particle weak solutions $(\rho_k, u_k)_{k=1}^\infty$, the $k$th of which consists of $O(2^{k( (d-1) + \frac{1}{2-\gamma} \min\{\alpha, \frac{\beta_2}{\beta_1}\})})$ agents, such that
		\begin{equation}
			\label{e:rate2}
			\sup_{t\in [0,T)} W_1(\rho(t),  \rho_k(t)) \le C_T 2^{-\frac{k}{2-\gamma} \min\{\alpha \beta_1, \beta_2 \}},
		\end{equation}
		where the constant $C_T$ may depend on $T$.
	\end{THEOREM}
	
	\begin{REMARK}
		We only use the measures $\rho^0_y$ for $y\in \supp \nu$; defining $\rho^0_y$ for all $y\in \mathbb{R}^{d-1}$ is just a notational convenience below.  Note also that one does not absolutely need $\phi$ to be Lipschitz in order to establish a bound like \eqref{e:rate2}; one can get by with something like \eqref{e:Phiymodulus}--\eqref{e:moc}, with $\omega(r) = r^q$ for some $q\in (0,1]$, at the cost of introducing another parameter into the statement and proof.  
	\end{REMARK}

	\begin{proof}[Proof of Theorem \ref{t:SPCSconvrateweak}]
	 Starting from $(\rho^0, u^0)$, let $(M^0, A, \nu)$, $M$, and $(\rho, u)$ be as in the entropic selection principle of Definition \ref{def:entropicselection}.  Define $M^{0,k}, \nu^k$, $L_k$, $(Q^k_\ell)_{\ell=1}^{L_k}$, $(y_\ell^k)_{\ell=1}^{L_k}$, $(\mu_\ell^k)_{\ell=1}^{L_k}$, and $\sigma^k$ as in the proof of Lemma~\ref{l:ydiscretization}.  Let $X^{0,k}(\cdot, y^k_\ell)$ denote the generalized inverse of $M^{0,k}(\cdot, y^k_\ell)$ for each $k,\ell$. Define 
		\begin{equation} 
			\label{e:rho0k}
			\rho_{y}^{0,k} = \partial_x M^{0,k}(\cdot, y),
			\qquad \rho^{0,k}(\dx, \dy) = \rho_y^{0,k}(\dx) \nu^k(\dy),
		\end{equation} 
		\begin{equation}
			\label{e:u0k} 
			u^{0,k}(x,y^k_\ell) = \frac{1}{\nu(Q^k_\ell)} \int_{Q^k_\ell} u^0(x,y)\,\nu(\dy),
			\qquad 	
			\psi^{0,k} = u^{0,k} + \Phi*\rho^{0,k}.
		\end{equation} 	
		The definitions \eqref{e:rho0k} are `as usual', while the definition of $u^{0,k}$ in \eqref{e:u0k} represents a departure from the discretization scheme we have used previously, implemented to guarantee H\"older regularity of $\psi^{0,k}$ in the horizontal direction. With $(\rho^{0,k},u^{0,k})$ in hand, let $(M^{0,k}, A^k, \nu^k)$ denote the associated admissible triple, and let $M^k$ denote the entropy solution of \eqref{e:ScalarBalance} that it generates.

		\emph{Step 1: Regularity of $M^0$ and $\psi^{0,k}$.} The regularity we need for $M^0$ follows directly from \eqref{e:rho0Holdery}:
		\begin{equation}
			\|M^0(\cdot, y) - M^0(\cdot, \widetilde{y})\|_{L^1(\mathbb{R})} 
			= W_1(\rho^0_{y}, \rho^0_{\widetilde{y}}) \le W_p(\rho^0_{y}, \rho^0_{\widetilde{y}}) \le C|y - \widetilde{y}|^\alpha.	
		\end{equation}
		Since $\phi$ is bounded and Lipschitz, it follows from our assumption on $u^0$ that
		\begin{equation}
			\label{e:psi0Holderrate2} 
			|\psi^0(x,y) - \psi^0(\widetilde{x},\widetilde{y})|
			\le C\big(|x - \widetilde{x}|^{\beta_1} + |y - \widetilde{y}|^{\beta_2}\big),\qquad x,\widetilde{x}\in [-R^0, R^0].
		\end{equation} 
		The function $\psi^{0,k} = u^{0,k} + \Phi*\rho^{0,k}$ also inherits H\"older regularity in the horizontal direction; H\"older regularity of $u^{0,k}$ follows from that of $u^0$ together with the averaging procedure used to define $u^{0,k}$, while $\Phi*\rho^{0,k}$ is Lipschitz.  We record the relevant estimate for use in Steps 2 and 3 below:
		\begin{equation} 
			\label{e:psi0kHolder}
			|\psi^{0,k}(x,y) - \psi^{0,k}(\widetilde{x},y)|	\le C|x-\widetilde{x}|^{\beta_1}.
		\end{equation} 
	%	Using \eqref{e:psi0Holderrate2}, we deduce the following (leveraging \eqref{e:Wpid} and Jensen's inequality to pass to the third line below).
	%	\begin{align*}
	%		|A(m,y) - A(m,\widetilde{y})|
	%		& \le \int_0^1 \big|\psi^0(X^0(\widetilde{m},y),y) - \psi^0(X^0(\widetilde{m},\widetilde{y}), \widetilde{y})\big| \dd \widetilde{m} \\
	%		& \le C \int_0^1  |X^0(\widetilde{m},y) - X^0(\widetilde{m},\widetilde{y})|^{\beta_1} \dd\widetilde{m} + C|y-\widetilde{y}|^{\beta_2}\\
	%		& \le C W_1(\rho^0_y, \rho^0_{\widetilde{y}})^{\beta_1} + C|y - \widetilde{y}|^{\beta_2} \\
	%		& \le C |y-\widetilde{y}|^{\min\{\alpha\beta_1, \beta_2\}}.
	%	\end{align*}
		
		\emph{Step 2: Transverse discretization and estimates.}	Using the regularity we have established in the previous step, we have
		\[
		\int \|M^0(\cdot,y) - M^{0,k}(\cdot, \sigma^k(y))\|_{L^1(\mathbb{R})}\,\nu(\dy)\le C 2^{-\alpha k}.
		\]
		Using \eqref{e:Wppconvex}, we also have (for $y\in Q^k_\ell$)
		\begin{equation} 
			\label{e:Wppest}
			W_p(\rho^0_y, \rho^{0,k}_{\sigma^k(y)})^p \le \frac{1}{\nu(Q^k_\ell)}\int_{Q^k_\ell} W_p(\rho^0_y, \rho^0_\eta)^p\,\nu(\dd\eta) \le  C 2^{-\alpha p k}.
		\end{equation}
		Since the $A^k$ we have constructed in this proof is not the one from the proof of Lemma \ref{l:ydiscretization}, we cannot estimate the flux term simply by averaging over each $Q^k_\ell$, the way we did in Step 2 of the proof of Theorem \ref{t:SPCSratestrong} once the regularity estimates were in place.  However, since $u^{0,k}$ is now defined from $u^0$ through an averaging procedure, and we can estimate the convolution term similarly to the first part of Lemma \ref{l:NLprelim}, we may write 
		\begin{align*} 
			|\psi^0(x,y) - & \psi^{0,k}(x,\sigma^k(y))| \\
			& \le |u^0(x,y) - u^{0,k}(x,\sigma^k(y))| + |\Phi*\rho^0(x,y) - \Phi*\rho^{0,k}(x,\sigma^k(y))| \\
			& \le \bigg( \sup_{\{\eta: \sigma^k(\eta) = \sigma^k(y)\}} |u^0(x,y) - u^0(x,\eta)|\bigg) \\
			& \qquad + 2R^0 \|\phi\|_{W^{1,\infty}} \bigg[ \int \|M^0(\cdot, \eta) - M^{0,k}(\cdot, \sigma^k(\eta))\|_{L^1(\mathbb{R})} \nu(\dd\eta)  + |y - \sigma^k(y)| + W_1(\nu, \nu^k)\bigg]  \\
			& \le C2^{-\beta_2 k} + C2^{-\alpha k} + C2^{-k} = C2^{-k\min\{\alpha, \beta_2\}}.
		\end{align*} 
		Therefore, we have the following:
		\begin{align*} 
			& \int [A(\cdot, y)  -  A^k(\cdot, \sigma^k(y))]_{C^{0,\gamma}([0,1])}\,\nu(\dy) \\
			& \qquad \le \int \| \partial_m A(\cdot, y)  - \partial_m A^k(\cdot, \sigma^k(y))\|_{L^{\frac{1}{1-\gamma}}([0,1])}\,\nu(\dy) \\
			& \qquad\le \int \|\psi^0(X^0(\cdot,y),y) - \psi^{0,k}(X^0(\cdot,y),\sigma^k(y))\|_{L^{\frac{1}{1-\gamma}}([0,1])}\,\nu(\dy) \\
			& \qquad\qquad + \int \|\psi^{0,k}(X^{0}(\cdot,y),\sigma^k(y)) - \psi^{0,k}(X^{0,k}(\cdot,\sigma^k(y)),\sigma^k(y))\|_{L^{\frac{1}{1-\gamma}}([0,1])} \,\nu(\dy) \\
			& \qquad\le C\bigg(  2^{-k\min\{\alpha, \beta_2\}} + \int \|X^0(\cdot,y) - X^{0,k}(\cdot,\sigma^k(y))\|_{L^{\frac{\beta_1}{1-\gamma}}([0,1])}^{\beta_1}\,\nu(\dy) \bigg). 
		\end{align*}
		To finish the estimate, we note that (using \eqref{e:Wpid} in the equality and \eqref{e:Wppest} in the last line) 
		\begin{align*}
			\int \|X^0(\cdot,y) - X^{0,k}(\cdot,\sigma^k(y))\|_{L^{\frac{\beta_1}{1-\gamma}}([0,1])}^{\beta_1}\,\nu(\dy) 
			& \le \bigg( \int \|X^0(\cdot,y) - X^{0,k}(\cdot,\sigma^k(y))\|_{L^{p}([0,1])}^{p}\,\nu(\dy) \bigg)^{\frac{\beta_1}{p}}\\
			& = \bigg( \int W_p(\rho^0_y, \rho^{0,k}_{\sigma^k(y)})^{p}\,\nu(\dy) \bigg)^{\frac{\beta_1}{p}} \\
			& \le  C 2^{-\alpha \beta_1 k}.
		\end{align*}
		Therefore, 
		\[
		\int [A(\cdot, y)  -  A^k(\cdot, \sigma^k(y))]_{C^{0,\gamma}([0,1])}\,\nu(\dy) \le C 2^{-k \min\{ \alpha \beta_1, \beta_2\}}.  
		\]
		Applying \eqref{e:stability2} thus yields
		\begin{equation}
			\begin{split}
			& \int \|M(\cdot,y,t) - M^{k}(\cdot, \sigma^k(y),t)\|_{L^1(\mathbb{R})}\,\nu(\dy) \\
			& \qquad \le C_T \big( 2^{-\alpha k} + (2^{-k \min\{\alpha \beta_1, \beta_2\}})^{\frac{1}{2-\gamma}} + 2^{-k}\big) \\
			& \qquad \le C_T 2^{-\frac{k}{2-\gamma} \min\{\alpha \beta_1, \beta_2\}}.
			\end{split} 
		\end{equation}
		
		\emph{Step 3: Horizontal discretization and estimates.} To generate a sequence of fully discrete initial triples $(M^{0,k}_N, A^k_N, \nu^k)$ that approximate $(M^{0,k}, A^k, \nu^k)$ appropriately, we use the following Lemma (a proof of which can be found in \cite{LeslieTan2023}).
		
		\begin{LEMMA}
			\label{l:horizdiscrete}
			Let $\lambda_F\in \cP_c(\mathbb{R})$ be a compactly supported probability measure on $\mathbb{R}$ with (right-continuous) distribution function $F(x) = \lambda_F((-\infty, x])$.  For each $N\in \mathbb{N}$, there exists an atomic probability measure  $\lambda_{F_N} = \sum_{j=1}^N a_j \delta(\cdot - x_j)$, with distribution function $F_N$, such that 
			\begin{equation}
				\|F - F_N\|_{L^1(\mathbb{R})} \le \frac{D}{N},
				\qquad 	
				\|F^{-1} - F_N^{-1}\|_{L^\infty((0,1])} \le \frac{D}{N},	
			\end{equation}  	
			where $D = \diam(\supp(\lambda_F))$, and $F^{-1}$ and $F_N^{-1}$ denote the generalized inverses of $F$ and $F_N$, respectively.
		\end{LEMMA}
		
		Denote 
		\[
		M^{0,k}_N(x, y^k_{\ell}) = \frac{1}{\mu^k_\ell}\sum_{i=1}^N m_{i,\ell,N}^k H(x - x^{0,k}_{i,\ell,N}),
		\]
		where the masses $m_{i,\ell,N}^k$ and initial positions $x^{0,k}_{i,\ell,N}$ are obtained through Lemma \ref{l:horizdiscrete} and the requirements
		\[
		\|M^{0,k}(\cdot, y^k_\ell) - M_N^{0,k}(\cdot, y^k_\ell)\|_{L^1(\mathbb{R})}\le \frac{2R^0}{N},
		\qquad 
		\|X^{0,k}(\cdot, y_\ell^k) - X^{0,k}_N(\cdot, y_\ell^k)\|_{L^\infty(0,1]} \le \frac{2R^0}{N},
		\]
		where $X^{0,k}(\cdot, y_\ell^k)$ is the generalized inverse of $M^{0,k}(\cdot, y^k_\ell)$.

		Define 
		\[
		\psi^{0,k}_{i,\ell, N} 
		= \psi^{0,k}(x^{0,k}_{i,\ell, N}, y^k_\ell),
		\qquad 
		v^{0,k}_{i,\ell, N}
		= \psi^{0,k}_{i,\ell, N} - \sum_{p = 1}^{L_k} \sum_{j=1}^N m^k_{j,p, N} \Phi( x^{0,k}_{i,\ell, N} - x^{0,k}_{j,p, N}, y^k_\ell - y^k_p),
		\]
		\[
		\theta^k_{i,\ell, N} = \frac{1}{\mu^k_\ell}\sum_{j=1}^i m^k_{j,\ell, N},\;\;i=1,\ldots, N,\qquad \theta^k_{0,\ell, N} = 0.
		\]
		Finally, let $(M^{0,k}_N, A^k_N, \nu^k)$ denote the admissible triple associated to the SPCS initial data defined above, and let $M^k_N$ denote the entropy solution generated from it.  Let $(\rho^{0,k}_N, u^{0,k}_N)$ and $(\rho^k_N, u^k_N)$ denote the associated Euler Alignment initial data and weak solution.  For $m\in (\theta^k_{i-1,\ell, N}, \theta^k_{i,\ell, N})$, we have   
		\begin{align*}
			|\partial_m A^k (m,y^k_\ell) - \partial_m A^k_N(m,y^k_\ell)|
			& = |\psi^{0,k}(X^{0,k}(m,y^k_\ell), y^k_\ell) - \psi^{0,k}_{i,\ell, N}| \\
			& = |\psi^{0,k}(X^{0,k}(m,y^k_\ell), y^k_\ell) - \psi^{0,k}(X^{0,k}_N(\theta^k_{i,\ell, N}, y^k_\ell), y^k_\ell)| \\
			& \le C|X^{0,k}(m,y^k_\ell) - X^{0,k}_N(m,y^k_\ell)|^{\beta_1} \\
			& \le CN^{-\beta_1}.
		\end{align*}
		This allows us to use \eqref{e:stability1} to conclude that 
		\begin{equation} \sup_{t\in [0,T)} \|M^k_N(t) - M^k(t)\|_{L^1(\dx\otimes \nu^k)} \le CN^{-1} + CTN^{-\beta_1} \le C_T N^{-\beta_1}.     
		\end{equation}
		
		\emph{Step 4: Conclusion.} Using \eqref{e:W1rhonuineq}, combining the above estimates, and choosing $N = N(k)\sim 2^{\frac{k}{2-\gamma} \min\{\alpha, \frac{\beta_2}{\beta_1}\}}$, we obtain  
		\begin{align*} 
			\sup_{t\in [0,T)} W_1(\rho(t),\rho^k_N(t)) 
			& \le C_T 2^{-\frac{k}{2-\gamma} \min\{\alpha \beta_1, \beta_2\}} + C_T N^{-\beta_1} + C 2^{-k} \le C_T 2^{-\frac{k}{2-\gamma}\min\{\alpha \beta_1, \beta_2\}}. 
		\end{align*} 
		Setting $\rho_k = \rho^k_{N(k)}$ then finishes the proof.
	\end{proof}

	\section{Flocking}
	\label{s:Flocking}
	
	We now consider the long-time behavior associated with unidirectional Cucker--Smale and Euler Alignment dynamics.  We start with the former, which already contains the main novelty for our purposes.  
	
	\subsection{Flocking for unidirectional Cucker--Smale dynamics}
	It has been known for some time (cf. \cite{HaLiu2008}) that solutions of the Cucker--Smale ODEs experience flocking and exponentially fast velocity alignment for any initial configuration, if the communication protocol $\phi:\mathbb{R}^d\to \mathbb{R}$ is even and `heavy-tailed'.  When we say that $\phi$ is `heavy-tailed', we mean that there exists a nonincreasing function $\underline{\phi}:[0,\infty)\to \mathbb{R}$ such that $\phi(\bm{x})\ge \underline{\phi}(|\bm{x}|)$ and $\underline{\phi}$ is nonintegrable at infinity:
	\begin{equation}
		\label{e:heavytail}
		\int_1^\infty \underline{\phi} = +\infty.
	\end{equation} 
	In fact, the authors of \cite{HaLiu2008} already obtain a more quantitative version of the heavy-tail condition that depends on the initial configuration.   Whether \eqref{e:heavytail} or a more refined version is used, however, one can think of the underlying mechanism as follows: Initially, every agent interacts with every other agent.  The heavy tail condition ensures that the spatial decay of these interactions is just slow enough that it is impossible for one agent to completely escape the influence of any other; it becomes trapped within some fixed diameter of the rest of the flock, so that the minimal strength of all pairwise interactions remains bounded below, which drives the exponentially fast velocity alignment.  
	
	Actually, even in Cucker and Smale's paper \cite{CS2007b}, it was shown that flocking can be guaranteed even if some agents never communicate with each other.  What is important is the connectivity of the underlying graph structure of the communication network.  In particular, flocking and exponentially fast velocity alignment are guaranteed as long as the time-dependent `weighted Fiedler number' $\kappa_2$ is nonintegrable: $\int_0^\infty \kappa_2(t)\,\dt = +\infty$.  For more on this direction (and a precise definition of the weighted Fiedler number), the reader may consult \cite{CS2007b, MT2014, Tadmor2021review, ShvydkoyBook}, among many other references.  Using the Fiedler number to establish flocking, however, typically results in bounds on the flock diameter (and the rate of alignment) that depend on the number of agents $N$; as such, these techniques often do not extend well to the hydrodynamic setting.  
	
	In the theorem below, we prove that the Cucker--Smale system experiences flocking (at a rate independent of the number of agents) even if the communication degenerates in a cylinder around the agents, provided that there is enough communication with agents outside each cylinder.  The cylinders are invariant under the unidirectional dynamics, so the  communication network is effectively frozen if one imposes an analog of the heavy-tail condition outside the cylinder of degeneracy.  If we take the radius $r$ of the cylinders to be zero, we recover a more standard type of flocking estimate, exactly in line with the quantitative version of Ha and Liu's estimate.    
	
	In what follows, we use $B(y,r)$ to denote the open ball of radius $r$, centered at $y$, in $\mathbb{R}^{d-1}$.
	
	\begin{THEOREM}
		\label{t:flockingCS}
		Assume $\phi$ satisfies \eqref{e:phieven}--\eqref{e:philoccnts}. Let $(m_i, \bm{x}_i(t), \bm{v}_i(t))_{i=1}^N$ be a solution of the Cucker--Smale system \eqref{e:CS} associated to the initial data $(\bm{x}_i^0,\bm{v}_i^0)_{i=1}^N$.  Assume that $\sum_{i=1}^N m_i = 1$ and that the initial data (and therefore the solution) is unidirectional; denote $\bm{x}_i(t) = (x_i(t), y_i)$, $\bm{v}_i(t) = (v_i(t), 0)$.  Define furthermore the quantities 
		\begin{equation} 
			\label{e:DAdef}
			\mathscr{D}_1(t) = \max_{i,j} (x_i(t) - x_j(t)),
			\qquad 
			\mathscr{D}_2 = \max_{i,j} |y_i - y_j|,
			\qquad 
			\mathscr{A}(t) = \max_{i,j} (v_i(t) - v_j(t)),
		\end{equation} 
		and define $\mathscr{D}_1^0$ and $\mathscr{A}^0$ analogously in terms of the initial data.  Assume there exist $r>0$, $\vartheta\in (0,1)$, and a nonincreasing function $\underline{\phi}\in L^1_{\loc}([0,\infty))$, such that the proportion of the total mass contained in any two balls of radius $r$ in $\mathbb{R}^{d-1}$ is bounded above by $1-\vartheta$:
		\begin{equation}
			\label{e:flockingreq1}
			\sum_{\{i:y_i\in B(y,r)\cup B(z,r)\}} m_i < 1-\vartheta,
			\qquad \text{ for all } y,z\in \mathbb{R}^{d-1},
		\end{equation} 
		and such that 
		\begin{equation}
			\label{e:heavytailsub}
			\phi(x,y)\ge \underline{\phi}(|x|),
			\quad \text{ whenever } x\in \mathbb{R}, \,r\le |y|\le \mathscr{D}_2, 
			\qquad\text{ and } \qquad
			\vartheta\int_{\mathscr{D}^0_1}^\infty \underline{\phi}(x)\,\dx > \mathscr{A}^0.
		\end{equation}
		Then the solution experiences flocking and exponentially fast alignment:  There exists $\overline{\mathscr{D}}_1>0$ such that
		\begin{equation}
			\label{e:uniflocking}
			\sup_{t\ge 0} \mathscr{D}_1(t)\le \overline{\mathscr{D}}_1,
			\qquad 
			\mathscr{A}(t)
			\le \mathscr{A}^0 e^{-\vartheta t \underline{\phi}(\overline{\mathscr{D}}_1)}.
		\end{equation}
		Furthermore, denoting $\overline{v} = \sum_{j=1}^N m_j v_j(0)$, the limit $x_i^\infty:=\lim_{t\to \infty} (x_i(t) - \overline{v}t)$ exists for each $i$, and  
		\begin{equation}
			|(x_i(t) - \overline{v}t) - x_i^\infty| \le \frac{\mathscr{A}^0}{\vartheta \underline{\phi}(\overline{\mathscr{D}}_1)} e^{-\vartheta t \underline{\phi}(\overline{\mathscr{D}}_1)}, 
			\qquad t\ge 0.
		\end{equation}
	\end{THEOREM}
	
	\begin{REMARK}
		Our purpose in presenting this theorem is to point out that unidirectional dynamics provide a powerful mechanism for driving flocking and velocity alignment, even when the communication protocol involves degeneracies.  Our hypotheses are almost certainly not optimal, and we fully expect that the statement could be sharpened by using chain connectivity arguments in the spirit of, e.g., \cite{MPT2019} or \cite{STtopo}.
	\end{REMARK}
	\begin{REMARK}
		The proof below is written for the classical Cucker--Smale dynamics, but the conclusion applies equally well to the sticky-particle Cucker--Smale dynamics.  This is because collisions can only enhance flocking: It is easy to deduce from the sticky particle rules that $\mathscr{A}(t)\le \mathscr{A}(t-)$. 
	\end{REMARK}
	
	\begin{proof}
		Fix a time $t$ such that $\mathscr{A}(t)$ is differentiable (a.e. $t$ will do, since $\mathscr{A}$ is absolutely continuous), then choose a pair $(i,j)$ such that $\mathscr{A}(t) = v_i(t) - v_j(t)$.  Then $v_k(t)-v_i(t)$ and $v_j(t) - v_k(t)$ are nonpositive for any $k$, so
		\begin{align*}
			\dot{\mathscr{A}}(t) 
			& \le \dot{v_i}(t) - \dot{v}_j(t) \\
			& = -\sum_{k=1}^N m_k \phi(\bm{x}_i(t) - \bm{x}_k(t)) (v_i(t) - v_k(t))
			+ \sum_{k=1}^N m_k \phi(\bm{x}_j(t) - \bm{x}_k(t)) (v_j(t) - v_k(t)) \\
			& \le \sum_{k:y_k\notin B(y_i,r)\cup B(y_j,r)} \bigg[ m_k \underbrace{\phi(\bm{x}_i(t) - \bm{x}_k(t))}_{\ge \underline{\phi}(\mathscr{D}_1(t))} (v_k(t) - v_i(t))
			+ m_k \underbrace{\phi(\bm{x}_j(t) - \bm{x}_k(t))}_{\ge \underline{\phi}(\mathscr{D}_1(t))} (v_j(t) - v_k(t)) \bigg] \\
			& \le - \Bigg( \sum_{k:y_k\notin B(y_i,r)\cup B(y_j,r)} m_k  \Bigg) \underline{\phi}(\mathscr{D}_1(t))  \,(v_i(t) - v_j(t)) \\
			& \le - \vartheta \underline{\phi}(\mathscr{D}_1(t)) \mathscr{A}(t).
		\end{align*}
		From here, it is easy to see that 
		\begin{equation}
			\mathscr{L}(t) := \mathscr{A}(t) + \vartheta \int_0^{\mathscr{D}_1(t)} \underline{\phi}(x)\,\dx    
		\end{equation}
		is a Lyapunov function for the dynamics.  It follows that if we choose $\overline{\mathscr{D}}_1$ such that 
		\[
		\vartheta \int_0^{\overline{\mathscr{D}}_1} \underline{\phi}(x)\,\dx = \mathscr{L}(0),
		\]
		which is possible by our assumption \eqref{e:heavytailsub}, then \eqref{e:uniflocking} holds.
	\end{proof}

	\subsection{Flocking for the unidirectional Euler Alignment system}

	\begin{THEOREM}
		\label{t:flockingEA}
		Assume $\phi$ satisfies \eqref{e:phieven}--\eqref{e:philoccnts}.  Let $\rho^0$ be a compactly supported probability measure on $\mathbb{R}^d$, i.e., $\rho^0\in \cP_c(\mathbb{R}^d)$.     Assume $u^0\in L^\infty(\rho^0)$.  Let $(\rho, u)$ be the associated solution of \eqref{e:EAuni} generated by the entropic selection principle.  Let $\rho(\dx, \dy, t) = \rho_y(\dx,t) \nu(\dy)$ be the usual disintegration with respect to the projection $(x,y)\mapsto y$. Let $x_+(t)$ and $x_-(t)$ denote the maximum and minimum of $x$ over $\supp\rho(t)$, and let $u_+(t)$ and $u_-(t)$ denote the essential supremum and infimum of $u$ over $\supp \rho(t)$.  Define
		\[
		\mathscr{D}_1(t) =  x_+(t) - x_-(t),  
		\qquad 
		\mathscr{D}_2 = \diam\supp \nu,
		\qquad 
		\mathscr{A}(t) = u_+(t) - u_-(t).
		\]
		Define $\mathscr{D}_1^0$ and $\mathscr{A}^0$ analogously.   Assume there exist $r>0$, $\vartheta\in (0,1)$, and a nonincreasing function $\underline{\phi}\in L^1_{\loc}([0,\infty))$ such that the proportion of the total mass contained in any two balls of radius $2r$ in $\mathbb{R}^{d-1}$ is bounded above by $1-\vartheta$, 
		\begin{equation}
			\label{e:flockingreq1cont}
			\nu(B(y,2r)\cup B(z,2r)) < 1-\vartheta,
			\qquad \text{ for all } y,z\in \mathbb{R}^{d-1},
		\end{equation}
		and such that \eqref{e:heavytailsub} holds.  Then $(\rho, u)$ experiences flocking and exponentially fast alignment:  There exists $\overline{\mathscr{D}}_1>0$ such that 
		\begin{equation}
			\label{e:uniflockingcont}
			\sup_{t\ge 0} \mathscr{D}_1(t)\le \overline{\mathscr{D}}_1,
			\qquad 
			\mathscr{A}(t)
			\le \mathscr{A}^0 e^{-\vartheta t \underline{\phi}(\overline{\mathscr{D}}_1)}.
		\end{equation}
		Finally, denoting the average velocity of the system by $\overline{u} = \int u^0 \rho^0$, there exist $\rho_\infty\in \cP_c(\mathbb{R}^d)$ and $C,c>0$ such that 
		\begin{equation}
			\label{e:strongflocking}
			W_1(\rho(\cdot + \overline{u}t, \cdot,t), \rho_\infty) \le C e^{-c t}.
		\end{equation}
	\end{THEOREM}
	
	This theorem can be proved by approximating the solution $(\rho, u)$ by SPCS dynamics and then using Theorem~\ref{t:flockingCS} to control the latter.  We omit the details,  which follow the same limiting argument as in \cite{LeslieTan2023}.

%	\bibliographystyle{plain}
%	\bibliography{unirefs}

\begin{thebibliography}{100}
	
	\bibitem{AmadoriChristoforou2026}
	Debora Amadori, Felisia~Angela Chiarello, and Cleopatra Christoforou.
	\newblock Euler-flocking system with nonlocal dissipation in 1{D}: periodic
	entropy solutions.
	\newblock {\em J. Hyperbolic Differ. Equ.}, 23(1):61--86, 2026.
	
	\bibitem{AmadoriChristoforou2022}
	Debora Amadori and Cleopatra Christoforou.
	\newblock {$BV$} solutions for a hydrodynamic model of flocking-type with
	all-to-all interaction kernel.
	\newblock {\em Math. Models Methods Appl. Sci.}, 32(11):2295--2357, 2022.
	
	\bibitem{AmadoriChristoforou2024}
	Debora Amadori and Cleopatra Christoforou.
	\newblock Unconditional flocking for weak solutions to self-organized systems
	of {E}uler-type with all-to-all interaction kernel.
	\newblock {\em Nonlinear Anal.}, 245:Paper No. 113576, 22, 2024.
	
	\bibitem{AmbrosioDalMaso1990}
	Luigi Ambrosio and Gianni Dal~Maso.
	\newblock A general chain rule for distributional derivatives.
	\newblock {\em Proc. Amer. Math. Soc.}, 108(3):691--702, 1990.
	
	\bibitem{ArnaizCastro2019}
	Victor Arnaiz and \'{A}ngel Castro.
	\newblock Singularity formation for the fractional {E}uler-alignment system in
	1{D}.
	\newblock {\em Trans. Amer. Math. Soc.}, 374(1):487--514, 2021.
	
	\bibitem{BianchiniDaneri2023}
	Stefano Bianchini and Sara Daneri.
	\newblock On the sticky particle solutions to the multi-dimensional
	pressureless {E}uler equations.
	\newblock {\em J. Differential Equations}, 368:173--202, 2023.
	
	\bibitem{Bouchut1994}
	Fran\c{c}ois Bouchut.
	\newblock On zero pressure gas dynamics.
	\newblock In {\em Advances in kinetic theory and computing}, volume~22 of {\em
		Ser. Adv. Math. Appl. Sci.}, pages 171--190. World Sci. Publ., River Edge,
	NJ, 1994.
	
	\bibitem{BouchutJames1998}
	Fran\c{c}ois Bouchut and Fran\c{c}ois James.
	\newblock One-dimensional transport equations with discontinuous coefficients.
	\newblock {\em Nonlinear Anal.}, 32(7):891--933, 1998.
	
	\bibitem{BouchutJames1999}
	Fran\c{c}ois Bouchut and Fran\c{c}ois James.
	\newblock Duality solutions for pressureless gases, monotone scalar
	conservation laws, and uniqueness.
	\newblock {\em Comm. Partial Differential Equations}, 24(11-12):2173--2189,
	1999.
	
	\bibitem{BouchutPerthame1998}
	Fran\c{c}ois Bouchut and Benoit Perthame.
	\newblock Kru\v{z}kov's estimates for scalar conservation laws revisited.
	\newblock {\em Trans. Amer. Math. Soc.}, 350(7):2847--2870, 1998.
	
	\bibitem{BrenierGangboSavareWestdickenberg2013}
	Yann Brenier, Wilfrid Gangbo, Giuseppe Savar\'{e}, and Michael Westdickenberg.
	\newblock Sticky particle dynamics with interactions.
	\newblock {\em J. Math. Pures Appl.}, 99(5):577--617, 2013.
	
	\bibitem{brenier1998sticky}
	Yann Brenier and Emmanuel Grenier.
	\newblock Sticky particles and scalar conservation laws.
	\newblock {\em SIAM J. Numer. Anal}, 35(6):2317--2328, 1998.
	
	\bibitem{BressanChenHuang2025a}
	Alberto Bressan, Geng Chen, and Shoujun Huang.
	\newblock Generic singularities for 2{D} pressureless flows.
	\newblock {\em Sci. China Math.}, 68(3):559--576, 2025.
	
	\bibitem{BressanChenHuang2025b}
	Alberto Bressan, Geng Chen, and Shoujun Huang.
	\newblock The initial stages of a generic singularity for a 2{D} pressureless
	gas.
	\newblock {\em Comm. Math. Phys.}, 406(12):Paper No. 288, 44, 2025.
	
	\bibitem{BressanNguyen2014}
	Alberto Bressan and Truyen Nguyen.
	\newblock Non-existence and non-uniqueness for multidimensional sticky particle
	systems.
	\newblock {\em Kinet. Relat. Models}, 7(2):205--218, 2014.
	
	\bibitem{CanizoCarrilloRosado2011}
	J.~A. Ca\~nizo, J.~A. Carrillo, and J.~Rosado.
	\newblock A well-posedness theory in measures for some kinetic models of
	collective motion.
	\newblock {\em Math. Models Methods Appl. Sci.}, 21(3):515--539, 2011.
	
	\bibitem{CarrilloChenYuanZatorska2025}
	Jos\'e{}~A. Carrillo, Gui-Qiang~G. Chen, Difan Yuan, and Ewelina Zatorska.
	\newblock Global solutions of the one-dimensional compressible {E}uler
	equations with nonlocal interactions via the inviscid limit.
	\newblock {\em Arch. Ration. Mech. Anal.}, 249(3):Paper No. 38, 62, 2025.
	
	\bibitem{CarrilloChoiHauray2014}
	Jos\'e{}~A. Carrillo, Young-Pil Choi, and Maxime Hauray.
	\newblock Local well-posedness of the generalized {C}ucker-{S}male model with
	singular kernels.
	\newblock In {\em M{MCS}, {M}athematical modelling of complex systems},
	volume~47 of {\em ESAIM Proc. Surveys}, pages 17--35. EDP Sci., Les Ulis,
	2014.
	
	\bibitem{CCMP2017}
	Jos\'{e}~A. Carrillo, Young-Pil Choi, Piotr~B. Mucha, and Jan Peszek.
	\newblock Sharp conditions to avoid collisions in singular {C}ucker-{S}male
	interactions.
	\newblock {\em Nonlinear Anal. Real World Appl.}, 37:317--328, 2017.
	
	\bibitem{CCP2017}
	Jos\'{e}~A. Carrillo, Young-Pil Choi, and Sergio~P. Perez.
	\newblock A review on attractive-repulsive hydrodynamics for consensus in
	collective behavior.
	\newblock In {\em Active particles. {V}ol. 1. {A}dvances in theory, models, and
		applications}, Model. Simul. Sci. Eng. Technol., pages 259--298.
	Birkh\"auser/Springer, Cham, 2017.
	
	\bibitem{CCTT2016}
	Jos\'{e}~A. Carrillo, Young-Pil Choi, Eitan Tadmor, and Changhui Tan.
	\newblock Critical thresholds in 1{D} {E}uler equations with non-local forces.
	\newblock {\em Math. Models Methods Appl. Sci.}, 26(1):185--206, 2016.
	
	\bibitem{Carrilloetal2017weakEA}
	Jos\'{e}~A. Carrillo, Eduard Feireisl, Piotr Gwiazda, and Agnieszka
	\'{S}wierczewska Gwiazda.
	\newblock Weak solutions for {E}uler systems with non-local interactions.
	\newblock {\em J. Lond. Math. Soc. (2)}, 95(3):705--724, 2017.
	
	\bibitem{carrillogaltung2025}
	Jos{\'e}~A Carrillo and Sondre Galtung.
	\newblock Equivalence of entropy solutions and gradient flows for pressureless
	{1D} {E}uler systems.
	\newblock {\em arXiv preprint arXiv:2312.04932}, 2023.
	
	\bibitem{carrillo2010particle}
	Jos\'{e} A.~A Carrillo, Massimo Fornasier, Giuseppe Toscani, and Francesco
	Vecil.
	\newblock Particle, kinetic, and hydrodynamic models of swarming.
	\newblock In {\em Mathematical modeling of collective behavior in
		socio-economic and life sciences}, pages 297--336. Springer, 2010.
	
	\bibitem{CarrilloChoiTadmor2026}
	José~A. Carrillo, Young-Pil Choi, and Eitan Tadmor.
	\newblock Lagrangian formulation and eulerian closure in alignment dynamics.
	\newblock {\em arXiv preprint arXiv:2604.10253}, 2026.
	
	\bibitem{CavallettiSedjroWestdickenberg2015}
	Fabio Cavalletti, Marc Sedjro, and Michael Westdickenberg.
	\newblock A simple proof of global existence for the 1{D} pressureless gas
	dynamics equations.
	\newblock {\em SIAM J. Math. Anal.}, 47(1):66--79, 2015.
	
	\bibitem{CavallettiSedjroWestdickenberg2019}
	Fabio Cavalletti, Marc Sedjro, and Michael Westdickenberg.
	\newblock A variational time discretization for compressible {E}uler equations.
	\newblock {\em Trans. Amer. Math. Soc.}, 371(7):5083--5155, 2019.
	
	\bibitem{chen2020global}
	Li~Chen, Changhui Tan, and Lining Tong.
	\newblock On the global classical solution to compressible {E}uler system with
	singular velocity alignment.
	\newblock {\em Methods Appl. Anal.}, 28(2):153--172, 2021.
	
	\bibitem{ChoiEAPressure2019}
	Young-Pil Choi.
	\newblock The global {C}auchy problem for compressible {E}uler equations with a
	nonlocal dissipation.
	\newblock {\em Math. Models Methods Appl. Sci.}, 29(1):185--207, 2019.
	
	\bibitem{ChoiHwang2024}
	Young-Pil Choi and Byung-Hoon Hwang.
	\newblock From {BGK}-alignment model to the pressured {E}uler-alignment system
	with singular communication weights.
	\newblock {\em J. Differential Equations}, 379:363--412, 2024.
	
	\bibitem{ChoiJung2023kin}
	Young-Pil Choi and Jinwook Jung.
	\newblock Local well-posedness for the kinetic {C}ucker-{S}male model with
	super-{C}oulombic communication weights.
	\newblock {\em J. Differential Equations}, 366:807--832, 2023.
	
	\bibitem{ChoiJung2024kin}
	Young-Pil Choi and Jinwook Jung.
	\newblock On weak solutions to the kinetic {C}ucker-{S}male model with singular
	communication weights.
	\newblock {\em Proc. Amer. Math. Soc.}, 152(8):3423--3436, 2024.
	
	\bibitem{ChoiJung2025kin}
	Young-Pil Choi and Jinwook Jung.
	\newblock H\"older continuous solutions to the kinetic {C}ucker-{S}male model
	with super-{C}oulombic singular weights.
	\newblock {\em Kinet. Relat. Models}, 18(5):787--799, 2025.
	
	\bibitem{ChoiKim2023}
	Young-Pil Choi and Jeongho Kim.
	\newblock Rigorous derivation of the {E}uler-alignment model with singular
	communication weights from a kinetic {F}okker-{P}lanck-alignment model.
	\newblock {\em Math. Models Methods Appl. Sci.}, 33(1):31--65, 2023.
	
	\bibitem{ChoiTadmor2026}
	Young-Pil Choi and Eitan Tadmor.
	\newblock Finite-time breakdown of the euler-alignment system for supercritical
	initial data.
	\newblock {\em arXiv preprint 2607.08060}, 2026.
	
	\bibitem{ConstantinDrivasShvydkoy2020}
	Peter Constantin, Theodore~D. Drivas, and Roman Shvydkoy.
	\newblock Entropy hierarchies for equations of compressible fluids and
	self-organized dynamics.
	\newblock {\em SIAM J. Math. Anal.}, 52(3):3073--3092, 2020.
	
	\bibitem{CS2007a}
	Felipe Cucker and Steve Smale.
	\newblock Emergent behavior in flocks.
	\newblock {\em IEEE Trans. Automat. Control}, 52(5):852--862, 2007.
	
	\bibitem{CS2007b}
	Felipe Cucker and Steve Smale.
	\newblock On the mathematics of emergence.
	\newblock {\em Jpn. J. Math.}, 2(1):197--227, 2007.
	
	\bibitem{Dafermos1972}
	Constantine~M. Dafermos.
	\newblock Polygonal approximations of solutions of the initial value problem
	for a conservation law.
	\newblock {\em J. Math. Anal. Appl.}, 38:33--41, 1972.
	
	\bibitem{Dafermosbook}
	Constantine~M. Dafermos.
	\newblock {\em Hyperbolic conservation laws in continuum physics}, volume 325
	of {\em Grundlehren der Mathematischen Wissenschaften [Fundamental Principles
		of Mathematical Sciences]}.
	\newblock Springer-Verlag, Berlin, fourth edition, 2016.
	
	\bibitem{DanchinMuchaPeszekWroblewski2018}
	Rapha\"{e}l Danchin, Piotr~B. Mucha, Jan Peszek, and Bartosz Wr\'{o}blewski.
	\newblock Regular solutions to the fractional {E}uler alignment system in the
	{B}esov spaces framework.
	\newblock {\em Math. Models Methods Appl. Sci.}, 29(1):89--119, 2019.
	
	\bibitem{Dermoune1999}
	Azzouz Dermoune.
	\newblock Probabilistic interpretation of sticky particle model.
	\newblock {\em Ann. Probab.}, 27(3):1357--1367, 1999.
	
	\bibitem{DermouneMoutsinga2003}
	Azzouz Dermoune and Octave Moutsinga.
	\newblock Generalized variational principles.
	\newblock In {\em S\'{e}minaire de {P}robabilit\'{e}s, {XXXVI}}, volume 1801 of
	{\em Lecture Notes in Math.}, pages 183--193. Springer, Berlin, 2003.
	
	\bibitem{DietertShvydkoy2019}
	Helge Dietert and Roman Shvydkoy.
	\newblock On {C}ucker-{S}male dynamical systems with degenerate communication.
	\newblock {\em Anal. Appl. (Singap.)}, 19(4):551--573, 2021.
	
	\bibitem{DKRT}
	Tam Do, Alexander Kiselev, Lenya Ryzhik, and Changhui Tan.
	\newblock Global regularity for the fractional {E}uler {A}lignment system.
	\newblock {\em Arch. Ration. Mech. Anal.}, 228(1):1--37, 2018.
	
	\bibitem{ERykovSinai1996}
	Weinan E, Yu~G. Rykov, and Ya~G. Sinai.
	\newblock Generalized variational principles, global weak solutions and
	behavior with random initial data for systems of conservation laws arising in
	adhesion particle dynamics.
	\newblock {\em Comm. Math. Phys.}, 177(2):349--380, 1996.
	
	\bibitem{FabisiakPeszek2024}
	Michal Fabisiak and Jan Peszek.
	\newblock Inevitable monokineticity of strongly singular alignment.
	\newblock {\em Math. Ann.}, 390(1):589--637, 2024.
	
	\bibitem{FigalliKang2019}
	Alessio Figalli and Moon-Jin Kang.
	\newblock A rigorous derivation from the kinetic {C}ucker-{S}male model to the
	pressureless {E}uler system with nonlocal alignment.
	\newblock {\em Anal. PDE}, 12(3):843--866, 2019.
	
	\bibitem{Galtung2025}
	Sondre~Tesdal Galtung.
	\newblock The sticky particle dynamics of the 1{D} pressureless
	{E}uler-alignment system as a gradient flow.
	\newblock {\em Appl. Math. Optim.}, 91(2):Paper No. 27, 49, 2025.
	
	\bibitem{Grenier1995}
	Emmanuel Grenier.
	\newblock Existence globale pour le syst\`eme des gaz sans pression.
	\newblock {\em C. R. Acad. Sci. Paris S\'{e}r. I Math.}, 321(2):171--174, 1995.
	
	\bibitem{HaHuangWang2014wk}
	Seung-Yeal Ha, Feimin Huang, and Yi~Wang.
	\newblock A global unique solvability of entropic weak solution to the
	one-dimensional pressureless {E}uler system with a flocking dissipation.
	\newblock {\em J. Differential Equations}, 257(5):1333--1371, 2014.
	
	\bibitem{HaKimParkZhang2019}
	Seung-Yeal Ha, Jeongho Kim, Jinyeong Park, and Xiongtao Zhang.
	\newblock Complete cluster predictability of the {C}ucker-{S}male flocking
	model on the real line.
	\newblock {\em Arch. Ration. Mech. Anal.}, 231(1):319--365, 2019.
	
	\bibitem{HaLiu2008}
	Seung-Yeal Ha and Jian-Guo Liu.
	\newblock A simple proof of the {C}ucker-{S}male flocking dynamics and
	mean-field limit.
	\newblock {\em Commun. Math. Sci.}, 7(2):297--325, 2009.
	
	\bibitem{HaParkZhang2018}
	Seung-Yeal Ha, Jinyeong Park, and Xiongtao Zhang.
	\newblock A first-order reduction of the {C}ucker-{S}male model on the real
	line and its clustering dynamics.
	\newblock {\em Commun. Math. Sci.}, 16(7):1907--1931, 2018.
	
	\bibitem{HT2008}
	Seung-Yeal Ha and Eitan Tadmor.
	\newblock From particle to kinetic and hydrodynamic descriptions of flocking.
	\newblock {\em Kinet. Relat. Models}, 1(3):415--435, 2008.
	
	\bibitem{HT2016}
	Siming He and Eitan Tadmor.
	\newblock Global regularity of two-dimensional flocking hydrodynamics.
	\newblock {\em C. R. Math. Acad. Sci. Paris}, 355(7):795--805, 2017.
	
	\bibitem{HoldenRisebrobook}
	Helge Holden and Nils~Henrik Risebro.
	\newblock {\em Front tracking for hyperbolic conservation laws}, volume 152 of
	{\em Applied Mathematical Sciences}.
	\newblock Springer, Heidelberg, second edition, 2015.
	
	\bibitem{HuangShiWang2025}
	Feimin Huang, Jiajin Shi, and Yi~Wang.
	\newblock Non-uniqueness of admissible weak solutions to the two-dimensional
	pressureless {E}uler system.
	\newblock {\em J. Differential Equations}, 418:238--257, 2025.
	
	\bibitem{HuangWang2001stickyunique}
	Feimin Huang and Zhen Wang.
	\newblock Well posedness for pressureless flow.
	\newblock {\em Comm. Math. Phys.}, 222(1):117--146, 2001.
	
	\bibitem{Hynd2019LagrangianSP}
	Ryan Hynd.
	\newblock Lagrangian coordinates for the sticky particle system.
	\newblock {\em SIAM J. Math. Anal.}, 51(5):3769--3795, 2019.
	
	\bibitem{Hynd2020probmeasureSP}
	Ryan Hynd.
	\newblock Probability measures on the path space and the sticky particle
	system.
	\newblock {\em Annali della Scuola Normale Superiore di Pisa, Classe di
		Scienze}, XXI(5):1333--1357, 2020.
	
	\bibitem{Hynd2020trajectory}
	Ryan Hynd.
	\newblock A trajectory map for the pressureless {E}uler equations.
	\newblock {\em Trans. Amer. Math. Soc.}, 373(10):6777--6815, 2020.
	
	\bibitem{Jin2015flockingdiss}
	Chunyin Jin.
	\newblock Well posedness for pressureless {E}uler system with a flocking
	dissipation in {W}asserstein space.
	\newblock {\em Nonlinear Anal.}, 128:412--422, 2015.
	
	\bibitem{Jin2016flockingdiss}
	Chunyin Jin.
	\newblock Existence and uniqueness of entropy solution to pressureless {E}uler
	system with a flocking dissipation.
	\newblock {\em Acta Math. Sci. Ser. B (Engl. Ed.)}, 36(5):1262--1284, 2016.
	
	\bibitem{KangVasseur2015}
	Moon-Jin Kang and Alexis~F. Vasseur.
	\newblock Asymptotic analysis of {V}lasov-type equations under strong local
	alignment regime.
	\newblock {\em Math. Models Methods Appl. Sci.}, 25(11):2153--2173, 2015.
	
	\bibitem{KMT2013}
	Trygve~K. Karper, Antoine Mellet, and Konstantina Trivisa.
	\newblock Existence of weak solutions to kinetic flocking models.
	\newblock {\em SIAM J. Math. Anal.}, 45(1):215--243, 2013.
	
	\bibitem{KMT2014}
	Trygve~K. Karper, Antoine Mellet, and Konstantina Trivisa.
	\newblock On strong local alignment in the kinetic {C}ucker-{S}male model.
	\newblock In {\em Hyperbolic conservation laws and related analysis with
		applications}, volume~49 of {\em Springer Proc. Math. Stat.}, pages 227--242.
	Springer, Heidelberg, 2014.
	
	\bibitem{KMT2015}
	Trygve~K. Karper, Antoine Mellet, and Konstantina Trivisa.
	\newblock Hydrodynamic limit of the kinetic {C}ucker-{S}male flocking model.
	\newblock {\em Math. Models Methods Appl. Sci.}, 25(1):131--163, 2015.
	
	\bibitem{KT}
	Alexander Kiselev and Changhui Tan.
	\newblock Global regularity for 1{D} {E}ulerian dynamics with singular
	interaction forces.
	\newblock {\em SIAM J. Math. Anal.}, 50(6):6208--6229, 2018.
	
	\bibitem{Kruzkov1970}
	S.~N. Kru\v{z}kov.
	\newblock First order quasilinear equations with several independent variables.
	\newblock {\em Mat. Sb. (N.S.)}, 81(123):228--255, 1970.
	
	\bibitem{Lear2021unising}
	Daniel Lear.
	\newblock Global existence and limiting behavior of unidirectional flocks for
	the fractional {E}uler alignment system.
	\newblock {\em SIAM J. Math. Anal.}, 55(4):3731--3754, 2023.
	
	\bibitem{LLST2020geometric}
	Daniel Lear, Trevor~M. Leslie, Roman Shvydkoy, and Eitan Tadmor.
	\newblock Geometric structure of mass concentration sets for pressureless
	{E}uler alignment systems.
	\newblock {\em Adv. Math.}, 401:Paper No. 108290, 30, 2022.
	
	\bibitem{LearReynoldsShvydkoy2022}
	Daniel Lear, David~N. Reynolds, and Roman Shvydkoy.
	\newblock Global solutions to multi-dimensional topological {E}uler alignment
	systems.
	\newblock {\em Ann. PDE}, 8(1):Paper No. 1, 43, 2022.
	
	\bibitem{LearShvydkoy2021}
	Daniel Lear and Roman Shvydkoy.
	\newblock Unidirectional flocks in hydrodynamic {E}uler alignment system {II}:
	{S}ingular models.
	\newblock {\em Commun. Math. Sci.}, 19(3):807--828, 2021.
	
	\bibitem{LearShvydkoy2019}
	Daniel Lear and Roman Shvydkoy.
	\newblock Existence and stability of unidirectional flocks in hydrodynamic
	{E}uler alignment systems.
	\newblock {\em Anal. PDE}, 15(1):175--196, 2022.
	
	\bibitem{L2019CTC}
	Trevor~M. Leslie.
	\newblock On the {L}agrangian trajectories for the one-dimensional {E}uler
	alignment model without vacuum velocity.
	\newblock {\em Comptes Rendus. Math\'ematique}, 358(4):421--433, 2020.
	
	\bibitem{LesliePeszek2026}
	Trevor~M. Leslie and Jan Peszek.
	\newblock Topological and purely topological alignment dynamics.
	\newblock {\em arXiv 2601.11828}, 2026.
	
	\bibitem{LS2019}
	Trevor~M. Leslie and Roman Shvydkoy.
	\newblock On the structure of limiting flocks in hydrodynamic {E}uler
	{A}lignment models.
	\newblock {\em Math. Models Methods Appl. Sci.}, 29(13):2419--2431, 2019.
	
	\bibitem{LeslieTan2023}
	Trevor~M. Leslie and Changhui Tan.
	\newblock Sticky particle cucker–smale dynamics and the entropic selection
	principle for the 1d euler-alignment system.
	\newblock {\em Communications in Partial Differential Equations},
	48(5):753--791, 2023.
	
	\bibitem{LeslieTan2024}
	Trevor~M. Leslie and Changhui Tan.
	\newblock Finite- and infinite-time cluster formation for alignment dynamics on
	the real line.
	\newblock {\em J. Evol. Equ.}, 24(1):Paper No. 8, 45, 2024.
	
	\bibitem{LiMiaoTanXue2024}
	Yatao Li, Qianyun Miao, Changhui Tan, and Liutang Xue.
	\newblock Global well-posedness and refined regularity criterion for the
	uni-directional {E}uler-alignment system.
	\newblock {\em Int. Math. Res. Not. IMRN}, (23):14393--14422, 2024.
	
	\bibitem{LiuPego2025}
	Jian-Guo Liu and Robert~L. Pego.
	\newblock Analysis of the adhesion model and the reconstruction problem in
	cosmology.
	\newblock {\em arXiv 2511.09800}, 2025.
	
	\bibitem{Lucier1986}
	Bradley~J. Lucier.
	\newblock A moving mesh numerical method for hyperbolic conservation laws.
	\newblock {\em Math. Comp.}, 46(173):59--69, 1986.
	
	\bibitem{MPT2019}
	Javier Morales, Jan Peszek, and Eitan Tadmor.
	\newblock Flocking {W}ith {S}hort-{R}ange {I}nteractions.
	\newblock {\em J. Stat. Phys.}, 176(2):382--397, 2019.
	
	\bibitem{MT2014}
	Sebastien Motsch and Eitan Tadmor.
	\newblock Heterophilious dynamics enhances consensus.
	\newblock {\em SIAM Rev.}, 56(4):577--621, 2014.
	
	\bibitem{Moutsinga2008}
	Octave Moutsinga.
	\newblock Convex hulls, sticky particle dynamics and pressure-less gas system.
	\newblock {\em Ann. Math. Blaise Pascal}, 15(1):57--80, 2008.
	
	\bibitem{MuchaPeszek2018}
	Piotr~B. Mucha and Jan Peszek.
	\newblock The {C}ucker-{S}male equation: singular communication weight,
	measure-valued solutions and weak-atomic uniqueness.
	\newblock {\em Arch. Ration. Mech. Anal.}, 227(1):273--308, 2018.
	
	\bibitem{NatileSavare2009}
	Luca Natile and Giuseppe Savar\'{e}.
	\newblock A {W}asserstein approach to the one-dimensional sticky particle
	system.
	\newblock {\em SIAM J. Math. Anal.}, 41(4):1340--1365, 2009.
	
	\bibitem{nguyen2008pressureless}
	Truyen Nguyen and Adrian Tudorascu.
	\newblock Pressureless {E}uler/{E}uler--{P}oisson systems via adhesion dynamics
	and scalar conservation laws.
	\newblock {\em SIAM J. Math. Anal.}, 40(2):754--775, 2008.
	
	\bibitem{NguyenTudorascu2015}
	Truyen Nguyen and Adrian Tudorascu.
	\newblock One-dimensional pressureless gas systems with/without viscosity.
	\newblock {\em Comm. Partial Differential Equations}, 40(9):1619--1665, 2015.
	
	\bibitem{Pes2014}
	Jan Peszek.
	\newblock Existence of piecewise weak solutions of a discrete
	{C}ucker-{S}male's flocking model with a singular communication weight.
	\newblock {\em J. Differential Equations}, 257(8):2900--2925, 2014.
	
	\bibitem{Pes2015}
	Jan Peszek.
	\newblock Discrete {C}ucker-{S}male flocking model with a weakly singular
	weight.
	\newblock {\em SIAM J. Math. Anal.}, 47(5):3671--3686, 2015.
	
	\bibitem{PeszekPoyatoFiberedOT}
	Jan Peszek and David Poyato.
	\newblock Heterogeneous gradient flows in the topology of fibered optimal
	transport.
	\newblock {\em Calc. Var. Partial Differential Equations}, 62(9):Paper No. 258,
	72, 2023.
	
	\bibitem{ReynoldsShvydkoy2020}
	David~N. Reynolds and Roman Shvydkoy.
	\newblock Local well-posedness of the topological {E}uler alignment models of
	collective behavior.
	\newblock {\em Nonlinearity}, 33(10):5176--5214, 2020.
	
	\bibitem{Shvydkoy2018NearlyAligned}
	Roman Shvydkoy.
	\newblock Global existence and stability of nearly aligned flocks.
	\newblock {\em Journal of Dynamics and Differential Equations}, Aug 2018.
	
	\bibitem{ShvydkoyBook}
	Roman Shvydkoy.
	\newblock {\em {D}ynamics and {A}nalysis of {A}lignment {M}odels of
		{C}ollective {B}ehavior}, volume~4 of {\em Ne\v{c}as Center Series}.
	\newblock Birkh\"{a}user Basel, 2021.
	
	\bibitem{ShvydkoySurvey}
	Roman Shvydkoy.
	\newblock Environmental averaging.
	\newblock {\em EMS Surv. Math. Sci.}, 11(2):277--413, 2024.
	
	\bibitem{ShvydkoyTadmorI}
	Roman Shvydkoy and Eitan Tadmor.
	\newblock Eulerian dynamics with a commutator forcing.
	\newblock {\em Transactions of Mathematics and Its Applications}, 1(1), 2017.
	
	\bibitem{ShvydkoyTadmorII}
	Roman {Shvydkoy} and Eitan {Tadmor}.
	\newblock Eulerian dynamics with a commutator forcing {II}: {F}locking.
	\newblock {\em Discrete Contin. Dyn. Syst.}, 37(11):5503--5520, 2017.
	
	\bibitem{ShvydkoyTadmorIII}
	Roman Shvydkoy and Eitan Tadmor.
	\newblock Eulerian dynamics with a commutator forcing {III}. {F}ractional
	diffusion of order {$0<\alpha<1$}.
	\newblock {\em Phys. D}, 376/377:131--137, 2018.
	
	\bibitem{STtopo}
	Roman Shvydkoy and Eitan Tadmor.
	\newblock Topologically based fractional diffusion and emergent dynamics with
	short-range interactions.
	\newblock {\em SIAM Journal on Mathematical Analysis}, 52(6):5792--5839, 2020.
	
	\bibitem{Tadmor2021review}
	Eitan Tadmor.
	\newblock On the mathematics of swarming: Emergent behavior in alignment
	dynamics.
	\newblock {\em Notices of the AMS}, 68(4):493--503, 2021.
	
	\bibitem{Tadmor2023}
	Eitan Tadmor.
	\newblock Swarming: hydrodynamic alignment with pressure.
	\newblock {\em Bull. Amer. Math. Soc. (N.S.)}, 60(3):285--325, 2023.
	
	\bibitem{Tadmor2026}
	Eitan Tadmor.
	\newblock Decrease of entropy and emergence of order in collective dynamics.
	\newblock {\em Commun. Contemp. Math.}, 28(5):Paper No. 2540006, 2026.
	
	\bibitem{TT2014}
	Eitan Tadmor and Changhui Tan.
	\newblock Critical thresholds in flocking hydrodynamics with non-local
	alignment.
	\newblock {\em Philos. Trans. R. Soc. Lond. Ser. A Math. Phys. Eng. Sci.},
	372:20130401, 2014.
	
	\bibitem{tan2020euler}
	Changhui Tan.
	\newblock On the {E}uler-alignment system with weakly singular communication
	weights.
	\newblock {\em Nonlinearity}, 33(4):1907--1924, 2020.
	
	\bibitem{tan2021eulerian}
	Changhui Tan.
	\newblock {E}ulerian dynamics in multi-dimensions with radial symmetry.
	\newblock {\em SIAM J. Math. Anal.}, 53(3):3040--3071, 2021.
	
	\bibitem{WangDing1997}
	Zhen Wang and Xiaqi Ding.
	\newblock Uniqueness of generalized solution for the {C}auchy problem of
	transportation equations.
	\newblock {\em Acta Math. Sci. (English Ed.)}, 17(3):341--352, 1997.
	
	\bibitem{WangHuangDing1997}
	Zhen Wang, Feimin Huang, and Xiaqi Ding.
	\newblock On the {C}auchy problem of transportation equations.
	\newblock {\em Acta Math. Appl. Sinica (English Ser.)}, 13(2):113--122, 1997.
	
	\bibitem{Zeldovich1969sb}
	Ya.~B. Zeldovich.
	\newblock Gravitational instability: An approximate theory for large density
	perturbations.
	\newblock {\em Astron. Astrophys.}, 5:84--89, 1970.
	
\end{thebibliography}
	
\def\cprime{$'$}

\end{document}